\documentclass[
	a4paper,
	11pt
]{amsart}

\usepackage[british]{babel}
\usepackage{enumitem}

\usepackage{times}
\usepackage[T1]{fontenc}

\usepackage[leqno]{amsmath}
\usepackage{amssymb}
\usepackage{amsthm}
\usepackage{amsaddr}
\usepackage{mathtools}
\usepackage{stmaryrd}
\usepackage[mathcal]{euscript}
\usepackage{tikz-cd}
\usetikzlibrary {shapes.arrows, arrows.meta, calc, intersections}
\tikzset{bg/.style={fill=gray, fill opacity=0.05}}
\definecolor{figcolor}{HTML}{3E4953}
\definecolor{accentcolor}{HTML}{53403E}
\def\rectcol{figcolor}  
\tikzset{bg/.style={fill=gray, fill opacity=0.05}}
\tikzset{textbg/.style={fill=white, fill opacity=.8, text opacity=1, inner sep=0pt}}

\usepackage[
backend=biber,
style=alphabetic,
sorting=nyt,
date=year,
giveninits=true,
useprefix=true,
doi=false,isbn=false,url=false,
maxnames=99
]{biblatex}
\bibliography{bibliography.bib}

\definecolor{LinkColor}{HTML}{697B8C}
\definecolor{CiteColor}{HTML}{8C6B69}
\definecolor{URLColor}{HTML}{898C69}
\usepackage[
	colorlinks=true,
	linkcolor=LinkColor,
	citecolor=CiteColor,
	urlcolor=URLColor
]{hyperref}
\usepackage{aliascnt}
\usepackage[capitalise,nameinlink,noabbrev]{cleveref}

\theoremstyle{plain}
\newtheorem{theorem}{Theorem}[section]
\newtheorem*{theorem*}{Theorem}
\crefname{theorem}{Theorem}{Theorems}

\newcommand{\newaliastheorem}[3]{%
	\newaliascnt{#1}{theorem}%
	\newtheorem{#1}[#1]{#2}%
	\aliascntresetthe{#1}%
	\crefname{#1}{#2}{#3}%
	\Crefname{#1}{#2}{#3}%
}

\newaliastheorem{proposition}{Proposition}{Propositions}
\newaliastheorem{corollary}{Corollary}{Corollaries}
\newaliastheorem{lemma}{Lemma}{Lemmas}
\newaliastheorem{conjecture}{Conjecture}{Conjectures}

\theoremstyle{definition}
\newaliastheorem{definition}{Definition}{Definitions}
\newaliastheorem{example}{Example}{Examples}
\newaliastheorem{counterexample}{Counterexample}{Counterexamples}
\newaliastheorem{construction}{Construction}{Constructions}
\newaliastheorem{intuition}{Intuition}{Intuitions}

\theoremstyle{remark}
\newaliastheorem{remark}{Remark}{Remarks}

\renewcommand{\leq}{\leqslant}
\renewcommand{\geq}{\geqslant}
\renewcommand{\epsilon}{\varepsilon}
\newcommand{\Tau}{\mathrm{T}}

\newcommand{\Opens}{\mathop{\mathcal{O}\mspace{-4mu}}}
\newcommand{\Powerset}{\mathop{\mathcal{P}\mspace{-1mu}}}

\newcommand{\Sl}{\mathop{\mathrm{S}\ell}\mspace{-1mu}}

\DeclareMathOperator{\id}{id}
\newcommand{\pr}{\mathrm{pr}}
\newcommand{\sqleq}{\sqsubseteq} 

\newcommand{\op}{{\mathrm{op}}}
\newcommand{\tensor}{\otimes}

\newcommand{\terminal}{{!}}
\newcommand{\caus}{\preccurlyeq}
\DeclareMathOperator{\pos}{\exists}

\newcommand{\tworightarrow}{\begin{tikzcd}[ampersand replacement=\&,cramped, column sep=1em]  \phantom{}\ar[r,two heads] \& \phantom{} \end{tikzcd}}
\let\rightarrowtail\relax\newcommand{\rightarrowtail}{\begin{tikzcd}[ampersand replacement=\&,cramped, column sep=1.2em]  \phantom{}\ar[r,tail] \& \phantom{} \end{tikzcd}}

\newcommand{\Leq}{\mathrel{\trianglelefteqslant}}
\newcommand{\up}{\mathop{\uparrow}\!}
\newcommand{\down}{\mathop{\downarrow}\!}

\newcommand{\preUp}{\scalebox{1.2}[1.05]{\rotatebox[origin=c]{90}{$\rightarrowtriangle$}}}
\newcommand{\preDown}{\scalebox{1.2}[1.05]{\rotatebox[origin=c]{-90}{$\rightarrowtriangle$}}}
\newcommand{\Up}{\text{\preUp}}
\newcommand{\Down}{\text{\preDown}}
\newcommand{\UpDown}{{\scriptscriptstyle\Up\mspace{-3mu}\Down}}
\newcommand{\Upsub}[1]{\Up_{\mspace{-4mu} #1}}
\newcommand{\Downsub}[1]{\Down_{\mspace{-2mu} #1}}

\newcommand{\arrowRscaled}[3]{%
	\tikz[baseline={#1*#3}, scale=#3]{%
		\begin{scope}[rotate=#2]
			\def\base{.13em}
			\def\tip{.47em}
			\def\half{.2em}
			
			\draw[
			line width=.5pt,
			line cap=round,
			line join=round
			] (-.47em,0) -- (\base,0);
			
			\path[fill, rounded corners=.4pt]
			(\base,\half) -- (\tip,0) -- (\base,-\half) -- cycle;
		\end{scope}
	}%
}
\newcommand{\arrowR}[2][-3pt]{%
	\mathchoice
	{\arrowRscaled{#1}{#2}{1}}      
	{\arrowRscaled{#1}{#2}{1}}      
	{\arrowRscaled{#1}{#2}{.7}}     
	{\arrowRscaled{#1}{#2}{.5}}     
}
\newcommand{\UpR}{\arrowR{90}}
\newcommand{\DownR}{\arrowR[-2.6pt]{-90}}

\newcommand{\uc}{\mathop{\vartriangle}\mspace{-2mu}}
\newcommand{\dc}{\mathop{\triangledown}\mspace{-2mu}}

\newcommand{\ucx}{\mathop{\underline{\mspace{-1.5mu}\uc}}}
\newcommand{\dcx}{\overline{\mspace{-1.5mu}\dc}}
\newcommand{\Uc}{\blacktriangle}
\newcommand{\Dc}{\blacktriangledown}
\newcommand{\Ucx}{\underline{\Uc\mspace{-.4mu}}}
\newcommand{\Dcx}{\overline{\Dc\mspace{-.4mu}}}
\DeclareRobustCommand{\ucdcTikz}[1]{%
	\mathord{%
		\tikz[
		baseline=-.05ex,
		line cap=round,
		line join=round,
		scale=#1,
		]{
			\draw[line width=.4pt]
			(0,0) -- (.48ex,1.30ex) -- (1.68ex,1.30ex) -- (1.20ex,0) -- cycle;
			\draw[line width=.4pt]
			(.48ex,1.30ex) -- (1.20ex,0);
		}%
	}%
}
\DeclareRobustCommand{\ucdcTEMP}{%
	\mathord{%
		\mathchoice
		{\ucdcTikz{1}}
		{\ucdcTikz{1}}
		{\ucdcTikz{.65}}
		{\ucdcTikz{.55}}
	}%
}
\newcommand{\ucdc}{{\ucdcTEMP}}

\DeclareRobustCommand{\ucdcxTikz}[3]{%
	\mathord{%
		\tikz[
		baseline=-.05ex,
		line cap=round,
		line join=round,
		scale=#1,
		]{
			\draw[line width=.4pt]
			(0,0) -- (.48ex,1.30ex) -- (1.68ex,1.30ex) -- (1.20ex,0) -- cycle;
			\draw[line width=.4pt]
			(.48ex,1.30ex) -- (1.20ex,0);
			\draw[line width=.35pt]
			(-.02ex,#2) -- (1.18ex,#2);
			\draw[line width=.35pt]
			(.52ex,#3) -- (1.71ex,#3);
		}%
	}%
}
\DeclareRobustCommand{\ucdcxTEMP}{%
	\mathord{%
		\mathchoice
		{\ucdcxTikz{1}{-.34ex}{1.64ex}}
		{\ucdcxTikz{1}{-.34ex}{1.64ex}}
		{\ucdcxTikz{.65}{-.34ex}{1.64ex}}
		{\ucdcxTikz{.55}{-.34ex}{1.64ex}}
	}%
}
\newcommand{\ucdcx}{{\ucdcxTEMP}}

\newcommand{\cat}[1]{\textnormal{\textbf{#1}}}
\newcommand{\Set}{\cat{Set}}
\newcommand{\Top}{\cat{Top}}
\newcommand{\Frm}{\cat{Frm}}
\newcommand{\Loc}{\cat{Loc}}
\newcommand{\ucdcFrm}{\ucdcTikz{1}\Frm}

\newcommand{\rLoc}{\cat{rLoc}}
\newcommand{\oLoc}{\cat{oLoc}}
\newcommand{\rocLoc}{\cat{rocLoc}}

\newcommand{\rTop}{\cat{rTop}}
\newcommand{\oTop}{\cat{oTop}}
\newcommand{\rocTop}{\cat{rocTop}}

\newcommand{\LeqLoc}{{\Leq}\cat{Loc}}
\newcommand{\para}{\rotatebox[origin=c]{45}{\scalebox{0.8}{$=$}}}
\newcommand{\pLeqLoc}{\LeqLoc_{\mspace{-1mu}\para}}
\newcommand{\pucdcFrm}{\ucdcFrm_\leq}

\DeclareMathOperator{\pt}{\mathsf{pt}}
\DeclareMathOperator{\loc}{\mathsf{L}}
\newcommand{\cone}{\mathsf{cone}}
\newcommand{\rel}{\mathsf{rel}}

\newcommand{\calB}{\mathcal{B}}
\newcommand{\calE}{\mathcal{E}}
\newcommand{\calM}{\mathcal{M}}

\makeatletter
\let\amsart@setauthors\@setauthors 
\makeatother
\makeatletter
\let\@titleblockemails\@empty
\let\@titleblockdate\@empty
\let\@date\@empty
\renewcommand{\email}[2][]{%
	\ifx\@empty\@titleblockemails
	\gdef\@titleblockemails{\href{mailto:#2}{\nolinkurl{#2}}}%
	\else
	\g@addto@macro\@titleblockemails{,\space\href{mailto:#2}{\nolinkurl{#2}}}%
	\fi
}
\renewcommand{\date}[1]{%
	\gdef\@titleblockdate{#1}%
	\global\let\@date\@empty
}
\def\@set@authors@addresses{%
	\amsart@setauthors
	\begingroup
	\par\vspace{0\baselineskip}
	\centering
	\def\author##1{}
	\def\\{\protect\linebreak}%
	\def\address##1##2{\par{\footnotesize\itshape ##2\par}}%
	\addresses
	\ifx\@empty\@titleblockemails\else
	\par{\footnotesize\@titleblockemails\par}%
	\fi
	\ifx\@empty\@titleblockdate\else
	\par\smallskip{\footnotesize\@titleblockdate\par}%
	\fi
	\endgroup
	\par
}
\makeatother

\title{Localic Relations with Open Cones}
\author{Nesta van der Schaaf}
\address{Université Paris-Saclay, CNRS, CentraleSupélec, ENS Paris-Saclay, Inria,\\ Laboratoire Méthodes Formelles, 91190, Gif-sur-Yvette, France
\\
\normalfont{\url{nesta.van-der-schaaf@inria.fr}}
\\
\normalfont{\url{https://nvds.site/}}
}
\date{2 September 2026}

\begin{document}
	\begin{abstract}
		Localic relations are relations internal to the category of locales, forming the point-free analogues of set-theoretic relations, and providing the general backdrop of localic order theory. This work studies \emph{open cone} localic relations, whose source and target maps are open, and provides a frame-theoretic description via point-free up and down closure operators, called~\emph{cones}.
		
		The cones arising from open cone localic relations form \emph{join-preserving} and \emph{parallel} pairs of maps on the underlying frame. Axiomatising this structure, a frame equipped with such a pair of cones is called a \emph{conic frame}. The main construction shows that, conversely, any conic frame induces a localic relation with open cones, whose cones are exactly the given ones. The main result is an adjunction with identity counit between the category of locales equipped with open cone localic relations, and the opposite of the category of conic frames. The unit gives a strongly dense inclusion of an open cone localic relation into the relation induced by its own cones.
		
		Fixed points of the adjunction are those relations recovered by their cones, and include kernel pairs of open maps and all weakly closed localic relations with open cones. As a special case we recover Kock's Godement theorem for locales. Moreover, for fixed points, internal reflexivity and transitivity are completely characterised in terms of the cones. 
	\end{abstract}
	\maketitle
	\setcounter{tocdepth}{1}
	\tableofcontents

\section{Introduction}
\label{section:introduction}
\emph{Localic relations} are relations internal to the category $\Loc$ of locales~(recalled in~\cref{section:frames}), and therefore serve as point-free analogues of set-theoretic relations. As such, they form the general backdrop of localic order theory, providing the foundations for point-free analogues of topological order theory~\cite{nachbin1965TopologyOrder}, as well as localic lattice theory~\cite{johnstone1985VietorisLocalesLocalic}.

Localic relations also form a natural starting point for (constructive) Stone-type dualities for spaces with additional structure~\cite{priestley1972OrderedTopologicalSpaces,esakia1974topologicalKripkemodels} or lattices with additional structure~\cite{jonsson1951BooleanAlgebrasOperators,petrovich1996DistributiveLatticesOperator}. For instance, localic partial orders are used by Townsend in~\cite{townsend1996preframeTechniquesConstructiveLocale,townsend1997LocalicPriestleyDuality} to provide a constructive Priestley duality for distributive lattices~(see also~\cite{korostenski2007LaxProperMaps}). Localic preorders specifically could provide a localic counterpart to other frame-theoretic treatments of ordered spaces such as biframes~\cite{banaschewski1983BiframesBispaces}, d-frames~\cite{jung2010BitopologicalNatureStone}, and the recently introduced ad-frames~\cite{goubault2026StoneDualityPreordered}. The present work grew out of similar considerations, in an attempt to refine the duality between ordered topological spaces and `ordered locales' in~\cite{heunen2024OrderedLocales}~(discussed more in~\cref{section:Egli-Milner ordered locales}). One of the general themes in duality theory is to find topological representations of algebraic structures. Here we approach the question from the localic side: admitting locales as the point-free counterparts of topological spaces, then structurally we can define localic relations and localic preorders as those internal to the category~$\Loc$. The question then becomes what these structures correspond to algebraically. This paper provides one such `externalisation' procedure, explained further below. 

Localic equivalence relations are important within locale theory itself for descent theory, and to describe quotients and proper maps~\cite{vermeulen1994ProperMapsLocales,plewe1997LocalicTriquotientMaps,manuell2023PresentingQuotientLocales}. For example, a sufficient condition for effectiveness of localic equivalence relations was found in~\cite{kock1989GodementTheoremLocales}, which we generalise in the present work~(\cref{section:strong density}). Subsequently, localic relations appear in topos theory~\cite{johnstone2002Elephant2}, where they are for instance used to represent \'etendues~(`locally localic' toposes) via certain sheaves of localic equivalence relations~\cite{kock1992EtendueEquivalenceRelation}. Since localic equivalence relations may be regarded as special types of localic groupoids, this also connects to the more general Joyal-Tierney representation theorem of Grothendieck toposes~\cite{joyal1984ExtensionGaloisTheory}, and the more recent work using quantale theory~\cite{resende2007EtaleGroupoidsQuantales,quijano2019EffectiveEquivalenceRelations}.

Connecting to logic, a localic relation could be seen as a point-free analogue of a Kripke frame, providing the relational semantics for modal logic~\cite{kripke1963SemanticalAnalysisModal,zach2025BoxesDiamonds}. Thus localic relations could provide a natural setting for intuitionistic~\cite{simpson1994ProofTheorySemanticsIntuitionistic}, (point-free) topological~\cite{aiello2003ReasoningSpaceModal}, or positive modal logics~\cite{dunn1995PositiveModalLogic,celani1997NewSemanticsPositive}. In turn, this might also connect to transition systems and model checking in computer science~\cite{baier2008PrinciplesModelChecking}. Finally, we mention that localic preorders could find application in the mathematical foundations of physics, in particular for localic or constructive treatments of Lorentzian causality theory~\cite{penrose1972TechniquesDifferentialTopology} or causal set theory~\cite{bombelli1987SpacetimeCausalSet}.

All that being said, working internally we lose the convenient point-wise description of relations in terms of pairs~$xRy$. Instead, a localic relation~$R$ on~$X$ is a sublocale~${R\rightarrowtail X\times X}$, and so to study them in detail we need to combine two concepts that can be technically difficult to handle: sublocales and product locales. 

The present work attempts to alleviate some of this difficulty by providing a frame-theoretic description of the class of localic relations whose source and target maps are open, called \emph{open cone relations}, in terms of localic up and down closure operators. While strong, the open cone assumption appears also in the literature on localic equivalence relations~\cite{kock1989GodementTheoremLocales}, and is often assumed more generally for localic groupoids~\cite{joyal1984ExtensionGaloisTheory,resende2007EtaleGroupoidsQuantales}.

Openness guarantees that these closure operators can be described by a pair of functions~${\Up,\Down\colon \Opens X\to \Opens X}$ on the underlying frame, which we generically call \emph{cones}. This generalises the cones~$\up,\down\colon \Powerset(S)\to \Powerset(S)$ on the powerset of a space~$S$ defined in the usual way:
\[
\up A = \{x\in S: \exists a\in A: aRx\}
\quad\text{and}\quad
\down A = \{x\in S: \exists a\in A: xRa\}.
\]

The assignment~${R\mapsto \Up,\Down}$ produces a pair of functions that are join-preserving and \emph{parallel}, generalising the notion of conjugate pairs from the theory of Boolean algebras with operators~\cite{jonsson1951BooleanAlgebrasOperators}. Axiomatising this structure, we get the main new notion: a \emph{conic frame}~$(L,\uc,\dc)$, consisting of a frame~$L$ equipped with a join-preserving parallel pair of functions~${\uc,\dc\colon L\to L}$.

Conversely, the main construction in this paper defines a localic relation~$R_\ucdc$ out of a conic frame structure~$\uc,\dc$. Spatially, a relation~$R$ can be straightforwardly recovered from its cones~$\up,\down$ and singletons, since~$x\in \down \{y\}$ iff~$xRy$. Generalising to point-free language, from~$\uc,\dc$ we define a binary relation~$\sim$ on the coproduct frame~$L\tensor L$ that emulates the behaviour of the frame congruence of a would-be relation, and then use the techniques from~\cite[\S III.11]{picado2012FramesLocalesTopology} and~\cite{moshier2017GeneratingSublocalesSubsets} to generate the localic relation~$R_\ucdc$ corresponding to the frame congruence generated by~$\sim$. This gives an assignment~${\uc,\dc\mapsto R_\ucdc}$ of a conic frame structure to an open cone localic relation, where~$R_\ucdc$ is the universal relation with the cones~$\uc,\dc$ in the sense that~$R\subseteq R_\ucdc$ iff~$\Up\sqleq \uc$ and~$\Down\sqleq \dc$.

Putting these constructions together, our main~\cref{theorem:adjunction} is an adjunction with identity counit
\[
\begin{tikzcd}[cramped,column sep=1.5cm]
	\rocLoc & {\ucdcFrm^\op}
	\arrow[""{name=0, anchor=center, inner sep=0}, "{\cone}", shift left=1.6, from=1-1, to=1-2]
	\arrow[""{name=1, anchor=center, inner sep=0}, "{\rel}", shift left=1.6, from=1-2, to=1-1]
	\arrow["\dashv"{anchor=center, rotate=-90}, draw=none, from=0, to=1]
\end{tikzcd}
\]
between the category~$\rocLoc$ of locales equipped with open cone relations, and the category~$\ucdcFrm$ of conic frames. In both categories, morphisms are suitable point-free analogues of monotone functions between sets. Thus, we emphasise that composition in~$\rocLoc$ is not relational composition (which for locales is not associative).

That the counit is the identity corresponds to the fact that~${\uc,\dc\mapsto R_\ucdc\mapsto \Up,\Down}$ indeed recovers~$\uc,\dc$. On the other hand we have an assignment~${R\mapsto \Up,\Down \mapsto R_\UpDown}$ and an inclusion~$R\subseteq R_\UpDown$, corresponding to the unit of the adjunction.
	
Recalling that fixed points of an adjunction are those objects in the respective categories that make the unit or counit an isomorphism, we see immediately that all conic frames are fixed points. On the other hand, an open cone localic relation~$R$ forms a fixed point iff~$R\cong R_\UpDown$. Even in the spatial setting there are counterexamples~(\ref{counterexample:relation on Sierpinski}). Nonetheless, we show that the relation~$R_\UpDown$ cannot deviate from~$R$ too much, since the canonical inclusion~$R\subseteq R_\UpDown$ is strongly dense, a notion of density introduced in~\cite{johnstone1989ConstructiveClosedSubgroup} that is constructively stronger than ordinary density. This has the immediate consequence that any~(weakly) closed localic relation with open cones is a fixed point. We show that the inclusion~$R\subseteq R_\UpDown$ generalises the strongly dense inclusion of a localic equivalence relation into its induced kernel pair, and so as a special case we obtain Kock's Godement theorem for locales~\cite{kock1989GodementTheoremLocales}.

Lastly, to be able to talk about localic preorders and localic equivalence relations, we study in~\cref{section:preorders} how reflexivity, transitivity and symmetry can be described in terms of cones. These axioms can be stated internally in terms of finite limits, and so make sense for any localic relation~$R$. For an open cone relation satisfying these properties, this implies the usual corresponding properties of the cones~$\Up,\Down$. The converse is not generally true~(\cref{counterexample:preorder via cones}). However, a fixed point~$R_\ucdc$ is reflexive/transitive iff~$\uc,\dc$ are inflationary/subidempotent, and it is symmetric iff~$\uc=\dc$. 

Connecting back to duality theory, in~\cite{schaaf2026LocalicEsakiaDuality} we build on Townsend's localic Priestley duality of distributive lattices~\cite{townsend1997LocalicPriestleyDuality} and use the recently introduced Heyting frames~\cite{bezhanishvili2023FrametheoreticPerspectiveEsakia} to describe a constructive, localic Esakia duality. For this, we define a one-sided version of conic frames~$(L,\dc)$ and generalise the adjunction~$\cone\dashv\rel$ with localic relations that have open source map.

\subsection*{Overview}
We start in~\cref{section:frames} with a recap of the necessary theory on frames and locales, in particular focusing in~\cref{section:sublocales} on the techniques from~\cite{picado2012FramesLocalesTopology,moshier2017GeneratingSublocalesSubsets} on generating sublocales. \cref{section:spatial intuition} then outlines the spatial intuition behind the two main constructions $R\mapsto \Up,\Down$ and $\uc,\dc\mapsto R_\ucdc$. The definitions and some basic examples on (open cone) localic relations are in~\cref{section:localic relations}, concluding with the assignment~$R\mapsto \Up,\Down$. On the frame-theoretic side, we introduce and study properties of parallelness in~\cref{section:parallelness} and give the main new definition of conic frames in~\cref{section:conic frames}, where we also prove most of the necessary technical results needed later. Putting those results together, in~\cref{section:localic relation induced by cones} we define the construction~$\uc,\dc\mapsto R_\ucdc$ and show that this relation again has open cones. This is all tied together in~\cref{section:adjunction} by proving that the two constructions in fact form part of a pair of adjoint functors~$\cone\dashv\rel$. The fixed points of this adjunction, including kernel pairs of open maps and weakly closed localic relations with open cones, are studied in~\cref{section:fixed points}. In~\cref{section:strong density} we show that the inclusion $R\subseteq R_\UpDown$ is strongly dense, from which in \cref{corollary:Kock} we reobtain Kock's Godement theorem~\cite{kock1989GodementTheoremLocales}. In~\cref{section:preorders} we study relational properties such as reflexivity and transitivity in terms of cones, and in~\cref{section:Egli-Milner ordered locales} we discuss the relation between conic frames and the previously studied notion of `ordered locale' from~\cite{heunen2024OrderedLocales}. Closing the paper with~\cref{section:discussion}, we provide final comments, discussion, and directions for future research.

\section{Frames and locales}
\label{section:frames}
In this section we recall the necessary background theory on frames and locales. For general textbook accounts we refer to~\cite{johnstone1982StoneSpaces,picado2012FramesLocalesTopology}. First, recall that a \emph{complete lattice}~$L$ is a partial order that has all meets~(greatest lower bounds) and joins~(least upper bounds). We denote the corresponding order of a lattice by~$\sqleq$, to be thought of as a generalised inclusion relation of regions. A \emph{frame} is a complete lattice~$L$ in which the infinite distributivity law holds:
\[
x \wedge \bigvee y_i = \bigvee x\wedge y_i.
\]
Morphisms of frames are functions $h\colon L\to M$ that preserve finite meets and all joins. This results in the category~$\Frm$ of frames. The category of \emph{locales} is defined as the opposite, namely~${\Loc:= \Frm^\op}$. While frames are algebraic objects, we interpret locales as \emph{spaces}. Thus, a \emph{locale} $X$ is formally defined by its underlying \emph{frame of opens} $\Opens X$, and a morphism of locales $f\colon X\to Y$ is just a morphism of frames going the other way, denoted by $f^{-1}\colon \Opens Y\to \Opens X$. Its right adjoint~(recalled in~\cref{section:galois adjunctions}) is denoted~$f^{-1}\dashv f_\ast$.

Denoting by~$\Top$ the category of topological spaces and continuous maps, \emph{(generalised) Stone duality} is the statement that there is an adjunction $\loc\dashv\pt$ between certain functors $\loc\colon \Top\to \Loc$ and $\pt\colon \Loc\to \Top$. Here the functor $\loc$ takes a topological space $S$ and produces the locale $\loc(S)$ whose underlying frame of opens is defined by the topology~$\Opens S\subseteq\Powerset(S)$ of~$S$. We refer to~\cite[\S IX.3]{maclane1994SheavesGeometryLogic} for details. A locale $X$ is called \emph{spatial} if it is a fixed point of this adjunction, meaning equivalently~${X=\loc(S)}$ for some~${S\in\Top}$.

The rest of this section recaps necessary definitions and results on sublocales, Galois adjunctions, and frame coproducts. We also need the following general concept. Just as in topology, a frame may be generated by a basis. A \emph{join-basis} for a frame $L$ is a subset $\calB\subseteq L$ such that every $x\in L$ can be written as the join over elements in $\calB$. 

\begin{lemma}
	\label{lemma:equality if contain same basics}
	Let $\calB$ be a join-basis for $L$, and take $x,y\in L$. If $b\sqleq x$ iff $b\sqleq y$ for all $b\in\calB$, then $x=y$.
\end{lemma}
\begin{proof}
	We can write each element in $L$ as the join over the basis elements it contains, so $x=\bigvee\{b\in\calB:b\sqleq x\}= \bigvee\{b\in\calB: b\sqleq y\}=y$.
\end{proof}

\subsection{Sublocales}
\label{section:sublocales}
\emph{Sublocales} are the appropriate localic analogue of subspaces in point-set topology. There are several equivalent ways of presenting them, each with their own advantage in different contexts. We briefly describe the most important ones here, but leave full details and their precise translations to~\cite[\S III]{picado2012FramesLocalesTopology}. First recall the following basic terminology, which will also be used later in the paper.

\begin{definition}
	\label{definition:closure operator}
	A function $c\colon P\to P$ on a partially ordered set is called:
	\begin{itemize}
		\item \emph{inflationary} if $x\leq c(x)$ for all $x\in P$;
		\item \emph{subidempotent} if $c(c(x))\leq c(x)$ for all $x\in P$;
		\item \emph{idempotent} if $c(c(x))=c(x)$ for all $x\in P$;
		\item a \emph{closure operator} if it is monotone, inflationary and subidempotent.
	\end{itemize}
\end{definition}

Fix a frame $L$. A \emph{sublocale set} is a subset $S\subseteq L$ that is closed under all meets of~$L$, and closed under the Heyting implication in the sense that for all~$s\in S$ and~$x\in L$ we have~$x\to s\in S$. A \emph{frame congruence} is an equivalence relation~$\equiv$ on~$L$ such that ${\equiv}\subseteq L\times L$ is a subframe (closed under finite meets and all joins). A \emph{nucleus} is a closure operator $\nu\colon L\to L$ that preserves binary meets.

\begin{definition}
	\label{definition:sublocale}
	A \emph{sublocale} of a locale $X$ is equivalently one of the following:
	\begin{itemize}
		\item an (isomorphism class of an) extremal monomorphism $A\rightarrowtail X$;
		\item a sublocale set $S$ of $\Opens X$;
		\item a frame congruence $\equiv$ on $\Opens X$;
		\item a nucleus $\nu$ on $\Opens X$.
	\end{itemize}
\end{definition}

Inclusion of sublocales is denoted $\subseteq$, which is structurally just a subobject inclusion of extremal monomorphisms, but can be more concretely represented by literal subset inclusions of sublocale sets, or as reverse inclusion of frame congruences or nuclei. For fixed~$X$, this gives a lattice $\Sl(X)$ of sublocales, which is in fact a \emph{coframe}.

Since frames and locales are formally dual, sublocales are represented by quotients in the category of frames, and since frames are algebraic objects these are in turn described by frame congruences. It is possible to start with a basic relation~$\sim$ on the frame and then generate the smallest congruence containing it, but this may not be straightforward to calculate. Instead, we use the techniques from \cite[\S III.11]{picado2012FramesLocalesTopology} and~\cite{moshier2017GeneratingSublocalesSubsets} to construct the sublocale directly. Note here that~$\sim$ is just a subset~${\sim}\subseteq L\times L$, without any further assumptions.

\begin{definition}
	\label{definition:saturated}
	For a binary relation ${\sim}\subseteq L\times L$ on a frame, an element~$s\in L$ is called \emph{$\sim$-saturated} if:
	\[
	\forall a,b,c\in L: a \sim b \quad\implies\quad \left(a\wedge c \sqleq s \text{ iff } b\wedge c\sqleq s\right).
	\]
	The set of $\sim$-saturated elements is denoted $L/{\sim}$.
\end{definition}

\begin{remark}[{\cite[\S 11.2]{picado2012FramesLocalesTopology}}]
	The set of $\sim$-saturated elements $L/{\sim}$ is a sublocale set of $L$, and is thus itself a frame. Equivalently, it is the sublocale described by the frame congruence~$\langle\sim\rangle$ generated by~$\sim$:
	\[
	L/{\sim} = L/\langle \sim\rangle.
	\]
\end{remark}

\begin{remark}
	\label{remark:quotient map}
	It can be shown that there is a frame map
	\[
	\mu_\sim \colon L \longrightarrow L/{\sim};
	\qquad
	\mu_\sim(x)=\bigwedge\{s\in L/{\sim}: x\sqleq s\},
	\]
	taking an element in the frame and producing the smallest $\sim$-saturated element containing it, called the \emph{quotient map} of~$\sim$. Seen as a map on~$L$ itself, $\mu_\sim$ is equivalently the nucleus of the induced sublocale, so in particular it is inflationary and subidempotent, and the elements in~$L/{\sim}$ are just the fixed points of~$\mu_\sim$.
\end{remark}

The quotient map~$\mu_\sim$ indeed behaves like a quotient. A function $h$ defined on $L$ is said to \emph{equalise} $\sim$ if $x\sim y$ implies $h(x)=h(y)$.

\begin{theorem}[{\cite[\S 11.3.1]{picado2012FramesLocalesTopology}}]
	\label{theorem:frame quotient theorem}
	Let $\sim$ be a binary relation on a frame $L$. Then the function $\mu_\sim\colon L\to L/{\sim}$ is a frame map that equalises~$\sim$. If~$h\colon L\to M$ is a frame map equalising $\sim$, there exists a unique frame map $\bar h$ making the following diagram commute:
	\[
	\begin{tikzcd}[cramped, row sep = 1.5em]
		L & {L/{\sim}} \\
		& M.
		\arrow["{\mu_\sim}", from=1-1, to=1-2]
		\arrow["h"', from=1-1, to=2-2]
		\arrow["{\exists! \bar h}", dashed, from=1-2, to=2-2]
	\end{tikzcd}
	\]
\end{theorem}

Lastly, in our setting the relation~$\sim$ will already define a meet-semilattice congruence with respect to some join-basis, and the following simplified saturation condition will come in handy.

\begin{proposition}[{\cite[\S 11.4]{picado2012FramesLocalesTopology}}]
	\label{proposition:saturated simplified when meet preserved}
	Let $\calB$ be a join-basis for $L$, and let~${{\sim}\subseteq L\times L}$ be such that for any $x,y\in L$ and $b\in\calB$ we have $x\sim y$ implies $(x\wedge b)\sim(y\wedge b)$. Then an element $s\in L$ is $\sim$-saturated iff
	\[
	x\sim y \quad\implies\quad \left(x\sqleq s \text{ iff } y\sqleq s\right).
	\]
\end{proposition}

\subsection{Galois adjunctions}
\label{section:galois adjunctions}
We briefly recall the necessary basics on \emph{(Galois) adjunctions}, and refer to~\cite[\S O-3]{gierz2003ContinuousLatticesDomains} or~\cite[\S AI.5]{picado2012FramesLocalesTopology} for details.

\begin{definition}
	A \emph{Galois adjunction} $f\dashv g$ consists of a pair of monotone functions~${f\colon P\to Q}$ and~${g\colon Q\to P}$ between posets satisfying:
	\[
	f(p)\leq q
	\quad\iff\quad
	p\leq g(q).
	\]
	The function $f$ is called the \emph{left adjoint}, and $g$ the \emph{right adjoint}.
\end{definition}

\begin{theorem}
	\label{theorem:adjoint iff preserves meets/joins}
	If $L,M$ are complete lattices, then a monotone function~${f\colon L\to M}$ is a left/right adjoint iff it preserves all joins/meets. Explicitly, if~$f\dashv g$ then
	\[
	f(x)= \bigwedge\{y\in M: x\leq g(y)\}
	\quad\text{and}\quad
	g(y)=\bigvee\{x\in L: f(x)\leq y\}.
	\]
\end{theorem}

\begin{remark}
	A fact we will often use in proofs is that if there is an adjunction~${f\dashv g}$, then the function $g\circ f$ is a closure operator. In particular note that, $f(x)\leq f(x)$, which is always true, holds iff $x\leq g \circ f(x)$.
\end{remark}

\subsection{Coproducts of frames}
For the purposes of this work we only require binary coproducts, though of course $\Frm$ is cocomplete~\cite[\S IV.4]{picado2012FramesLocalesTopology}. If $L,M$ are frames then their \emph{coproduct} $L\tensor M$ is the frame freely generated by the symbols~$x\tensor y$, where $x\in L$ and $y\in M$, subject to the following equations~\cite{dowker1977SumsCategoryFrames,johnstone1991PreframePresentationsPresent,picado2015NotesProductLocales}:
\begin{gather*}
	(x_1\tensor y_1)\wedge (x_2\tensor y_2) = (x_1\wedge x_2)\tensor (y_1\wedge y_2),
	\\
	\bigvee (x_i\tensor y) = \left(\bigvee x_i\right)\tensor y,
	\quad\text{and}\quad
	\bigvee(x\tensor y_i)= x\tensor \left(\bigvee y_i\right).
\end{gather*}
Thus $L\tensor M$ has join-basis $\{x\tensor y: x\in L,y\in M\}$, whose elements are called \emph{(basic) rectangles}. Note that this basis is in fact a sub-meet-semilattice. This construction is analogous to how the product topology is defined on the Cartesian product of two topological spaces, but note that spatial and localic products may differ quite drastically (for example, the localic product $\loc(\mathbb{Q})\times\loc(\mathbb{Q})$ is no longer spatial~\cite[\S II.2.14]{johnstone1982StoneSpaces}).

We collect some structural results about binary frame coproducts from~\cite[\S IV.5.5]{picado2012FramesLocalesTopology} that will be used in future calculations.

\begin{lemma}
	\label{lemma:projections}\label{lemma:coproduct of frame maps}\label{lemma:pair of locale maps}
	We have:
	\begin{itemize}
		\item the coproduct injection maps $\iota_1\colon L\to L\tensor M$ and $\iota_2\colon M\to L\tensor M$ are
		\[
		\iota_1(x) = x\tensor \top
		\quad\text{and}\quad
		\iota_2(y)= \top\tensor y;
		\]
		
		\item for any frame maps $h,k$ the map $h\tensor k$ is given on rectangles by
		\[
		(h\tensor k)(x\tensor y) = h(x)\tensor k(y);
		\]
		
		\item the codiagonal $L\tensor L \to L$ is given on rectangles by
		\[
		x\tensor y \longmapsto x\wedge y;
		\]
		
		\item for frame maps $h\colon L\to N$ and $k\colon M\to N$ the copair ${[h,k]\colon L\tensor M \to N}$ is given on rectangles by
		\[
		[h,k](x\tensor y)= h(x)\wedge k(y).
		\]
	\end{itemize}
\end{lemma}

\section{Spatial intuition}
\label{section:spatial intuition}
In this more informal section we outline the spatial intuitions for the definitions and results in the later sections. As outlined in the~\nameref{section:introduction}, the main two constructions involve translating between relations $R$ and the corresponding cones~$\up,\down$. For relations on spaces this is a straightforward affair, as we spell out below. The main task is to translate this into a suitable point-free description. This is summarised in~\cref{intuition:cones from st,intuition:sim}.

A \emph{relation} on a topological space $S$ is just a subset~$R\subseteq S\times S$. We use the common notation of $xRy$ to denote $(x,y)\in R$. When needed, we assume $R$ is equipped with the subspace topology. To help intuition, we provide basic illustrations of the main concepts and constructions, where for simplicity a generic relation $R$ is drawn as the grayed out upper diagonal region:

\smallskip
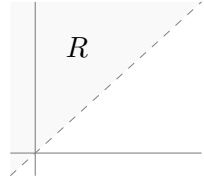
\def\xmax{2.2} 
\def\ymax{2} 
\hfil
\begin{tikzpicture}[baseline={(0,0)}]
	\def\a{0.2*\xmax} 
	\def\b{0.7*\xmax} 
	
	\pgfmathsetmacro{\xmin}{-0.15*\xmax}
	\pgfmathsetmacro{\ymin}{-0.15*\ymax}
	\coordinate (D1) at (\xmin,\ymin);
	\coordinate (D2) at (\xmax,\ymax);
	
	\begin{scope}
		\clip (D1) -- (D2) -- (\xmax,\ymax) -- (\xmin,\ymax) -- cycle;
		\fill[bg] (\xmin,\ymin) rectangle (\xmax,\ymax);
	\end{scope}
	
	\pgfmathsetmacro{\m}{(\ymax-\ymin)/(\xmax-\xmin)}
	\pgfmathsetmacro{\c}{\ymin - \m*(\xmin)}
	\pgfmathsetmacro{\ya}{\m*\a + \c}
	\pgfmathsetmacro{\yb}{\m*\b + \c}
	
	\draw[gray] (0,-0.15*\ymax) -- (0,\ymax);
	\draw[gray] (-0.15*\xmax,0) -- (\xmax,0);
	\draw[dashed, gray, very thin] (D1) -- (D2);

	\node[black] at (\xmax*.25,\ymax*.7) {$R$};
\end{tikzpicture}\hfil

\begin{definition}
	Let $R$ be a relation on a space $S$. The \emph{(induced) cones} are the pair of functions~$\up,\down\colon \Powerset(S)\to \Powerset(S)$ defined on the powerset by
	\[
	\up A:=\{x\in S: \exists a\in A: a R x\}
	\quad\text{and}\quad
	\down A:= \{y\in S: \exists a\in A: y R a\}.
	\]
\end{definition}
We adopt the following visual intuition for regions $A\in\Powerset(S)$ and their cones. Note that, only for simplicity, the cones are drawn as if the underlying relation is a preorder (so that the cones contain $A$):

\smallskip
\hfil
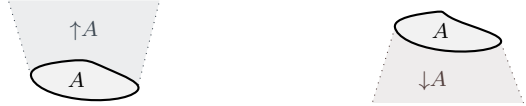
\begin{tikzpicture}[every node/.style={font=\scriptsize}]
	\def\w{1} 
	\def\h{1} 
	
	\coordinate (P1) at (-.6,.1);
	\coordinate (P2) at (0,-.2);
	\coordinate (P3) at (.8,-.1);
	\coordinate (P4) at (.4,.2);
	\coordinate (P5) at (-.1,.3);
	
	\path let \p1=(P1), \p3=(P3) in
	\pgfextra{
		\pgfmathsetlengthmacro{\xc}{0.5*(\x1+\x3)}\xdef\xc{\xc}
		\ifdim\y1>\y3 \xdef\ymax{\y1}\else \xdef\ymax{\y3}\fi
		\pgfmathsetlengthmacro{\ytop}{\ymax + \h cm}\xdef\ytop{\ytop}
		\pgfmathsetlengthmacro{\xL}{\xc - \w cm}\xdef\xL{\xL}
		\pgfmathsetlengthmacro{\xR}{\xc + \w cm}\xdef\xR{\xR}
	};
	
	\coordinate (TL) at (\xL,\ytop);
	\coordinate (TR) at (\xR,\ytop);
	
	\fill[figcolor, opacity=0.10]
	(P1) -- (TL) -- (TR) -- (P3) -- cycle;
	
	\draw[dotted, figcolor] (P1) -- (TL);
	\draw[dotted, figcolor] (P3) -- (TR);
	
	\fill[black!05, draw=black,thick]
	plot[smooth cycle, tension=0.9]
	coordinates {(P1) (P2) (P3) (P4) (P5)};
	
	\node at (0,0.05) {$A$};
	\node[figcolor] at (.1,.7) {$\up A$};
\end{tikzpicture}
\hfil\begin{tikzpicture}[every node/.style={font=\scriptsize}]
	\def\w{1} 
	\def\h{-1} 
	
	\coordinate (P1) at (-.6,.1);
	\coordinate (P2) at (0,-.2);
	\coordinate (P3) at (.8,-.1);
	\coordinate (P4) at (.2,.2);
	\coordinate (P5) at (-.1,.3);
	
	\path let \p1=(P1), \p3=(P3) in
	\pgfextra{
		\pgfmathsetlengthmacro{\xc}{0.5*(\x1+\x3)}\xdef\xc{\xc}
		\ifdim\y1>\y3 \xdef\ymax{\y1}\else \xdef\ymax{\y3}\fi
		\pgfmathsetlengthmacro{\ytop}{\ymax + \h cm}\xdef\ytop{\ytop}
		\pgfmathsetlengthmacro{\xL}{\xc - \w cm}\xdef\xL{\xL}
		\pgfmathsetlengthmacro{\xR}{\xc + \w cm}\xdef\xR{\xR}
	};
	
	\coordinate (TL) at (\xL,\ytop);
	\coordinate (TR) at (\xR,\ytop);
	
	\fill[accentcolor, opacity=0.10]
	(P1) -- (TL) -- (TR) -- (P3) -- cycle;
	
	\draw[dotted, accentcolor] (P1) -- (TL);
	\draw[dotted, accentcolor] (P3) -- (TR);
	
	\fill[black!05, draw=black,thick]
	plot[smooth cycle, tension=0.9]
	coordinates {(P1) (P2) (P3) (P4) (P5)};
	\node at (0,0.05) {$A$};
	\node[accentcolor] at (-.1,-.6) {$\down A$};
\end{tikzpicture}\hfil

Since the cones are defined in terms of an existential quantifier, we get the following important property.

\begin{proposition}
	\label{proposition:cones preserve joins in Set}
	If $R$ is a relation on a space then $\up,\down$ preserve all joins.
\end{proposition}
\begin{proof}
	This follows simply as:
	\begin{align*}
	x\in \up\bigcup A_i
	&\iff 
	\exists a\in \bigcup A_i: a R x
	\\&\iff 
	\exists i\exists a\in A_i: a R x
	\\&\iff 
	\exists i: x\in \up A_i
	\\&\iff
	x\in \bigcup \up A_i.
	\qedhere
	\end{align*}
\end{proof}

Another important property of $\up,\down$ is that they are \emph{parallel}. The notion of parallelness on meet-semilattices will be formally introduced in~\cref{section:parallelness}, but is motivated as follows. Indeed, it is clear that the cones coming from a relation $R$ are not completely independent, but share a correspondence:
\[
x\in \down \{y\}
\quad\iff\quad
x R y
\quad\iff\quad
y\in \up\{x\}.
\]
Using singletons and join-preservation, specifying one of the cones fully determines the other. Of course, generally, singletons may not be available, and being parallel can be captured by the following more general condition:
\[
\up A \cap B \subseteq \up (A\cap \down B)
\qquad\text{and}\qquad
A\cap \down B \subseteq \down (\up A\cap B),
\]
for subsets~$A,B\subseteq S$. It is straightforward to see that plugging back in singletons and using join-preservation that~$\up,\down$ are parallel in this sense precisely when~${x\in \down\{y\}}$ iff~${y\in\up\{x\}}$. From this we immediately get the following.

\begin{lemma}
	There is a bijective correspondence between relations~$R$ on a set~$S$ and join-preserving parallel pairs $\uc,\dc$ on its powerset~$\Powerset(S)$.
\end{lemma}

\begin{intuition}
	\label{intuition:cones}
	Building on this insight, if we replace the powerset~$\Powerset(S)$ by some lattice~$L$, of which we think as modeling the regions in some space, then a would-be relation on that space can be modeled by a pair of join-preserving parallel cones~$\uc,\dc$ on~$L$. This intuition forms the basis of our main new notion: \emph{conic frames} in~\cref{definition:conic frame}.
\end{intuition}

So far we have ignored topological structure on the base space. The rest of this section outlines the spatial intuition behind how to translate between relations and cones while being susceptible to point-free generalisation.

\subsection{Relations to cones}
\label{section:relations to cones}
The first task is to translate the definition of~$\up,\down$ into point-free language. In this, the \emph{source} and \emph{target maps}~${s,t\colon R\to S}$ of the relation play a central role. Recall that on a pair~$xRy$ they are defined just as the projections~$s(x,y)=x$ and~$t(x,y)=y$. For arbitrary subsets~$A,B\subseteq S$ we then get:
\[
	s(A\times B\cap R)= A\cap \down B
	\quad\text{and}\quad
	t(A\times B\cap R)=\up A\cap B.
\]
Next, the preimages are calculated as
\[
	s^{-1}(A) = A\times S \cap R
	\quad\text{and}\quad
	t^{-1}(B) = S\times B\cap R.
\]
Combining these two formulas, we see that the cones~$\up,\down$ can be recovered from the source and target maps~$s,t$ as:
\[
	\up A = t s^{-1}(A)
	\qquad\text{and}\qquad
	\down B = st^{-1}(B).
\]
The visual intuition behind these formulas is pictured in the graph as follows:

\smallskip
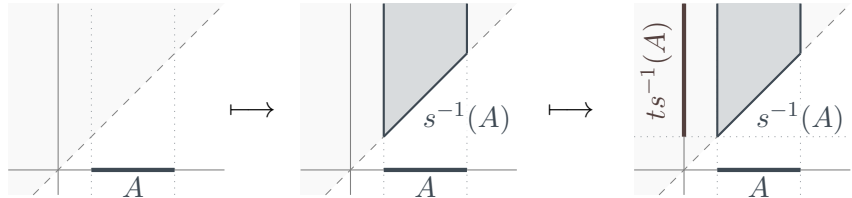
\def\xmax{2.2} 
\def\ymax{2.2} 
\def\a{0.2*\xmax} 
\def\b{0.7*\xmax} 
\def\xmin{-0.3*\xmax}
\def\ymin{-0.15*\ymax}
\pgfmathsetmacro{\BIG}{5*max(\xmax-\xmin,\ymax-\ymin)}
\hfil
\begin{tikzpicture}[baseline={(0,0)}]
	\pgfmathsetmacro{\dmin}{max(\xmin,\ymin)}
	\pgfmathsetmacro{\dmax}{min(\xmax,\ymax)}
	\coordinate (D1) at (\dmin,\dmin);
	\coordinate (D2) at (\dmax,\dmax);
	
	\begin{scope}
		\path[use as bounding box] (\xmin,\ymin) rectangle (\xmax,\ymax);
		\clip (\xmin,\ymin) rectangle (\xmax,\ymax);
		\clip (\xmin-\BIG,\xmin-\BIG) -- (\xmax+\BIG,\xmax+\BIG)
		-- (\xmax+\BIG,\ymax+\BIG) -- (\xmin-\BIG,\ymax+\BIG) -- cycle;
		\fill[bg] (\xmin,\ymin) rectangle (\xmax,\ymax);
	\end{scope}
	
	\draw[gray] (0,\ymin) -- (0,\ymax);
	\draw[gray] (\xmin,0) -- (\xmax,0);
	\draw[dashed, gray, very thin] (D1) -- (D2);
	
	\draw[\rectcol!50, dotted] (\a,\ymin) -- (\a,\ymax);
	\draw[\rectcol!50, dotted] (\b,\ymin) -- (\b,\ymax);
	
	\draw[\rectcol, ultra thick] (\a,0) -- (\b,0)
	node[below=1pt, midway, fill=white, inner sep=1pt] {$A$};
\end{tikzpicture}
\hspace{-1ex}\begin{tikzpicture}[baseline={(0,0)}]
	\node at (0,0.35*\ymax) {$\longmapsto$};
\end{tikzpicture}~
\begin{tikzpicture}[baseline={(0,0)}]
	\pgfmathsetmacro{\dmin}{max(\xmin,\ymin)}
	\pgfmathsetmacro{\dmax}{min(\xmax,\ymax)}
	\coordinate (D1) at (\dmin,\dmin);
	\coordinate (D2) at (\dmax,\dmax);
	
	\begin{scope}
		\path[use as bounding box] (\xmin,\ymin) rectangle (\xmax,\ymax);
		\clip (\xmin,\ymin) rectangle (\xmax,\ymax);
		\clip (\xmin-\BIG,\xmin-\BIG) -- (\xmax+\BIG,\xmax+\BIG)
		-- (\xmax+\BIG,\ymax+\BIG) -- (\xmin-\BIG,\ymax+\BIG) -- cycle;
		\fill[bg] (\xmin,\ymin) rectangle (\xmax,\ymax);
	\end{scope}
	
	\pgfmathsetmacro{\ya}{\a}
	\pgfmathsetmacro{\yb}{\b}
	
	\draw[gray] (0,\ymin) -- (0,\ymax);
	\draw[gray] (\xmin,0) -- (\xmax,0);
	\draw[dashed, gray, very thin] (D1) -- (D2);
	
	\draw[\rectcol!50, dotted] (\a,\ymin) -- (\a,\ymax);
	\draw[\rectcol!50, dotted] (\b,\ymin) -- (\b,\ymax);
	
	\draw[\rectcol, ultra thick] (\a,0) -- (\b,0)
	node[below=1pt, midway, fill=white, inner sep=1pt] {$A$};
	
	\begin{scope}
		\clip (\xmin,\ymin) rectangle (\xmax,\ymax);
		\clip (\xmin-\BIG,\xmin-\BIG) -- (\xmax+\BIG,\xmax+\BIG)
		-- (\xmax+\BIG,\ymax+\BIG) -- (\xmin-\BIG,\ymax+\BIG) -- cycle;
		\fill[\rectcol!20] (\a,0) rectangle (\b,\ymax);
	\end{scope}
	
	\draw[\rectcol, thick] (\a,\ya) -- (\a,\ymax);
	\draw[\rectcol, thick] (\b,\yb) -- (\b,\ymax);
	\draw[\rectcol, thick] (\a,\ya) -- (\b,\yb);
	
	\node[\rectcol,textbg]
	at (\xmax*.7,\ymax*.31) {$s^{-1}(A)$};
\end{tikzpicture}
\hspace{1ex}\begin{tikzpicture}[baseline={(0,0)}]
	\node at (0,0.35*\ymax) {$\longmapsto$};
\end{tikzpicture}\hspace{1.5ex}
\begin{tikzpicture}[baseline={(0,0)}]
	\pgfmathsetmacro{\dmin}{max(\xmin,\ymin)}
	\pgfmathsetmacro{\dmax}{min(\xmax,\ymax)}
	\coordinate (D1) at (\dmin,\dmin);
	\coordinate (D2) at (\dmax,\dmax);
	
	\begin{scope}
		\path[use as bounding box] (\xmin,\ymin) rectangle (\xmax,\ymax);
		\clip (\xmin,\ymin) rectangle (\xmax,\ymax);
		\clip (\xmin-\BIG,\xmin-\BIG) -- (\xmax+\BIG,\xmax+\BIG)
		-- (\xmax+\BIG,\ymax+\BIG) -- (\xmin-\BIG,\ymax+\BIG) -- cycle;
		\fill[bg] (\xmin,\ymin) rectangle (\xmax,\ymax);
	\end{scope}
	
	\pgfmathsetmacro{\ya}{\a}
	\pgfmathsetmacro{\yb}{\b}
	
	\draw[gray] (0,\ymin) -- (0,\ymax);
	\draw[gray] (\xmin,0) -- (\xmax,0);
	\draw[dashed, gray, very thin] (D1) -- (D2);
	
	\draw[\rectcol!50, dotted] (\a,\ymin) -- (\a,\ymax);
	\draw[\rectcol!50, dotted] (\b,\ymin) -- (\b,\ymax);
	\draw[\rectcol!50, dotted] (\xmin,\a) -- (\xmax,\a);
	
	\draw[\rectcol, ultra thick] (\a,0) -- (\b,0)
	node[below=1pt, midway, fill=white, inner sep=1pt] {$A$};
	
	\begin{scope}
		\clip (\xmin,\ymin) rectangle (\xmax,\ymax);
		\clip (\xmin-\BIG,\xmin-\BIG) -- (\xmax+\BIG,\xmax+\BIG)
		-- (\xmax+\BIG,\ymax+\BIG) -- (\xmin-\BIG,\ymax+\BIG) -- cycle;
		\fill[\rectcol!20] (\a,0) rectangle (\b,\ymax);
	\end{scope}
	
	\draw[\rectcol, thick] (\a,\ya) -- (\a,\ymax);
	\draw[\rectcol, thick] (\b,\yb) -- (\b,\ymax);
	\draw[\rectcol, thick] (\a,\ya) -- (\b,\yb);
	
	\node[\rectcol, textbg]
	at (\xmax*.7,\ymax*.31) {$s^{-1}(A)$};
	
	\draw[accentcolor, ultra thick] (0,\a) -- (0,\ymax)
	node[above, midway, rotate=90] {$ts^{-1}(A)$};
\end{tikzpicture}
\hfil

\smallskip\smallskip
In the localic world we are restricted to the case where $A,B\subseteq S$ are open subsets, which the preimage maps~$s^{-1}$ and~$t^{-1}$ preserve, but the image maps of~$s,t$ may not. To handle this, we impose the following compatibility condition of the relation with the topology.

\begin{definition}
	\label{definition:open cones on space}
	A relation $R$ on a space $S$ is said to have \emph{open cones} if the maps~$\up,\down$ preserve open subsets.
\end{definition}

\begin{example}
	\label{example:open cones spaces}
	We recall some examples from~\cite[\S 2.2]{schaaf2024TowardsPointFreeSpacetimes} and elsewhere:
	\begin{itemize}
		\item Any relation on a discrete space has open cones;
		\item The diagonal relation $\Delta\subseteq S\times S$ of any space has open cones;
		\item Similarly, the codiagonal relation $S\times S\subseteq S\times S$ has open cones;
		\item Any open subset $R\subseteq S\times S$ is a relation with open cones, since product projections in $\Top$ are open. However, $R$ having open cones does not imply $R$ is open as a subset of $S\times S$. Indeed, the diagonal relation trivially has open cones, but is open as a subset iff $S$ has the discrete topology;
	
		\item The real line $\mathbb{R}$ with the Euclidean topology and standard ordering~$\leq$ has open cones. In fact, any linearly ordered topological space has open cones;
		
		\item More generally, if $(P,\leq)$ is a poset, its \emph{interval topology} is generated by the subbase containing the sets $\up x$ and $\down x$, for $x\in P$. With respect to this topology, $\leq$ has open cones. Further, if $P$ is bicontinuous~\cite{gierz2003ContinuousLatticesDomains}, then intervals with respect to its way-below relation $\ll$ define a basis for a topology, with respect to which both $\leq$ and $\ll$ have open cones;
		
		\item The \emph{causal relation} $\preccurlyeq$ of a smooth spacetime is a preorder that has open cones~\cite[Theorem~3.14]{schaaf2024TowardsPointFreeSpacetimes}. The \emph{chronology relation} $\ll$ similarly is a transitive relation with open cones (in fact, it makes arbitrary subsets open)~\cite{penrose1972TechniquesDifferentialTopology}. In~\cite{martin2006DomainSpacetimeIntervals} it was shown that for a certain class of spacetimes (the globally hyperbolic ones) this is actually a special case of the previous example of bicontinuous posets;
		
		\item Preordered spaces with open cones are called \emph{I-spaces} in~\cite{priestley1972OrderedTopologicalSpaces,minguzzi2013ConvexityQuasiuniformizability};
		
		\item (Bi-)Esakia spaces are the Priestley spaces with open cones~\cite{bezhanishvili2010BitopologicalDualityDistributive};
		
		\item If $\equiv$ is an equivalence relation on a space $S$ then the quotient map ${S\to S/{\equiv}}$ is open iff~$\equiv$ has open cones;
		
		\item If $G$ is a topological group and $A\subseteq G$ any subset, then the relation defined by~$xRy$ iff~$x^{-1} y \in A$ has open cones, since multiplication by any fixed element is a homeomorphism:
		\[
		\up U = UA = \bigcup_{a\in A} Ua
		\quad\text{and}\quad
		\down U = UA^{-1} = \bigcup_{a\in A} Ua^{-1};
		\]

		\item As a special case, if $(V,\leq)$ is a preordered topological vector space, then its order is determined by its positive cone $V^+=\{v\in V: 0\leq v\}$ as~$x\leq y$ iff~$y-x\in V^+$, and so $\leq$ has open cones.
	\end{itemize}
\end{example}

We always equip a relation $R$ with the subspace topology induced by $S\times S$, meaning the basic opens in $R$ are those of the form $U\times V\cap R$, for open $U,V\subseteq S$. Recall that a continuous function $g\colon S\to T$ between topological spaces is called \emph{open} if the set-theoretic image $g(U)$ is open in $T$ for ever open subset~${U\subseteq S}$. The open cone condition can then be translated to openness of the~source~and~target~maps.

\begin{lemma}
	\label{lemma:open cones iff s t open}
	A relation $R$ has open cones iff the source and target maps are open.
\end{lemma}
\begin{proof}
	From the formulas $\up U = ts^{-1}(U)$ and $\down U = st^{-1}(U)$ we see that if~$s,t$ are open, then~$\up,\down$ preserve opens. Conversely, we have basic opens in $R$ of the form~$U\times V \cap R$ for open $U,V\subseteq S$, and if $\up,\down$ preserve opens then $s(U\times V\cap R) = U\cap \down V$ and $t(U\times V\cap R) = \up U \cap V$ are open. Since image maps preserve unions, the maps~$s$ and~$t$ preserve arbitrary opens.
\end{proof}

\begin{intuition}
	\label{intuition:cones from st}
	In summary, to get cones $\up,\down\colon\Opens S\to \Opens S$ from a relation~$R$, we assume that its source and target maps~$s,t$ are open, and define
	\[
	\up,\down\colon\Opens S\longrightarrow\Opens S;
	\qquad
	\up U = t s^{-1}(U)
	\quad\text{and}\qquad
	\down V = st^{-1}(V).
	\]
	This is now susceptible to point-free generalisation, which is done in~\cref{section:localic relations with open cones}.
\end{intuition}

\subsection{Cones to relations}
\label{section:cones to relations}
The second task is to construct a relation $R$ from a given pair of cones $\up,\down$ on~$\Opens S$. Now this needs to be a sub\emph{locale}, not just a subset.

There are at least two ways to describe the spatial intuition behind a sublocale. First, any subset $A\subseteq S$ of a topological space can be equipped with the subspace topology $\{U\cap A:U\in\Opens S\}$, making the inclusion map $i\colon A\rightarrowtail S$ an extremal monomorphism in $\Top$. This hence defines a sublocale whose opens are of the form~$U\cap A$, for~$U\in \Opens S$. On the other hand, we can represent the induced sublocale via the frame congruence
\[
U\equiv_A V
\quad\iff\quad
U\cap A = V\cap A.
\]
In this form, opens are $\equiv_A$-equivalence classes, which are interpreted as containing those opens $U,V\subseteq S$ that coincide inside of $A$, but may differ elsewhere. This same intuition is also captured by the induced nucleus $\nu_A(U) = (U\cup (S\setminus A))^\circ$, representing the maximal element in the equivalence class, taking an open and `disregarding' everything outside of $A$. In this work both interpretations are convenient, but the focus is on the frame congruence presentation.

Now applying the same intuition to a relation~$R\subseteq S\times S$, we can represent the induced sublocale either via elements of the form $O\cap R$, or as $\equiv_R$-equivalence classes $[O]_R$:

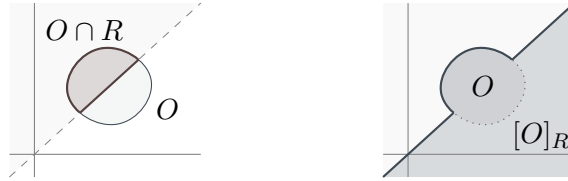
\def\xmax{2.2} 
\def\ymax{2} 
\hfil
\begin{tikzpicture}[baseline={(0,0)}]
	\def\a{0.2*\xmax} 
	\def\b{0.7*\xmax} 
	
	\pgfmathsetmacro{\xmin}{-0.15*\xmax}
	\pgfmathsetmacro{\ymin}{-0.15*\ymax}
	\coordinate (D1) at (\xmin,\ymin);
	\coordinate (D2) at (\xmax,\ymax);
	
	\pgfmathsetmacro{\m}{(\ymax-\ymin)/(\xmax-\xmin)}
	\pgfmathsetmacro{\c}{\ymin - \m*(\xmin)}
	
	\def\diaglinepath{(D1) -- (D2)}%
	\def\upperdiagregionpath{(D1) -- (D2) -- (\xmax,\ymax) -- (\xmin,\ymax) -- cycle}%
	
	\begin{scope}
		\clip \upperdiagregionpath;
		\fill[bg] (\xmin,\ymin) rectangle (\xmax,\ymax);
	\end{scope}
	
	\draw[gray] (0,\ymin) -- (0,\ymax);
	\draw[gray] (\xmin,0) -- (\xmax,0);
	\draw[dashed, gray, very thin] \diaglinepath;
	
	\coordinate (P1) at (.4*\xmax,.7*\ymax);
	\coordinate (P2) at (.2*\xmax,.4*\ymax);
	\coordinate (P3) at (.5*\xmax,.2*\ymax);
	\coordinate (P4) at (.7*\xmax,.5*\ymax);
	
	\def\blobpath{%
		plot[smooth cycle, tension=0.9]
		coordinates {(P1) (P2) (P3) (P4)}%
	}
	
	\path[name path=blob] \blobpath;
	\path[name path=diag] \diaglinepath;
	
	\path[name intersections={of=blob and diag, by={I1,I2}}];
	
	\fill[figcolor!05, draw=figcolor] \blobpath;
	
	\begin{scope}
		\clip \upperdiagregionpath;
		
		\fill[accentcolor!20] \blobpath;
		
		\draw[accentcolor, thick] \blobpath;
	\end{scope}
	
	\draw[accentcolor, thick] (I1) -- (I2);
	
	\node at (\xmax*.3,\ymax*.82) {$O\cap R$};
	\node at (.8*\xmax,.3*\ymax) {$O$};
\end{tikzpicture}
\hfil
\begin{tikzpicture}[baseline={(0,0)}]
	\pgfmathsetmacro{\xmin}{-0.15*\xmax}
	\pgfmathsetmacro{\ymin}{-0.15*\ymax}
	\coordinate (D1) at (\xmin,\ymin);
	\coordinate (D2) at (\xmax,\ymax);
	
	\def\diaglinepath{(D1) -- (D2)}%
	\def\Rpath{(D1) -- (D2) -- (\xmax,\ymax) -- (\xmin,\ymax) -- cycle}%
	\def\Boxpath{(\xmin,\ymin) rectangle (\xmax,\ymax)}%
	
	\begin{scope}[even odd rule]
		\clip \Boxpath \Rpath;
		\fill[figcolor!20] \Boxpath;
	\end{scope}
	
	\begin{scope}
		\clip \Rpath;
		\fill[bg] \Boxpath;
	\end{scope}
	
	\draw[gray] (0,\ymin) -- (0,\ymax);
	\draw[gray] (\xmin,0) -- (\xmax,0);
	\draw[dashed, gray, very thin] \diaglinepath;
	
	\coordinate (P1) at (.4*\xmax,.7*\ymax);
	\coordinate (P2) at (.2*\xmax,.4*\ymax);
	\coordinate (P3) at (.5*\xmax,.2*\ymax);
	\coordinate (P4) at (.7*\xmax,.5*\ymax);
	
	\def\blobpath{%
		plot[smooth cycle, tension=0.9]
		coordinates {(P1) (P2) (P3) (P4)}%
	}
	
	\fill[figcolor!25, draw=figcolor, dotted] \blobpath;
	\node at (\xmax*.45,\ymax*.45) {$O$};
	\node at (.8*\xmax,.12*\ymax) {$[O]_R$};
	
	\path[name path=blob] \blobpath;
	\path[name path=diag]  \diaglinepath;
	\path[name intersections={of=blob and diag, by={I1,I2}}];
	
	\path let \p1=(I1), \p2=(I2) in
	\pgfextra{
		\ifdim\x1<\x2
		\global\coordinate (IL) at (I1);
		\global\coordinate (IR) at (I2);
		\else
		\global\coordinate (IL) at (I2);
		\global\coordinate (IR) at (I1);
		\fi
	};
	
	\begin{scope}
		\clip \Rpath;
		\draw[figcolor, thick] \blobpath;
	\end{scope}
	
	\draw[figcolor, thick] (D1) -- (IL);
	\draw[figcolor, thick] (IR) -- (D2);
\end{tikzpicture}
\hfil

\smallskip\smallskip\noindent

The idea is to construct a sublocale $R$ from $\up,\down$ using the techniques from~\cref{section:sublocales}, where the generating relation~$\sim$ is defined via the following important spatial intuition: for any~${A,B,C,D\in\Powerset(S)}$ we have
\[
A\times B\cap R = C\times D\cap R
\quad\iff\quad
\begin{array}{l}
	A\cap \down B = C\cap \down D,\\
	\up A\cap B = \up C\cap D.
\end{array}
\]
Here, the left-hand side describes the behaviour of the frame congruence $\equiv_R$ on the basic opens, which is now completely described by the condition on the right-hand side involving only the cones~$\up,\down$. It says that two rectangles are `the same' iff their projections under the source and target maps are equal. Here is a simple visual example for the target projection:

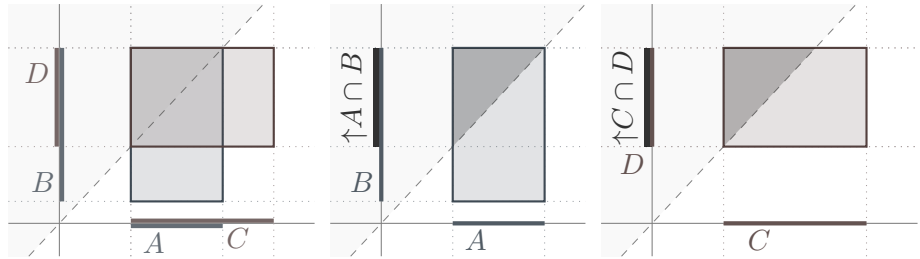
\def\xmax{2.7} 
\def\ymax{2.9} 
\def\a{0.2*\xmax} 
\def\b{0.7*\xmax} 
\pgfmathsetmacro{\rAxL}{0.35*\xmax}
\pgfmathsetmacro{\rAxR}{0.8*\xmax}
\pgfmathsetmacro{\rAyB}{0.10*\ymax}
\pgfmathsetmacro{\rAyT}{0.8*\ymax}
\pgfmathsetmacro{\rBxL}{0.35*\xmax}
\pgfmathsetmacro{\rBxR}{1.05*\xmax}
\pgfmathsetmacro{\rByB}{0.35*\ymax}
\pgfmathsetmacro{\rByT}{0.8*\ymax}
\smallskip\hfil
\begin{tikzpicture}[baseline={(0,0)}]
	\pgfmathsetmacro{\xmin}{-0.15*\xmax}
	\pgfmathsetmacro{\ymin}{-0.15*\ymax}
	\pgfmathsetmacro{\xdrawmax}{1.25*\xmax} 
	\pgfmathsetmacro{\ydrawmax}{\ymax}
	
	\pgfmathsetmacro{\xminbg}{\xmin - 0.1*\xmax} 
	\pgfmathsetmacro{\BIG}{20*max(\xmax,\ymax)}
	
	\coordinate (D1) at (\xmin,\ymin);
	\coordinate (D2) at (\xmax,\ymax);
	
	\pgfmathsetmacro{\m}{(\ymax-\ymin)/(\xmax-\xmin)}
	\pgfmathsetmacro{\c}{\ymin - \m*(\xmin)}
	
	\pgfmathsetmacro{\xL}{\xminbg-\BIG}
	\pgfmathsetmacro{\xR}{\xmax+\BIG}
	\pgfmathsetmacro{\yL}{\m*\xL + \c}
	\pgfmathsetmacro{\yR}{\m*\xR + \c}
	
	\begin{scope}
		\path[use as bounding box] (\xminbg,\ymin) rectangle (\xdrawmax,\ydrawmax);
		
		\clip (\xminbg,\ymin) rectangle (\xmax,\ymax);
		
		\clip (\xL,\yL) -- (\xR,\yR) -- (\xR,\ymax+\BIG) -- (\xL,\ymax+\BIG) -- cycle;
		
		\fill[bg] (\xminbg,\ymin) rectangle (\xmax,\ymax);
	\end{scope}
	
	\draw[gray] (0,\ymin) -- (0,\ydrawmax);
	\draw[gray] (\xminbg,0) -- (\xdrawmax,0); 
	\draw[dashed, gray, very thin] (D1) -- (D2);

	\draw[figcolor!50, dotted]    (\rAxL,\ymin) -- (\rAxL,\ydrawmax);
	\draw[figcolor!50, dotted]    (\rAxR,\ymin) -- (\rAxR,\ydrawmax);
	\draw[accentcolor!50, dotted] (\rBxL,\ymin) -- (\rBxL,\ydrawmax);
	\draw[accentcolor!50, dotted] (\rBxR,\ymin) -- (\rBxR,\ydrawmax);
	
	\draw[figcolor!50, dotted]    (\xminbg,\rAyB) -- (\xdrawmax,\rAyB);
	\draw[figcolor!50, dotted]    (\xminbg,\rAyT) -- (\xdrawmax,\rAyT);
	\draw[accentcolor!50, dotted] (\xminbg,\rByB) -- (\xdrawmax,\rByB);
	\draw[accentcolor!50, dotted] (\xminbg,\rByT) -- (\xdrawmax,\rByT);
	
	\fill[figcolor, opacity=0.15] (\rAxL,\rAyB) rectangle (\rAxR,\rAyT);
	\draw[figcolor, thick]        (\rAxL,\rAyB) rectangle (\rAxR,\rAyT);
	
	\fill[accentcolor, opacity=0.15] (\rBxL,\rByB) rectangle (\rBxR,\rByT);
	\draw[accentcolor, thick]        (\rBxL,\rByB) rectangle (\rBxR,\rByT);
	
	\draw[accentcolor!80, ultra thick,yshift=1pt] (\rBxL,0) -- (\rBxR,0)
	node[below=1pt, near end, inner sep=1pt] {$C$};
	\draw[figcolor!80, ultra thick,yshift=-1pt] (\rAxL,0) -- (\rAxR,0)
	node[below=.7pt, near start, inner sep=1pt] {$A$};
	\draw[figcolor!80, ultra thick,xshift=1pt] (0,\rAyB) -- (0,\rAyT)
	node[left=1pt, very near start, inner sep=1pt] {$B$};
	\draw[accentcolor!80, ultra thick,xshift=-1pt] (0,\rByB) -- (0,\rByT)
	node[left=1pt, near end, inner sep=1pt] {$D$};
\end{tikzpicture}
\hfil
\begin{tikzpicture}[baseline={(0,0)}]
	\pgfmathsetmacro{\xmin}{-0.15*\xmax}
	\pgfmathsetmacro{\ymin}{-0.15*\ymax}
	\pgfmathsetmacro{\xdrawmax}{1*\xmax} 
	\pgfmathsetmacro{\ydrawmax}{\ymax}
	
	\pgfmathsetmacro{\xminbg}{\xmin - 0.1*\xmax} 
	\pgfmathsetmacro{\BIG}{20*max(\xmax,\ymax)}
	
	\coordinate (D1) at (\xmin,\ymin);
	\coordinate (D2) at (\xmax,\ymax);
	
	\pgfmathsetmacro{\m}{(\ymax-\ymin)/(\xmax-\xmin)}
	\pgfmathsetmacro{\c}{\ymin - \m*(\xmin)}
	
	\pgfmathsetmacro{\xL}{\xminbg-\BIG}
	\pgfmathsetmacro{\xR}{\xmax+\BIG}
	\pgfmathsetmacro{\yL}{\m*\xL + \c}
	\pgfmathsetmacro{\yR}{\m*\xR + \c}
	
	\begin{scope}
		\path[use as bounding box] (\xminbg,\ymin) rectangle (\xdrawmax,\ydrawmax);
		\clip (\xminbg,\ymin) rectangle (\xmax,\ymax);
		\clip (\xL,\yL) -- (\xR,\yR) -- (\xR,\ymax+\BIG) -- (\xL,\ymax+\BIG) -- cycle;
		\fill[bg] (\xminbg,\ymin) rectangle (\xmax,\ymax);
	\end{scope}
	
	\draw[gray] (0,\ymin) -- (0,\ydrawmax);
	\draw[gray] (\xminbg,0) -- (\xdrawmax,0);
	\draw[dashed, gray, very thin] (D1) -- (D2);

	\draw[figcolor!50, dotted] (\rAxL,\ymin) -- (\rAxL,\ydrawmax);
	\draw[figcolor!50, dotted] (\rAxR,\ymin) -- (\rAxR,\ydrawmax);
	\draw[figcolor!50, dotted] (\xminbg,\rAyB) -- (\xdrawmax,\rAyB);
	\draw[figcolor!50, dotted] (\xminbg,\rAyT) -- (\xdrawmax,\rAyT);
	\pgfmathsetmacro{\YprojB}{max(\rAyB, \m*\rAxL + \c)}
	\pgfmathsetmacro{\YprojT}{\rAyT}
	
	\draw[figcolor!50, dotted] (\xminbg,\YprojB) -- (\xdrawmax,\YprojB);
	\draw[figcolor!50, dotted] (\xminbg,\YprojT) -- (\xdrawmax,\YprojT);
	
	\fill[figcolor, opacity=0.15] (\rAxL,\rAyB) rectangle (\rAxR,\rAyT);
	\draw[figcolor, thick]        (\rAxL,\rAyB) rectangle (\rAxR,\rAyT);
	
	\begin{scope}
		\clip (\rAxL,\rAyB) rectangle (\rAxR,\rAyT);
		\clip (\xL,\yL) -- (\xR,\yR) -- (\xR,\ymax+\BIG) -- (\xL,\ymax+\BIG) -- cycle;
		\fill[black, opacity=0.2] (\rAxL,\rAyB) rectangle (\rAxR,\rAyT);
	\end{scope}
	
	\pgfmathsetmacro{\ybandw}{.1} 
	\fill[black, opacity=0.8]
	(-\ybandw,\YprojB) rectangle (0,\YprojT)
	node[above, midway, rotate=90] {$\up A\cap B$};
	
	\draw[figcolor!90, ultra thick] (\rAxL,0) -- (\rAxR,0)
	node[below=.7pt, near start, inner sep=1pt] {$A$};
	\draw[figcolor!90, ultra thick] (0,\rAyB) -- (0,\rAyT)
	node[left=1pt, very near start, inner sep=1pt] {$B$};
\end{tikzpicture}
\hfil
\begin{tikzpicture}[baseline={(0,0)}]
	\pgfmathsetmacro{\xmin}{-0.15*\xmax}
	\pgfmathsetmacro{\ymin}{-0.15*\ymax}
	\pgfmathsetmacro{\xdrawmax}{1.3*\xmax} 
	\pgfmathsetmacro{\ydrawmax}{\ymax}
	
	\pgfmathsetmacro{\xminbg}{\xmin - 0.1*\xmax} 
	\pgfmathsetmacro{\BIG}{20*max(\xmax,\ymax)}
	
	\coordinate (D1) at (\xmin,\ymin);
	\coordinate (D2) at (\xmax,\ymax);
	
	\pgfmathsetmacro{\m}{(\ymax-\ymin)/(\xmax-\xmin)}
	\pgfmathsetmacro{\c}{\ymin - \m*(\xmin)}
	
	\pgfmathsetmacro{\xL}{\xminbg-\BIG}
	\pgfmathsetmacro{\xR}{\xmax+\BIG}
	\pgfmathsetmacro{\yL}{\m*\xL + \c}
	\pgfmathsetmacro{\yR}{\m*\xR + \c}
	
	\begin{scope}
		\path[use as bounding box] (\xminbg,\ymin) rectangle (\xdrawmax,\ydrawmax);
		\clip (\xminbg,\ymin) rectangle (\xmax,\ymax);
		\clip (\xL,\yL) -- (\xR,\yR) -- (\xR,\ymax+\BIG) -- (\xL,\ymax+\BIG) -- cycle;
		\fill[bg] (\xminbg,\ymin) rectangle (\xmax,\ymax);
	\end{scope}
	
	\draw[gray] (0,\ymin) -- (0,\ydrawmax);
	\draw[gray] (\xminbg,0) -- (\xdrawmax,0);
	\draw[dashed, gray, very thin] (D1) -- (D2);
	
	\draw[accentcolor!50, dotted] (\rBxL,\ymin) -- (\rBxL,\ydrawmax);
	\draw[accentcolor!50, dotted] (\rBxR,\ymin) -- (\rBxR,\ydrawmax);
	\draw[accentcolor!50, dotted] (\xminbg,\rByB) -- (\xdrawmax,\rByB);
	\draw[accentcolor!50, dotted] (\xminbg,\rByT) -- (\xdrawmax,\rByT);
	
	\pgfmathsetmacro{\YprojB}{max(\rByB, \m*\rBxL + \c)}
	\pgfmathsetmacro{\YprojT}{\rByT}
	
	\draw[accentcolor!50, dotted] (\xminbg,\YprojB) -- (\xdrawmax,\YprojB);
	\draw[accentcolor!50, dotted] (\xminbg,\YprojT) -- (\xdrawmax,\YprojT);
	
	\fill[accentcolor, opacity=0.15] (\rBxL,\rByB) rectangle (\rBxR,\rByT);
	\draw[accentcolor, thick]        (\rBxL,\rByB) rectangle (\rBxR,\rByT);
	
	\begin{scope}
		\clip (\rBxL,\rByB) rectangle (\rBxR,\rByT);
		\clip (\xL,\yL) -- (\xR,\yR) -- (\xR,\ymax+\BIG) -- (\xL,\ymax+\BIG) -- cycle;
		\fill[black, opacity=0.2] (\rBxL,\rByB) rectangle (\rBxR,\rByT);
	\end{scope}
	
	\pgfmathsetmacro{\ybandw}{.1}
	\fill[black, opacity=0.8]
	(-\ybandw,\YprojB) rectangle (0,\YprojT)
	node[above, midway, rotate=90] {$\up C\cap D$};
	
	\draw[accentcolor!90, ultra thick] (\rBxL,0) -- (\rBxR,0)
	node[below=.7pt, near start, inner sep=1pt] {$C$};
	\draw[accentcolor!90, ultra thick] (0,\rByB) -- (0,\rByT)
	node[below left=1pt, at start, inner sep=1pt] {$D$};
\end{tikzpicture}\hfil

\smallskip\smallskip\noindent
Therefore, without knowing the congruence~$\equiv_R$, we can emulate its behaviour on rectangles using only the cones $\up,\down$, and from this we generate a frame congruence, and hence a localic relation. 

\begin{intuition}
	\label{intuition:sim}
	Given a pair of cones $\up,\down\colon \Opens S\to \Opens S$, we construct a sublocale~$R$ as the $\sim$-saturated subsets of $\Opens S\tensor \Opens S$, where~$\sim$ is the relation defined by the spatial intuition
	\[
	A\tensor B \sim C\tensor D
	\quad\iff\quad
	\begin{array}{l}
		A\cap \down B = C\cap \down D,\\
		\up A\cap B = \up C\cap D.
	\end{array}
	\]
\end{intuition}

\subsection{Monotonicity}
To finish this section, we define the morphisms of spaces equipped with relations, which is just the usual notion of monotonicity, and give a suitable translation to point-free language.

\begin{definition}
	\label{definition:monotone function}
	A function $f\colon (X,R)\to (Y,Q)$ between sets equipped with relations is called \emph{monotone} if $x R y$ implies $f(x) Q f(y)$.
	
	We denote the category of topological spaces equipped with relations and continuous monotone functions between them by $\rTop$. The full subcategory where the relations have open cones is denoted $\rocTop$.
\end{definition}

Since the definition of monotonicity is stated in terms of points, we need to translate this to point-free language. The following lemma suffices for this.

\begin{lemma}
	\label{lemma:monotonicity in terms of cones}
	The following conditions for a function $f\colon (X,R)\to (Y,Q)$ between sets equipped with relations are equivalent:
	\begin{itemize}
		\item $f$ is monotone;
		\item $(f\times f)[R]\subseteq Q$;
		\item $f(\up A)\subseteq \up f(A)$ for all $A\in\Powerset(X)$;
		\item $f(\down A)\subseteq \down f(A)$ for all $A\in\Powerset(X)$;
		\item $\up f^{-1}(B)\subseteq f^{-1}(\up B)$ for all $B\in\Powerset(Y)$;
		\item $\down f^{-1}(B)\subseteq f^{-1}(\down B)$ for all $B\in\Powerset(Y)$.
	\end{itemize}
\end{lemma}
\begin{proof}
	This equivalence is relatively standard. Here we sketch only the proof for the preimage inclusions, since those are most relevant.
	
	If $f$ is monotone and $B\in \Powerset(Y)$, for $x\in \down f^{-1}(B)$ we get some $y\in f^{-1}(B)$ with $xRy$, and so $f(x)Qf(y)\in B$, which means $f(x)\in \down B$, and so~${x\in f^{-1}(\down B)}$.
	
	Conversely, if $\down f^{-1}(B)\subseteq f^{-1}(\down B)$ for all $B\in\Powerset(Y)$, when $xRy$ pick the singleton $B=\{f(y)\}$. Then $x\in \down f^{-1}(\{f(y)\})$ and hence $x\in f^{-1}(\down \{f(y)\})$, which just unpacks to~$f(x) Q f(y)$.
\end{proof}

\section{Localic relations}
\label{section:localic relations}
Localic relations are the localic analogue of relations on a set or space. Instead of subsets, they are defined in terms of sublocales. Besides losing the convenient point-wise manipulations on pairs~$xRy$, difficulties can be compounded due to the often technical nature of dealing with sublocales and products. This section gives the basic definitions of localic relations, leading to the category~$\rLoc$ of locales equipped with relations and monotone maps between them. After that, we define the class of localic relations with open cones, allowing us to define the localic cones~$R\mapsto \Up,\Down$. The resulting full subcategory of locales equipped with open cone localic relations is~$\rocLoc$. Throughout, we discuss several examples.

\begin{definition}
	A \emph{localic relation} on $X$ is a sublocale $r\colon R\rightarrowtail X\times X$.
\end{definition}

\begin{remark}
	Of course, a more general notion of localic relations between two different locales is possible, but not necessary for the present work. Some brief remarks are in~\cref{section:relations between different locales}.
\end{remark}

\begin{definition}
	\label{definition:source and target}
	The \emph{source} and \emph{target maps} of a localic relation $r\colon R\rightarrowtail X\times X$ are defined as the maps $s,t\colon R\to X$ given by $s:=\pr_1\circ r$ and $t:=\pr_2\circ r$.
	
	For~$U,V\in \Opens X$ we thus have~(\cref{lemma:pair of locale maps}):
	\begin{gather*}
	r^{-1}(U\tensor V) = s^{-1}(U)\wedge t^{-1}(V),\\
	s^{-1}(U) = r^{-1}(U\tensor \top)
	\qquad\text{and}\qquad
	t^{-1}(V) = r^{-1}(\top\tensor V).
	\end{gather*}
\end{definition}

\def\xmax{2.2} 
\def\ymax{2} 
\hfill
\begin{tikzpicture}[baseline={(0,0)}]
	\def\a{0.2*\xmax} 
	\def\b{0.7*\xmax} 
	
	\pgfmathsetmacro{\xmin}{-0.15*\xmax}
	\pgfmathsetmacro{\ymin}{-0.15*\ymax}
	\coordinate (D1) at (\xmin,\ymin);
	\coordinate (D2) at (\xmax,\ymax);
	
	\begin{scope}
		\clip (D1) -- (D2) -- (\xmax,\ymax) -- (\xmin,\ymax) -- cycle;
		\fill[bg] (\xmin,\ymin) rectangle (\xmax,\ymax);
	\end{scope}
	
	\pgfmathsetmacro{\m}{(\ymax-\ymin)/(\xmax-\xmin)}
	\pgfmathsetmacro{\c}{\ymin - \m*(\xmin)}
	\pgfmathsetmacro{\ya}{\m*\a + \c}
	\pgfmathsetmacro{\yb}{\m*\b + \c}
	
	\draw[gray] (0,-0.15*\ymax) -- (0,\ymax);
	\draw[gray] (-0.15*\xmax,0) -- (\xmax,0);
	\draw[dashed, gray, very thin] (D1) -- (D2);
	\draw[\rectcol!50, dotted] (\a,-.15*\ymax) -- (\a,\ymax);
	\draw[\rectcol!50, dotted] (\b,-.15*\ymax) -- (\b,\ymax);
	
	\draw[\rectcol, ultra thick] (\a,0) -- (\b,0)
	node[below=1pt, midway,fill=white, inner sep= 1pt] {$U$};
	
	\begin{scope}
		\clip (D1) -- (D2) -- (\xmax,\ymax) -- (\xmin,\ymax) -- cycle;
		\fill[\rectcol!20] (\a,0) rectangle (\b,\ymax);
	\end{scope}
	
	\draw[\rectcol, thick] (\a,\ya) -- (\a,\ymax);
	\draw[\rectcol, thick] (\b,\yb) -- (\b,\ymax);
	\draw[\rectcol, thick] (\a,\ya) -- (\b,\yb);
	
	\node[\rectcol, textbg] at (\xmax*.8,\ymax*.4) {$s^{-1}(U)$};
\end{tikzpicture}
\hfill
\begin{tikzpicture}[baseline={(0,0)}]
	\def\a{0.25*\ymax} 
	\def\b{0.8*\ymax} 
	
	\pgfmathsetmacro{\xmin}{-0.15*\xmax}
	\pgfmathsetmacro{\ymin}{-0.15*\ymax}
	\coordinate (D1) at (\xmin,\ymin);
	\coordinate (D2) at (\xmax,\ymax);
	\begin{scope}
		\clip (D1) -- (D2) -- (\xmax,\ymax) -- (\xmin,\ymax) -- cycle;
		\fill[bg] (\xmin,\ymin) rectangle (\xmax,\ymax);
	\end{scope}
	
	\pgfmathsetmacro{\m}{(\ymax-\ymin)/(\xmax-\xmin)}
	\pgfmathsetmacro{\c}{\ymin - \m*(\xmin)}
	
	\pgfmathsetmacro{\xa}{(\a-\c)/\m}
	\pgfmathsetmacro{\xb}{(\b-\c)/\m}
	
	\draw[gray] (0,\ymin) -- (0,\ymax);
	\draw[gray] (\xmin,0) -- (\xmax,0);
	\draw[dashed, gray, very thin] (D1) -- (D2);
	
	\draw[\rectcol!50, dotted] (\xmin,\a) -- (\xmax,\a);
	\draw[\rectcol!50, dotted] (\xmin,\b) -- (\xmax,\b);
	
	\begin{scope}
		\clip (D1) -- (D2) -- (\xmax,\ymax) -- (\xmin,\ymax) -- cycle;
		\fill[\rectcol!20] (\xmin,\a) rectangle (\xmax,\b);
	\end{scope}
	
	\draw[\rectcol, thick] (\xmin,\a) -- (\xa,\a);
	\draw[\rectcol, thick] (\xmin,\b) -- (\xb,\b);
	\draw[\rectcol, thick] (\xa,\a) -- (\xb,\b);
	
	\node[\rectcol,textbg] at (\xmax*.85,\ymax*.45) {$t^{-1}(V)$};
	
	\draw[\rectcol, ultra thick] (0,\a) -- (0,\b)
	node[midway, right=-0.05] {$V$};
\end{tikzpicture}
\hfill
\def\xmax{2.5} 
\def\ymax{2} 
\begin{tikzpicture}[baseline={(0,0)}]
	\def\aU{0.25*\xmax}  
	\def\bU{0.7*\xmax}  
	\def\aV{0.25*\ymax} 
	\def\bV{0.8*\ymax}  
	
	\pgfmathsetmacro{\xmin}{-0.15*\xmax}
	\pgfmathsetmacro{\ymin}{-0.15*\ymax}
	\coordinate (D1) at (\xmin,\ymin);
	\coordinate (D2) at (\xmax,\ymax);
	\begin{scope}
		\clip (D1) -- (D2) -- (\xmax,\ymax) -- (\xmin,\ymax) -- cycle;
		\fill[bg] (\xmin,\ymin) rectangle (\xmax,\ymax);
	\end{scope}
	
	\pgfmathsetmacro{\m}{(\ymax-\ymin)/(\xmax-\xmin)}
	\pgfmathsetmacro{\c}{\ymin - \m*(\xmin)}
	
	\pgfmathsetmacro{\yUa}{\m*\aU + \c}
	\pgfmathsetmacro{\yUb}{\m*\bU + \c}
	
	\pgfmathsetmacro{\xVa}{(\aV-\c)/\m}
	\pgfmathsetmacro{\xVb}{(\bV-\c)/\m}
	
	\draw[gray] (0,\ymin) -- (0,\ymax);
	\draw[gray] (\xmin,0) -- (\xmax,0);
	\draw[dashed, gray, very thin] (D1) -- (D2);
	
	\draw[\rectcol!50, dotted] (\aU,\ymin) -- (\aU,\ymax);
	\draw[\rectcol!50, dotted] (\bU,\ymin) -- (\bU,\ymax);
	\draw[\rectcol!50, dotted] (\xmin,\aV) -- (\xmax,\aV);
	\draw[\rectcol!50, dotted] (\xmin,\bV) -- (\xmax,\bV);
	
	\draw[\rectcol, ultra thick] (\aU,0) -- (\bU,0)
	node[below=1pt, midway, fill=white, inner sep=1pt] {$U$};
	\draw[\rectcol, ultra thick] (0,\aV) -- (0,\bV)
	node[midway, right=-.05] {$V$};
	
	
	\begin{scope}
		\clip (D1) -- (D2) -- (\xmax,\ymax) -- (\xmin,\ymax) -- cycle;
		\clip (\aU,\aV) rectangle (\bU,\bV);
		\fill[\rectcol!20] (\aU,\aV) rectangle (\bU,\bV);
	\end{scope}
	
	\begin{scope}
		\clip (D1) -- (D2) -- (\xmax,\ymax) -- (\xmin,\ymax) -- cycle;
		
		\draw[\rectcol, thick] (\aU,\aV) -- (\aU,\bV); 
		\draw[\rectcol, thick] (\aU,\bV) -- (\bU,\bV); 
		\draw[\rectcol, thick] (\bU,\aV) -- (\bU,\bV); 
		\draw[\rectcol, thick] (\aU,\aV) -- (\bU,\aV); 
	\end{scope}
	
	\begin{scope}
		\clip (\aU,\aV) rectangle (\bU,\bV);
		\draw[\rectcol, thick] (\aU,\yUa) -- (\bU,\yUb);
	\end{scope}
	
	\node[\rectcol, textbg]
	at (\xmax*1,\ymax*.42) {$r^{-1}(U\tensor V)$};
\end{tikzpicture}\hfill

\smallskip\smallskip\smallskip
We discuss some basic examples of localic relations, with more in~\cref{section:localic relations with open cones}.

\begin{example}
	\label{example:diagonal relation}
	The \emph{diagonal sublocale} is described in~\cite[\S IV.5.3]{picado2012FramesLocalesTopology}. Denote by ${\delta\colon X\rightarrowtail X\times X}$ its inclusion, the unique map such that $\pr_i\circ \delta=\id_X$. Then the underlying frame map is given on rectangles by~(\cref{lemma:coproduct of frame maps}):
	\[
	\delta^{-1}(U\tensor V) = U\wedge V.
	\]
\end{example}

\begin{example}
\label{example:kernel pair equivalence relation}
	For any locale map $q\colon X\to Q$, its \emph{kernel pair} is the localic equivalence relation ${R:=X\times_Q X\rightarrowtail X\times X}$ defined via the pullback of~$q$ along itself (see~\cref{definition:localic relation properties} on how to define equivalence relations in~$\Loc$). Its source and target maps~$s,t$ are the universal ones such that $q\circ s= q\circ t$. Equivalently, $R$ is the equaliser of~$q\circ \pr_1$ and~$q\circ \pr_2$, from which it follows that it is indeed a sublocale of~$X\times X$. The spatial intuition is that~$xRy$ iff~$q(x)=q(y)$.
	
	Any localic equivalence relation $R$ induces the coequaliser $q\colon X\tworightarrow X/R$ of its source and target maps, but~$R$ is in general not recovered as the kernel pair of this~$q$. There is only a sublocale inclusion~$R\rightarrowtail X\times_{X/R} X$. A localic relation~$R$ is called \emph{effective} when this is an isomorphism. We discuss this further in~\cref{example:open kernel pairs}.
\end{example}

\subsection{Monotonicity}
We define the category~$\rLoc$ of locales equipped with localic relations, and monotone maps between them. Monotonicity here is the suitable internalised version of~\cref{definition:monotone function}.

\begin{definition}
	A \emph{monotone map} $f\colon (X,R)\to (Y,Q)$ between locales equipped with localic relations is a locale map~${f\colon X\to Y}$ such that:
	\[
	\begin{tikzcd}[cramped]
		R & Q \\
		{X\times X} & {Y\times Y.}
		\arrow["{\exists \bar f}"', dashed, from=1-1, to=1-2]
		\arrow[tail, from=1-1, to=2-1]
		\arrow[tail, from=1-2, to=2-2]
		\arrow["{f\times f}"', from=2-1, to=2-2]
	\end{tikzcd}
	\]
	Note that such a locale map $\bar f$, if it exists, is necessarily unique. Denote the category of locales equipped with relations together with monotone maps by $\rLoc$.
\end{definition}

\begin{remark}
	\label{remark:monotone inclusion}
	Using the canonical (epimorphism, sublocale inclusion) factorisation system of~$\Loc$, it is straightforward to see that $f$ is monotone iff $R\subseteq (f\times f)^{-1}[Q]$, or equivalently~$(f\times f)[R]\subseteq Q$. Hence the internal definition aligns with the point-set intuition in~\cref{lemma:monotonicity in terms of cones}. From the latter inclusion it is also clear that $R\subseteq Q$ precisely when the identity map $\id_X\colon (X,R)\to (X,Q)$ is internally monotone. On a similar note, the following lemma will be useful.
\end{remark}

\begin{lemma}\label{lemma:monotone map modification}
	If $f\colon (X,R)\to (Y,Q)$ is monotone, and there are subobject inclusions $P\subseteq R$ and $Q\subseteq T$, then $f\colon (X,P)\to (Y,T)$ is also monotone.
\end{lemma}
\begin{proof}
	This is evident from the following diagram:
	\[
	\begin{tikzcd}[cramped]
		P & R & Q & T \\
		& {X\times X} & {Y\times Y.}
		\arrow[dashed, tail, from=1-3, to=1-4]
		\arrow["t", tail, from=1-4, to=2-3]
		\arrow["q"', tail, from=1-3, to=2-3]
		\arrow["{\bar f}"', dashed, from=1-2, to=1-3]
		\arrow["r"', tail, from=1-2, to=2-2]
		\arrow["{f\times f}"', from=2-2, to=2-3]
		\arrow["p"', tail, from=1-1, to=2-2]
		\arrow[dashed, tail, from=1-1, to=1-2]
	\end{tikzcd}
	\qedhere
	\]
\end{proof}

Recall from~\cref{definition:monotone function} the category $\rTop$. For a space~$(S,R)$ equipped with a relation, we endow $R\subseteq S\times S$ with the subspace topology. Thus it is a topological space in its own right, and induces a locale $\loc(R)$. We get a sublocale inclusion $\loc(r)\colon \loc(R)\rightarrowtail \loc(S\times S)$. To make it a localic relation on $\loc(S)$, we in turn compose this with the canonical dense inclusion ${\pi_S\colon \loc(S\times S)\rightarrowtail \loc(S)\times \loc(S)}$. See~\cite[\S IV.5.4]{picado2012FramesLocalesTopology} for more details on~$\pi_S$. It is then straightforward to see using general categorical arguments (\cite[\S D]{schaaf2024TowardsPointFreeSpacetimes}) that we get the following.

\begin{proposition}
	\label{proposition:functor rTop to rLoc}
	The functor $\loc\colon \Top\to \Loc$ extends to a functor
	\begin{align*}
		\loc\colon\rTop &\longrightarrow \rLoc;
		\\
		(S,R)&\longmapsto (\loc(S),\loc(R));
		\\
		g&\longmapsto \loc(g).
	\end{align*}
\end{proposition}

\subsection{Localic relations with open cones}
\label{section:localic relations with open cones}
In this section we define the class of localic relations with open cones, and show how to define the assignment~$R\mapsto \Up,\Down$ that will form the object part of the functor from open cone localic relations to conic frames.

To state the definition of open cones in localic language, we recall from~\cref{lemma:open cones iff s t open} that a relation on a space has open cones precisely when its source and target maps are open continuous functions, which has a direct localic analogue. Open maps of locales were first defined and studied in~\cite[\S V]{joyal1984ExtensionGaloisTheory}, recalled now.

\begin{definition}
	\label{definition:open map}
	A morphism of locales $f\colon X\to Y$ is called \emph{open} if there exists a left adjoint~${f_!\dashv f^{-1}}$ satisfying \emph{Frobenius reciprocity}:
	\[
	f_!\left(U\wedge f^{-1}(V)\right)=f_!(U)\wedge V.
	\]
\end{definition}
\begin{remark}
	Checking in with spatial intuition, it can be shown that a locale map is open precisely when its image preserves open sublocales. Or in turn, that the underlying frame map is a morphism of complete Heyting algebras~\cite[Proposition~V.1.1]{joyal1984ExtensionGaloisTheory}. For us, the above definition is most useful. For open $f\colon X\to Y$ it gives a well-behaved map $f_!\colon\Opens X\to \Opens Y$ serving as an image map of~$f$.
\end{remark}

\begin{definition}
	\label{definition:open cones}
	We say a relation $R$ on a locale $X$ has \emph{open cones} if the source and target maps $s,t\colon R\to X$ are open as locale maps. They are also called \emph{open cone relations}.
	
	Denote by $\rocLoc$ the full subcategory of $\rLoc$ consisting of locales equipped with an open cone relation.
\end{definition}

Explicitly, that $R$ has open cones means we get left adjoints~$s_!\dashv s^{-1}$ and~${t_!\dashv t^{-1}}$ that satisfy Frobenius reciprocity. Following the spatial~\cref{intuition:cones from st} that the cones may be expressed as ${\up U = ts^{-1}(U)}$ and $\down U= st^{-1}(U)$, openness allows us to write down a localic analogue.

\begin{definition}
	\label{definition:induced cones}
	Let $X$ be a locale. If $R$ is a relation on $X$ with open cones, we define its \emph{(induced) cones} as follows:
	\[
	\Up := t_!s^{-1}\colon \Opens X\longrightarrow\Opens X
	\quad\text{and}\quad
	\Down:= s_!t^{-1}\colon \Opens X\longrightarrow\Opens X.
	\]
	Denote the \emph{(induced) reduced cones} by
	\[
	\UpR,\DownR\colon \Opens X\times \Opens X\longrightarrow \Opens X;
	\quad
	\UpR(U,V):= \Up U\wedge V
	\text{ and  }
	\DownR(U,V):= U\wedge \Down V.
	\]
\end{definition}

\begin{remark}
	It is possible to define cones much more generally in the setting of relations internal to a category (with finite limits) equipped with an orthogonal factorisation system~$(\calE,\calM)$~\cite[\S D.5.1]{schaaf2024TowardsPointFreeSpacetimes}. In that case, the cones are morphisms on the $\calM$-subobjects. A localic relation $R$ on $X$ thus induces a more general notion of cone $\up,\down\colon \Sl(X)\to \Sl(X)$ on the coframe of sublocales. These can either be defined via relational composition, with $\up(-) = - \circ R$ and~$\down(-) = R\circ -$, or equivalently in terms of internal images and preimages as~$\up(-) = t[s^{-1}[-]]$ and~$\down(-) = s[t^{-1}[-]]$. From this it can be shown that if $R$ has open cones, then the cones $\Up,\Down$ from \cref{definition:induced cones} are just the restriction of $\up,\down$ to open sublocales. Whether a general localic relation can be recovered from its cones on sublocales is unknown. See however the remarks in~\cref{section:internal cones}.
\end{remark}

Continuing on from the basic examples of localic relations at the start of this section, we discuss some examples of relations with open cones.

\begin{example}
	\label{example:induced cones of diagonal}
	Recall the diagonal relation on $X$ from~\cref{example:diagonal relation}. It trivially has open cones since its source and target maps are $\id_X$, and the induced cones are just~${\Up = \Down = \id_{\Opens X}}$.
\end{example}

\begin{example}
	\label{example:space with oc gives locale with oc}
	Recall that any space $S$ equipped with a set-theoretic relation $R$ that has open cones (\cref{definition:open cones on space}) induces open source and target functions ${s,t\colon R\to S}$. Via the proof of~\cite[Proposition IX.5]{maclane1994SheavesGeometryLogic} we can see that the functor ${\loc\colon \Top\to \Loc}$ preserves open maps, and thus the induced source and target maps ${\loc(s),\loc(t)\colon \loc(R)\to \loc(S)}$ are again open. The induced cones then coincide with the spatial intuition from~\cref{section:spatial intuition}:
	\[
	\loc(s)_!\loc(t)^{-1}(U)= \down U,
	\quad\text{and}\quad
	\loc(t)_!\loc(s)^{-1}(V) = \up V. 
	\]
	Hence the functor ${\loc\colon \rTop\to \rLoc}$ from~\cref{proposition:functor rTop to rLoc} restricts to a functor ${\loc\colon \rocTop\to \rocLoc}$, and the spatial open cone relations from~\cref{example:open cones spaces} can be considered as localic examples as well.
\end{example}

\begin{example}
	\label{example:open sublocale in overt X has open cones}
	Recall that a locale $X$ is called \emph{overt} if the terminal map $X\to 1$ is open~\cite[\S V.3]{joyal1984ExtensionGaloisTheory}. For example, any spatial locale is overt. In classical logic, every locale is overt.
	
	In that case, any open sublocale $R$ of $X\times X$ has open cones. Namely, $R$ is open iff $r$ is an open locale map. That $X$ is overt implies the product projections $\pr_i\colon X\times X\to X$ are open, so $s=\pr_1\circ r$ and $t=\pr_2\circ r$ are open. In particular, if $X$ is overt then the top relation $X\times X$ has open cones.
	
	Conversely, if $X$ has a point $x\colon 1\to X$ and the relation $X\times X$ has open cones, then the pullback of $\pr_i$ along~$x$ gives the terminal map $X\to 1$, which must then be open~(see~\cref{lemma:open map pullback square} below), and so~$X$ is overt.
\end{example}

\begin{example}
	A \emph{uniformity}~$\calE$ on a locale~$X$ consists of a family of open relations $E\in\Opens (X\times X)$, called \emph{entourages}, that model an approximate equality relation on~$X$. Uniformities on overt locales are studied in~\cite{manuell2024UniformLocales}. In that case, by~\cref{example:open sublocale in overt X has open cones} all entourages have open cones. In particular the uniformly below relation can be rewritten as $U\triangleleft V$ iff $\exists E\in\calE: \Upsub{E}U\sqleq V$.
\end{example}

\begin{example}
	\label{example:open kernel pairs}
	We saw in~\cref{example:kernel pair equivalence relation} that any locale map $q\colon X\to Q$ induces a localic (equivalence) relation $X\times_Q X\rightarrowtail X\times X$ called its kernel pair. Suppose now that~$q$ is open. Using the fact that open maps of locales are preserved under pullback (recorded as \cref{lemma:open map pullback square} below), we see that the source and target maps of the kernel pair must be open, and hence $X\times_Q X$ defines an open cone localic relation. These are studied by Kock in~\cite{kock1989GodementTheoremLocales}, there called `$\alpha$-open' equivalence relations, which are just our localic equivalence relations with open cones. Equivalence relations on locales can be defined according to~\cref{definition:localic relation properties}.
	
	Starting with a general open cone localic equivalence relation~$R$, its quotient map~$q\colon X \tworightarrow X/R$ is an open epimorphism, and we get~${R\rightarrowtail X\times_{X/R} X}$ into its induced kernel pair. This is not always an isomorphism, but it is shown in~\cite[Corollary~2.3]{kock1989GodementTheoremLocales} that it is always strongly dense~(\cref{definition:strongly dense}). Kock then derives a Godement-type theorem, stating that any (weakly) closed localic equivalence relation with open cones is effective: $R\cong X\times_{X/R}X$. We show in~\cref{section:strong density} that Kock's Godement theorem follows from the more general strongly dense inclusion~$R\subseteq R_\UpDown$ of an open cone localic relation into the relation induced by its cones~(\cref{proposition:R in Rupdown}).
\end{example}

\begin{example}
	Kock and Moerdijk use sheaves of open cone localic equivalence relations to characterise \emph{\'etendues}, toposes that look locally like the topos of sheaves on a locale~\cite{kock1992EtendueEquivalenceRelation}. Thus localic equivalence relations provide presentations of \'etendues as
	quotient toposes.
\end{example}

\section{Parallelness}
\label{section:parallelness}\label{section:parallel cones}
As we have seen, relations on spaces and their properties can be studied in terms of their cones. But what classifies the cones that come from relations? The notion of \emph{parallel cones} was introduced in the setting of ordered locales in~\cite[\S 4.5]{schaaf2024TowardsPointFreeSpacetimes} and further used in~\cite{heunen2025CausalCoverageOrdered}, but appears also much earlier in the literature, see~\cref{remark:parallelness in the literature} for a small overview. It captures the idea that a pair of cones really comes from one relation, which can be visualised by saying the up and down cones must run `in parallel'. 

For this section we fix a meet-semilattice $(L,\wedge,\top)$. We generically refer to a pair of functions $\uc,\dc\colon L\to L$ as \emph{cones}. For now, we do not even assume that they are monotone. The cones are thought of as localic upward and downward closure operators, or \emph{future} and \emph{past} operators, respectively. For~$x\in L$ we write~${\uc(x)=\uc x}$ and~${\dc (x)= \dc x}$.

\begin{definition}
	\label{definition:parallel cones}
	A pair of cones $\uc,\dc \colon L\to L$ on a meet-semilattice is called \emph{parallel} if for all $x,y\in L$:
	\[
		\uc x\wedge y \sqleq \uc (x\wedge \dc y)
		\qquad\text{and}\qquad
		x \wedge \dc y \sqleq \dc (\uc x\wedge y).
	\]
	
	\hfil
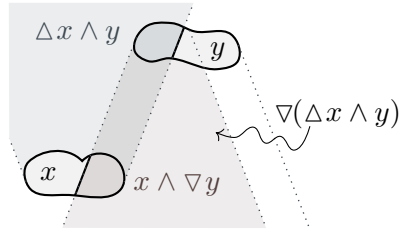
	\begin{tikzpicture}[y=.75pt, x=.8pt, inner sep=0pt, outer sep=0pt]
		\def\topdrop{-8}
		\def\xregionpath{%
			(21.6, 25.5).. controls (30.1, 25.0) and (41.1, 21.8) .. (46.0, 28.4)..
			controls (49.2, 32.6) and (47.4, 40.7) .. (43.2, 43.9)..
			controls (39.1, 47.0) and (32.1, 47.6) .. (27.8, 42.5)..
			controls (20.8, 50.4) and (8.2, 48.2) .. (2.5, 42.0)..
			controls (-0.6, 38.6) and (0.5, 31.7) .. (3.5, 28.4)..
			controls (7.4, 24.1) and (14.8, 25.7) .. (20.5, 25.5)..
			controls (20.9, 25.5) and (21.3, 25.5) .. (21.6, 25.5) -- cycle%
		}
		\def\yregionpath{%
			(78.2, 116.4).. controls (75.5, 117.6) and (73.0, 119.4) .. (70.0, 119.7)..
			controls (65.8, 120.1) and (60.8, 120.4) .. (57.5, 117.8)..
			controls (54.1, 115.2) and (51.4, 110.2) .. (52.4, 106.0)..
			controls (53.1, 103.3) and (56.2, 101.4) .. (59.0, 100.6)..
			controls (62.1, 99.7) and (65.4, 102.2) .. (68.7, 102.1)..
			controls (70.1, 102.0) and (71.4, 101.8) .. (72.8, 101.4)..
			controls (79.0, 99.9) and (84.9, 96.1) .. (91.2, 96.6)..
			controls (94.7, 96.9) and (99.5, 97.6) .. (101.1, 100.8)..
			controls (103.1, 104.8) and (101.3, 111.1) .. (97.9, 114.0)..
			controls (93.6, 117.7) and (86.4, 114.2) .. (80.9, 115.5)..
			controls (80.0, 115.7) and (79.1, 116.0) .. (78.2, 116.4) -- cycle%
		}
		\def\ucconepath{%
			(1.4, 32.1) -- (-6.5, 53.9) -- (-6.7, 121.9) --
			(80.2, 121.9) -- (47.5, 31.9) -- cycle%
		}
		\def\dcconepath{%
			(53.3, 112.6) -- (78.2, 116.2) --
			(114.3, 16.9) -- (18.5, 17.0) -- cycle%
		}
		\def\uccutpath{(47.5, 31.9) -- (80.2, 121.9)}
		\def\dccutpath{(53.3, 112.6) -- (18.5, 17.0)}
		
		\path[name path=xregion] \xregionpath;
		\path[name path=yregion, shift={(0,\topdrop)}] \yregionpath;
		\path[name path=uccut] \uccutpath;
		\path[name path=dccut, shift={(0,\topdrop)}] \dccutpath;
		\path[name intersections={of=yregion and uccut, by={uy1,uy2}}];
		\path[name intersections={of=xregion and dccut, by={dx1,dx2}}];
		
		\begin{scope}[blend group=multiply]
			\path[fill=figcolor!10,line cap=butt,line join=miter,line width=1.0pt,miter limit=4.0] \ucconepath;
			\path[fill=accentcolor!10,line cap=butt,line join=miter,line width=1.0pt,miter limit=4.0,shift={(0,\topdrop)}] \dcconepath;
		\end{scope}
		
		\path[fill=black!05] \xregionpath;
		\path[fill=black!05,shift={(0,\topdrop)}] \yregionpath;
		\begin{scope}
			\path[clip,shift={(0,\topdrop)}] \dcconepath;
			\fill[accentcolor!20] \xregionpath;
		\end{scope}
		\begin{scope}
			\clip \ucconepath;
			\path[fill=figcolor!20,shift={(0,\topdrop)}] \yregionpath;
		\end{scope}
		\path[thick,draw=black] \xregionpath;
		\path[thick,draw=black,shift={(0,\topdrop)}] \yregionpath;
		\draw[thick,draw=black] (dx1) -- (dx2);
		\draw[thick,draw=black] (uy1) -- (uy2);
		
		\path[draw=figcolor,line cap=butt,line join=miter,line width=0.5pt,miter limit=4.0,dotted,shift={(0,\topdrop)}]
		\dccutpath;
		
		\path[draw=figcolor,line cap=butt,line join=miter,line width=0.5pt,miter limit=4.0,dotted,shift={(0,\topdrop)}]
		(101.3, 108.7) -- (134.6, 17.0);
		
		\path[draw=figcolor,line cap=butt,line join=miter,line width=0.5pt,miter limit=4.0,dotted,shift={(0,\topdrop)}]
		(78.2, 116.2) -- (114.3, 17.0);
		
		\path[draw=figcolor,line cap=butt,line join=miter,line width=0.5pt,miter limit=4.0,dotted]
		\uccutpath;
		
		\path[->,draw=black,line cap=butt,line join=miter,shift={(0,\topdrop)}]
		(134.9, 68.0).. controls (134.9, 68.0) and (132.2, 56.6) .. (127.6, 56.0).. 
		controls (124.4, 55.7) and (124.1, 62.3) .. (121.0, 62.9).. 
		controls (118.5, 63.3) and (116.6, 58.9) .. (114.1, 59.5).. 
		controls (110.9, 60.4) and (111.6, 66.8) .. (108.4, 67.5).. 
		controls (105.2, 68.1) and (103.2, 62.8) .. (99.9, 62.4).. 
		controls (96.7, 61.9) and (90.6, 64.9) .. (90.6, 64.9);
		
		\path[draw=figcolor,line cap=butt,line join=miter,line width=0.5pt,miter limit=4.0,dotted]
		(1.4, 32.1) -- (-6.5, 53.9);
		
		\node[anchor=south west,shift={(0,\topdrop)}] at (88, 101) {$y$};
		\node[anchor=south west] at (8, 33) {$x$};
		\node[text=figcolor,anchor=south west] at (5, 101) {$\uc x \wedge y$};
		\node[text=accentcolor,anchor=south west] at (52, 25) {$x\wedge \dc y$};
		\node[anchor=south west,shift={(0,\topdrop)}] at (118, 67) {$\dc (\uc x \wedge y)$};
	\end{tikzpicture}
	\hfil
\end{definition}

\begin{example}
	\label{example:cones from relation are parallel}
	If $R$ is a relation on a set $X$, then the cones $\up,\down$ form a join-preserving parallel pair on the powerset $\Powerset(X)$. Indeed, we already saw at the start of~\cref{section:spatial intuition} that have an even stronger correspondence:
	\[
	x\in \down \{y\}
	\quad\iff\quad
	x R y
	\quad\iff\quad
	y\in \up\{x\}.
	\]
	In fact, it can be shown using singletons that join-preserving cones $\uc,\dc$ on a powerset are parallel iff $x\in\dc\{y\}\iff y\in\uc\{x\}$, and so in turn relations on $X$ are in bijective correspondence with join-preserving parallel cones on $\Powerset(X)$.
\end{example}

\begin{remark}
	\label{remark:parallel implies in future iff in past}
	Suppose that the meet-semilattice also has a bottom element $\bot$. If~$\uc,\dc$ are parallel cones that preserve $\bot$, then it can be shown:
	\[
	x\wedge \dc y = \bot
	\qquad\iff\qquad
	\uc x\wedge y = \bot.
	\]
	The left figure depicts the visual intuition of this condition, while the right depicts its failure:
	
	\hfil
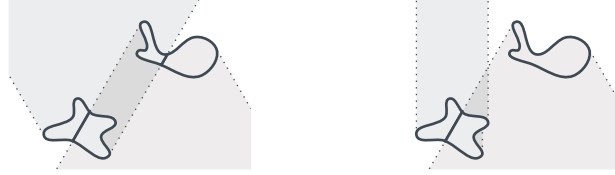
	\begin{tikzpicture}[y=.9pt, x=.9pt, inner sep=0pt, outer sep=0pt]
		\begin{scope}[blend group = multiply]
			\path[fill=figcolor!10,line cap=butt,line join=miter,line width=1.0pt,miter limit=4.0]
			(21.742, 36.868) -- (7.013, 62.38) -- (7.013, 93.686) -- (86.518, 93.561) --
			(48.876, 28.364) -- (34.703, 39.703) -- cycle;
			
			\path[fill=accentcolor!10,line cap=butt,line join=miter,line width=1.0pt,miter limit=4.0]
			(26.987, 22.695) -- (62.147, 83.594) -- (68.719, 68.049) --
			(92.819, 74.894) -- (108.404, 48.206) -- (108.404, 22.695) -- cycle;
		\end{scope}
		
		\path[draw=figcolor,fill=accentcolor!10,line width=1.0pt]
		(73.393, 70.829) -- (70.83, 66.387)..
		controls (76.175, 64.11) and (81.996, 59.781) .. (87.305, 61.85)..
		controls (91.269, 63.395) and (95.181, 68.615) .. (94.017, 72.707)..
		controls (92.977, 76.361) and (87.711, 79.017) .. (83.82, 77.796)..
		controls (79.968, 76.588) and (77.034, 71.5) .. (73.393, 70.829) -- cycle;
		
		\path[draw=figcolor,fill=figcolor!10,line width=1.0pt]
		(33.829, 34.545) -- (40.89, 46.775)..
		controls (39.164, 48.672) and (38.152, 52.263) .. (36.129, 52.369)..
		controls (35.92, 52.381) and (35.7, 52.354) .. (35.467, 52.285)..
		controls (32.949, 51.536) and (35.338, 46.725) .. (33.826, 44.577)..
		controls (31.1, 40.703) and (20.905, 41.763) .. (21.741, 37.101)..
		controls (22.427, 33.279) and (29.445, 35.922) .. (33.163, 34.797)..
		controls (33.383, 34.731) and (33.606, 34.634) .. (33.829, 34.545) -- cycle;
		
		\path[draw=figcolor,line cap=butt,line join=miter,line width=0.5pt,miter limit=4.0,dash pattern=on 0.5pt off 2.0pt]
		(62.147, 83.594) -- (26.987, 22.695);
		
		\path[draw=figcolor,line cap=butt,line join=miter,line width=0.5pt,miter limit=4.0,dash pattern=on 0.5pt off 2.0pt]
		(92.819, 74.894) -- (108.226, 48.206);
		
		\path[draw=figcolor,line cap=butt,line join=miter,line width=0.5pt,miter limit=4.0,dash pattern=on 0.5pt off 2.0pt]
		(21.742, 36.868) -- (7.013, 62.38);
		
		\path[draw=figcolor,line cap=butt,line join=miter,line width=0.5pt,miter limit=4.0,dash pattern=on 0.5pt off 2.0pt]
		(48.876, 28.364) -- (86.518, 93.561);
		
		\path[draw=figcolor,fill=figcolor!10,line width=1.0pt]
		(33.829, 34.545).. controls (39.086, 32.45) and (44.988, 24.573) .. (48.876, 28.364)..
		controls (50.992, 30.427) and (45.095, 33.766) .. (45.727, 36.653)..
		controls (46.324, 39.386) and (52.32, 39.807) .. (51.71, 42.537)..
		controls (50.944, 45.972) and (44.729, 44.164) .. (41.748, 46.034)..
		controls (41.44, 46.227) and (41.157, 46.482) .. (40.89, 46.775) -- cycle;
		
		\path[draw=figcolor,fill=figcolor!10,line width=1.0pt]
		(73.393, 70.829).. controls (72.801, 70.72) and (72.19, 70.723) .. (71.553, 70.884)..
		controls (66.408, 72.181) and (68.042, 86.14) .. (63.05, 84.341)..
		controls (59.399, 83.025) and (66.613, 76.074) .. (64.538, 72.794)..
		controls (63.816, 71.652) and (61.512, 73.676) .. (60.683, 71.54)..
		controls (59.7, 69.01) and (66.844, 67.97) .. (70.146, 66.672)..
		controls (70.372, 66.583) and (70.601, 66.484) .. (70.83, 66.387) -- cycle;
	\end{tikzpicture}
	\hfil
	\begin{tikzpicture}[y=.9pt, x=.9pt, inner sep=0pt, outer sep=0pt]
		\begin{scope}[blend group = multiply]
			\path[fill=accentcolor!10,line cap=butt,line join=miter,line width=1.0pt,miter limit=4.0]
			(26.987, 22.695) -- (62.147, 83.594) -- (68.719, 68.049) --
			(92.819, 74.894) -- (108.404, 48.206) -- (108.404, 22.695) -- cycle;
			
			\path[fill=figcolor!10,line cap=butt,line join=miter,line width=1.0pt,miter limit=4.0]
			(21.742, 37.101) -- (35.134, 37.101) -- (49.337, 29.375) --
			(49.337, 42.537) -- (51.711, 42.537) -- (51.711, 93.561) -- (21.742, 93.561) -- cycle;
		\end{scope}
		
		\path[draw=figcolor,line cap=butt,line join=miter,line width=0.5pt,miter limit=4.0,dash pattern=on 0.5pt off 2.0pt]
		(21.742, 37.101) -- (21.742, 93.561);
		
		\path[draw=figcolor,fill=accentcolor!10,line cap=butt,line join=miter,line width=1.0pt,miter limit=4.0]
		(64.538, 72.794).. controls (63.816, 71.652) and (61.512, 73.676) .. (60.683, 71.54)..
		controls (59.701, 69.01) and (66.844, 67.97) .. (70.146, 66.673)..
		controls (75.675, 64.499) and (81.769, 59.692) .. (87.305, 61.849)..
		controls (91.269, 63.394) and (95.18, 68.615) .. (94.016, 72.708)..
		controls (92.977, 76.361) and (87.711, 79.017) .. (83.82, 77.797)..
		controls (79.342, 76.392) and (76.105, 69.736) .. (71.554, 70.884)..
		controls (66.408, 72.181) and (68.042, 86.14) .. (63.05, 84.341)..
		controls (59.399, 83.025) and (66.613, 76.073) .. (64.538, 72.794) -- cycle;
		
		\path[draw=figcolor,fill=figcolor!10,line width=1.0pt]
		(33.829, 34.545) -- (40.89, 46.775)..
		controls (38.985, 48.868) and (37.95, 53.023) .. (35.467, 52.285)..
		controls (32.949, 51.536) and (35.338, 46.725) .. (33.826, 44.577)..
		controls (31.1, 40.703) and (20.905, 41.763) .. (21.741, 37.101)..
		controls (22.427, 33.279) and (29.445, 35.922) .. (33.163, 34.797)..
		controls (33.383, 34.731) and (33.606, 34.634) .. (33.829, 34.545) -- cycle;
		
		\path[draw=figcolor,line cap=butt,line join=miter,line width=0.5pt,miter limit=4.0,dash pattern=on 0.5pt off 2.0pt]
		(49.337, 29.375) -- (49.337, 39.703);
		
		\path[draw=figcolor,line cap=butt,line join=miter,line width=0.5pt,miter limit=4.0,dash pattern=on 0.5pt off 2.0pt]
		(51.711, 42.537) -- (51.711, 93.561);
		
		\path[draw=figcolor,line cap=butt,line join=miter,line width=0.5pt,miter limit=4.0,dash pattern=on 0.5pt off 2.0pt]
		(62.147, 83.594) -- (26.987, 22.695);
		
		\path[draw=figcolor,line cap=butt,line join=miter,line width=0.5pt,miter limit=4.0,dash pattern=on 0.5pt off 2.0pt]
		(92.819, 74.894) -- (108.226, 48.206);
		
		\path[draw=figcolor,fill=figcolor!10,line cap=butt,line join=miter,line width=1.0pt,miter limit=4.0]
		(33.829, 34.545).. controls (39.086, 32.45) and (44.988, 24.573) .. (48.876, 28.364)..
		controls (50.992, 30.427) and (45.095, 33.766) .. (45.727, 36.653)..
		controls (46.324, 39.386) and (52.32, 39.807) .. (51.71, 42.537)..
		controls (50.944, 45.972) and (44.729, 44.164) .. (41.748, 46.034)..
		controls (41.44, 46.227) and (41.157, 46.482) .. (40.89, 46.775) -- cycle;
	\end{tikzpicture}
\end{remark}

\begin{example}
	For any pair of angles $\alpha,\beta\in [0,\pi/2)$, the induced cones $\up_\alpha,\down_\beta$ on~$\mathbb{R}^n$ ($n\geq 2$) from Example 4.53 in~\cite{schaaf2024TowardsPointFreeSpacetimes} are parallel if and only if $\alpha=\beta$. If the angles are equal then the cones come from a relation (the causality relation in \emph{Minkowski space}), and so are parallel~(\cref{example:cones from relation are parallel}). If the angles are not equal then we easily find a contradiction to the parallelness condition in~\cref{remark:parallel implies in future iff in past}:
\end{example}
\hfil
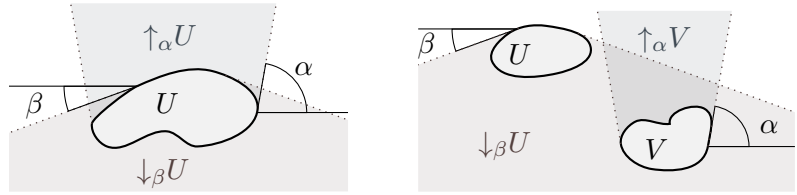
\begin{tikzpicture}[y=1pt, x=1pt, inner sep=0pt, outer sep=0pt]
	\begin{scope}[blend group = multiply]
		\path[fill=accentcolor!10,line cap=butt,line join=miter,line width=1.0pt,miter limit=4.0]
		(48.876, 79.388) -- (0.687, 61.848) -- (0.687, 39.703) -- (128.247, 39.776) --
		(128.247, 66.672) -- (82.121, 83.461) -- cycle;
		\path[fill=figcolor!10,line cap=butt,line join=miter,line width=1.0pt,miter limit=4.0]
		(31.869, 65.214) -- (94.023, 69.47) -- (101.27, 110.569) -- (23.871, 110.569) -- cycle;
	\end{scope}
	
	\path[draw=black,fill=black!05,thick]
	(31.869, 65.214).. controls (31.533, 61.824) and (34.36, 57.939) .. (37.538, 56.71)..
	controls (44.805, 53.901) and (52.47, 63.23) .. (60.215, 62.38)..
	controls (64.415, 61.919) and (67.333, 56.51) .. (71.554, 56.71)..
	controls (80.458, 57.133) and (93.624, 61.99) .. (94.231, 70.884)..
	controls (94.733, 78.246) and (84.376, 83.241) .. (77.223, 85.057)..
	controls (67.883, 87.428) and (57.673, 83.321) .. (48.876, 79.388)..
	controls (45.651, 77.945) and (32.113, 67.686) .. (31.869, 65.214) -- cycle;
	
	\path[draw=accentcolor,line cap=butt,line join=miter,line width=0.5pt,miter limit=4.0,dash pattern=on 0.5pt off 2.0pt]
	(94.023, 69.47) -- (101.27, 110.569);
	\path[draw=accentcolor,line cap=butt,line join=miter,line width=0.5pt,miter limit=4.0,dash pattern=on 0.5pt off 2.0pt]
	(31.869, 65.214) -- (23.871, 110.569);
	\path[draw=accentcolor,line cap=butt,line join=miter,line width=0.5pt,miter limit=4.0,dash pattern=on 0.5pt off 2.0pt]
	(48.876, 79.388) -- (0.687, 61.848);
	\path[draw=accentcolor,line cap=butt,line join=miter,line width=0.5pt,miter limit=4.0,dash pattern=on 0.5pt off 2.0pt]
	(82.121, 83.461) -- (128.247, 66.672);
	
	\path[draw=black]
	(94.231, 69.47) -- (128.247, 69.47);
	
	\path[draw=black]
	(97.301, 87.868) arc(79.895:-0.05:18.681) -- (94.023, 69.477) -- cycle;
	
	\node[anchor=south west]
	(text45) at (55.875, 69.211) {$U$};
	
	\node[text=figcolor,line cap=butt,line join=miter,line width=1.0pt,miter limit=4.0,anchor=south west]
	(text46) at (51.294, 92.662) {$\up_{\alpha} U$};
	
	\node[text=accentcolor,line cap=butt,line join=miter,line width=1.0pt,miter limit=4.0,anchor=south west]
	(text47) at (48.774, 41) {$\down_{\beta}U$};
	
	\node[anchor=south west]
	(text48) at (108.094, 83.461) {$\alpha$};
	
	\path[draw=black]
	(48.876, 79.388) -- (0.687, 79.388);
	
	\path[draw=black]
	(23.01, 69.992) arc(199.697:180.074:27.529 and 27.838) -- (48.928, 79.375) -- cycle;
	
	\node[anchor=south west]
	(text49) at (6.844, 67) {$\beta$};
\end{tikzpicture}\hfil
\begin{tikzpicture}[y=1pt, x=1pt, inner sep=0pt, outer sep=0pt]
	\begin{scope}[blend group = multiply]
		\path[fill=figcolor!10,line cap=butt,line join=miter,line width=1.0pt,miter limit=4.0]
		(62.768, 60.453) -- (53.932, 110.569) -- (104.499, 110.569) -- (95.502, 59.545) -- cycle;
		\path[fill=accentcolor!10,line cap=butt,line join=miter,line width=1.0pt,miter limit=4.0]
		(24.355, 104.008) -- (-13.486, 90.235) -- (-13.486, 42.537) --
		(128.247, 42.537) -- (128.247, 72.428) -- (46.042, 102.348) -- cycle;
	\end{scope}
	
	\path[draw=black]
	(95.842, 59.502) -- (128.247, 59.502);
	
	\path[draw=black]
	(98.184, 74.65) arc(80.426:-0.05000000000001137:15.355 and 15.353) -- (95.63, 59.51) -- cycle;
	
	\node[text=figcolor,line cap=butt,line join=miter,line width=1.0pt,miter limit=4.0,anchor=south west]
	(text46) at (69.182, 95) {$\up_{\alpha} V$};
	
	\node[text=accentcolor,line cap=butt,line join=miter,line width=1.0pt,miter limit=4.0,anchor=south west]
	(text47) at (9.247, 54.261) {$\down_{\beta}U$};
	
	\node[text=black,line cap=butt,line join=miter,line width=1.0pt,miter limit=4.0,anchor=south west]
	(text48) at (115, 64.735) {$\alpha$};
	
	\path[draw=black]
	(23.72, 103.757) -- (-13.486, 103.757);
	
	\path[draw=black]
	(0.687, 95.388) arc(199.697:180.074:24.519 and 24.793) -- (23.771, 103.744) -- cycle;
	
	\node[anchor=south west]
	(text49) at (-13.744, 92.5) {$\beta$};
	
	\path[draw=black,fill=black!05, thick]
	(34.703, 104.899).. controls (29.932, 105.569) and (24.462, 104.849) .. (20.53, 102.065)..
	controls (17.211, 99.715) and (12.62, 95.28) .. (14.299, 91.576)..
	controls (17.721, 84.027) and (30.659, 85.221) .. (38.747, 87.038)..
	controls (43.858, 88.186) and (51.879, 91.015) .. (51.489, 96.239)..
	controls (51.02, 102.517) and (40.938, 104.024) .. (34.703, 104.899) -- cycle;
	
	\path[draw=black,fill=black!05,thick]
	(81.187, 68.049).. controls (76.89, 67.366) and (72.037, 70.816) .. (68.156, 68.847)..
	controls (65.191, 67.343) and (62.24, 63.735) .. (62.768, 60.453)..
	controls (63.561, 55.532) and (69.459, 51.781) .. (74.388, 51.041)..
	controls (81.892, 49.915) and (93.835, 50.087) .. (95.502, 59.545)..
	controls (96.752, 66.635) and (98.31, 71.297) .. (94.231, 73.718)..
	controls (90.154, 76.138) and (80.222, 72.69) .. (81.187, 68.049)..
	controls (81.187, 68.049) and (81.187, 68.049) .. (81.187, 68.049) -- cycle;
	
	\path[draw=accentcolor,line cap=butt,line join=miter,line width=0.5pt,miter limit=4.0,dash pattern=on 0.5pt off 2.0pt]
	(24.355, 104.008) -- (-13.486, 90.235);
	\path[draw=accentcolor,line cap=butt,line join=miter,line width=0.5pt,miter limit=4.0,dash pattern=on 0.5pt off 2.0pt]
	(46.621, 102.137) -- (128.247, 72.428);
	\path[draw=accentcolor,line cap=butt,line join=miter,line width=0.5pt,miter limit=4.0,dash pattern=on 0.5pt off 2.0pt]
	(62.768, 60.453) -- (53.932, 110.569);
	\path[draw=accentcolor,line cap=butt,line join=miter,line width=0.5pt,miter limit=4.0,dash pattern=on 0.5pt off 2.0pt]
	(95.502, 59.545) -- (104.499, 110.569);
	
	\node[anchor=south west]
	(text57) at (20.53, 90.726) {$U$};
	
	\node[anchor=south west]
	(text57-5) at (71.796, 54.936) {$V$};
\end{tikzpicture}\hfil

\begin{remark}
	\label{remark:parallelness in the literature}
	Note the resemblance of parallelness to the Frobenius reciprocity condition (\cref{definition:open map}). We will see in~\cref{proposition:induced cones are parallel} that Frobenius reciprocity is indeed used to show that the induced cones~$\Up,\Down$ of an open cone localic relation are parallel. Some discussion on this is at the end of~\cref{section:internal cones}.
	
	Moreover, the property in~\cref{remark:parallel implies in future iff in past} appears already in the study of Boolean algebras with operators by J\'onsson and Tarski~\cite[Definition~1.11]{jonsson1951BooleanAlgebrasOperators}. There, a pair of functions~$\uc,\dc$ on a Boolean algebra are called \emph{conjugate} if it holds that~${\uc x\wedge y=\bot}$ iff~$x\wedge \dc y= \bot$. In the Boolean setting, $\uc,\dc$~(preserving~$\bot$) are conjugate iff they are parallel~\cite[Theorem~1.15]{jonsson1951BooleanAlgebrasOperators}, but in general parallelness is a stronger property. Thus parallelness can be seen as a positive generalisation of being conjugate. The use by J\'onnson and Tarski of the conjugate condition is essentially the same as the present motivation: to ensure that~$\uc,\dc$ indeed come from a single relation, and to give an algebraic way to talk about a relation and its opposite. Conjugate pairs are also used in modal logic, see for example~\cite{moller2006AlgebrasModalOperators} and~\cite[\S 3.3]{goldblatt2003MathematicalModalLogic}.
	
	Specifically in locale theory, the parallelness condition occurs in~\cite[\S 3]{manuell2023PresentingQuotientLocales} as an openness condition of localic quotients. More generally, the `modular law' of an allegory has the same form as the parallelness condition~\cite{freyd1990CategoriesAllegories}.	
\end{remark}

In~\cref{section:cones to relations}, describing the strategy of constructing a relation out of cones, we encountered terms of the form $\up A\cap B$ and $A\cap \down B$. Here we briefly introduce some notation and basic lemmas about these assignments, and how they relate to parallelness.

\begin{definition}
	\label{definition:reduced cones}
	For a pair of cones $\uc,\dc\colon L\to L$ on a meet-semilattice, define the \emph{reduced cones}:
	\[
	\Uc,\Dc\colon L\times L\longrightarrow L
	\quad\text{by}\quad
	\Uc(x,y):= \uc x\wedge y
	\quad\text{and}\quad
	\Dc(x,y):= x\wedge \dc y.
	\]
\end{definition}

\begin{remark}
	\label{remark:cone from reduced cone}
	Obviously the cones $\uc,\dc$ are recovered from the reduced ones by plugging in the top element of the semilattice:
	\[
		\uc(-)=\Uc(-,\top)
		\qquad\text{and}\qquad
		\dc(+)=\Dc(\top,+).
	\]
\end{remark}

\begin{lemma}
	\label{lemma:parallel in terms of reduced cones}
	The cones $\uc,\dc\colon L\to L$ are parallel iff for all $x,y\in L$:
	\[
	\Uc(x,y)\sqleq \uc \Dc(x,y)
	\qquad\text{and}\qquad
	\Dc(x,y)\sqleq \dc \Uc(x,y).
	\]
\end{lemma}
\begin{proof}
	Obvious from the definition.
\end{proof}

\begin{lemma}
	\label{lemma:parallel iff reduced cones idempotent}
	The cones $\uc,\dc\colon L\to L$ are parallel iff the reduced cones $\Uc,\Dc$ are \emph{idempotent} in the following sense:
	\[
	\Dc(\Dc(x,y),\Uc(x,y))=\Dc(x,y)
	\qquad\text{and}\qquad
	\Uc(\Dc(x,y),\Uc(x,y))=\Uc(x,y).
	\]
\end{lemma}
\begin{proof}
	Abbreviate $\Dc=\Dc(x,y)$ and $\Uc=\Uc(x,y)$. Then $\Dc(\Dc,\Uc) = \Dc\wedge \dc \Uc$ and $\Uc(\Dc,\Uc)= \uc \Dc \wedge \Uc$, so by~\cref{lemma:parallel in terms of reduced cones} it is clear that the stated condition is equivalent to $\uc,\dc$ being parallel.
\end{proof}

\begin{remark}
	In~\cref{section:cones to relations,intuition:sim} we saw that the relation~$\sim$ will identify rectangles precisely when their source and target projections are equal. In other words, precisely when their projections under~$\Uc,\Dc$ are equal. \cref{lemma:parallel iff reduced cones idempotent} then says that a rectangle $x\tensor y$ will be canonically identified with its reduction~${\Dc(x,y)\tensor \Uc(x,y)}$, and this reduction is idempotent. Spatially, the reduction is of course just the intersection of~$x\tensor y$ with~$R$, so idempotence is just that of~$\cap$.
\end{remark}

\section{Conic frames}
\label{section:conic frames}
In this section we formally introduce and study the main new notion: \emph{conic frames}. We saw in~\cref{intuition:cones} that relations on sets are in bijective correspondence with join-preserving parallel pairs on the powerset. Moreover, we will show in~\cref{proposition:induced cones preserve joins,proposition:induced cones are parallel} below that any open cone localic relation induces such a pair of operators on its underlying frame. Axiomatising this intuition, we get the definition of a conic frame~$(L,\uc,\dc)$. The bulk of this section concerns the induced relation~$\sim$ on~$L\tensor L$ based on the spatial~\cref{intuition:sim}. We set up the technical theory needed to describe the resulting localic relation~$R_\ucdc$, in particular characterising parallelness in terms of~$\sim$ and the reduced cones~$\Uc,\Dc$, and showing how join-preservation of~$\uc,\dc$ will allow us to construct the left adjoints of the source and target maps of~$R_\ucdc$, while also giving necessary properties of the frame quotient map~$\mu_\sim$. The section ends with the definition of the category~$\ucdcFrm$ of conic frames.

\begin{definition}
	\label{definition:conic frame}
	A \emph{conic frame} $(L,\uc,\dc)$ is a frame $L$ equipped with join-preserving parallel cones $\uc,\dc\colon L\to L$.
\end{definition}

\begin{example}
	\label{example:conic frames basic examples}
	We discuss some basic examples:
	\begin{itemize}
		\item for any frame $L$ the triple $(L,\id_L,\id_L)$ is a conic frame; 
		
		\item for any frame $L$ the triple $(L,\bot,\bot)$ with the constant bottom cones is a conic frame;
		
		\item for any frame $L$ and $v,w\in L$ the maps $v\wedge -$ and $w\wedge -$ preserve joins by infinite distributivity, and they are parallel iff $v=w$. Thus $(L,v\wedge -,w\wedge -)$ is a conic frame iff $v=w$;
		
		\item for any conic frame $(L,\uc,\dc)$ its \emph{opposite} is the conic frame~$(L,\dc,\uc)$;
			
		\item for any open locale map $f\colon X\to X$ the triple $(\Opens X, f_!,f^{-1})$ is a conic frame, where parallelness follows from Frobenius reciprocity;
		
		\item for any space~$S$ with open cone relation~$R$ the tuple $(\Opens S,\up,\down\,)$ is a conic frame by~\cref{proposition:cones preserve joins in Set,example:cones from relation are parallel}.
	\end{itemize}
\end{example}

Most importantly, and generalising the last example, we show that the cones~$\Up,\Down$ coming from any open cone localic relation~$R$ on a locale~$X$ satisfy the axioms of a conic frame. Denote its source and target maps as usual by~$s,t$.

\begin{proposition}
	\label{proposition:induced cones preserve joins}
	The cones $\Up,\Down$ preserve all joins.
\end{proposition}
\begin{proof}
	This follows since $s^{-1},t^{-1}$ are frame maps and $s_!,t_!$ are left adjoints.
\end{proof}

\begin{proposition}
	\label{proposition:induced cones are parallel}
	The cones $\Up,\Down$ are parallel.
\end{proposition}
\begin{proof}
	We prove $\Up U\wedge V\sqleq \Up (U\wedge \Down V)$ for all~$U,V\in\Opens X$; the other inclusion is analogous. First note that since~$s_!\dashv s^{-1}$ we get~$O\sqleq s^{-1}s_!(O)$ for all~$O\in\Opens R$,~so:
	\begin{align*}
		t_!\left(s^{-1}(U)\wedge t^{-1}(V)\right)
		&\sqleq
		t_! \left( s^{-1}(U)\wedge s^{-1}s_!t^{-1}(V)\right)
		&& (s_!\dashv s^{-1})
		\\&=
		t_!s^{-1}\left(U\wedge s_! t^{-1}(V)\right)
		&&(\text{$s^{-1}$ preserves $\wedge$})
		\\&=
		\Up(U\wedge\Down V).
	\end{align*}
	Using Frobenius reciprocity for $t$, the left-hand side reduces to
	\[
	t_!\left(s^{-1}(U)\wedge t^{-1}(V)\right)
	=
	t_!s^{-1}(U) \wedge V = \Up U\wedge V,
	\]
	giving the desired inclusion.
\end{proof}

Thus, for any open cone localic relation~$(X,R)$ the tuple $(\Opens X,\Up,\Down)$ is a conic frame by~\cref{proposition:induced cones preserve joins,proposition:induced cones are parallel}. Of course, the adjunction in~\cref{section:adjunction} will reveal that this class of examples really covers all conic frames.	
	
The rest of this section will be dedicated to setting up the required technical background to construct a localic relation out of a conic frame. Thus, we fix a frame $L$ equipped with a pair of cones $\uc,\dc\colon L\to L$. For now, we only assume that $\uc,\dc$ preserve the bottom element. The goal is to construct a relation~$\sim$ on the coproduct~$L\tensor L$ out of $\uc,\dc$, which then on the localic side will induce a sublocale: the induced relation~$R_\ucdc$. The relation~$\sim$ is defined according to the spatial~\cref{intuition:sim}.

\begin{definition}
	\label{definition:sim}
	Define a relation~$\sim$ on the rectangles of $L\tensor L$:
	\[
		a\tensor b\sim c\tensor d
		\quad\iff\quad
		\Uc(a,b)=\Uc(c,d)
		\text{ and }
		\Dc(a,b)=\Dc(c,d).
	\]
\end{definition}

\begin{remark}
	Note that $\sim$ is well-defined as a relation on $L\tensor L$ iff $\uc,\dc$ preserve the bottom element $\bot\in L$. Otherwise, this may lead to contradictions since the bottom element in the coproduct may be represented as $x\tensor \bot$ or $\bot\tensor y$ for any~$x,y\in L$. Since conic frames have join-preserving cones, this assumption is harmless.
\end{remark}

\subsection{Parallelness and $\sim$}
The next point of order is to see how parallelness of~$\uc,\dc$ interacts with~$\sim$. In particular, we show that under parallelness the relation~$\sim$ defines a meet-semilattice congruence.

\begin{lemma}
	\label{lemma:parallel in terms of sim}
	The cones $\uc,\dc$ are parallel iff for all $x,y\in L$:
	\[
	x\tensor y \sim \Dc(x,y)\tensor \Uc(x,y).
	\]
\end{lemma}
\begin{proof}
	Abbreviate $\Uc(x,y)=\Uc$ and $\Dc(x,y)=\Dc$. Then the relation holds precisely when
	\[
		\Uc = \Uc(\Dc,\Uc)
		\qquad\text{and}\qquad
		\Dc= \Dc(\Dc,\Uc),
	\]
	which is equivalent to parallelness by~\cref{lemma:parallel iff reduced cones idempotent}.
\end{proof}

\begin{remark}
	Hence if $\uc,\dc$ are parallel, so $\Uc,\Dc$ are idempotent, any rectangle $x\tensor y$ has a canonical reduced form~$\Dc(x,y)\tensor \Uc(x,y)$ in its $\sim$-equivalence class. Indeed:
	\[
	\Dc\tensor \Uc \sim \Dc(\Dc,\Uc) \tensor \Uc(\Dc,\Uc)= \Dc\tensor \Uc.
	\]
\end{remark}

\begin{proposition}
	\label{lemma:sim is meet congruence}
	\label{proposition:sim is a meet congruence}
	If the cones $\uc,\dc$ are monotone and parallel, then $\sim$ is a meet-semilattice congruence.
\end{proposition}
\begin{proof}
	It is clearly an equivalence relation, so it remains to show it is closed under meets. Thus take $a\tensor b\sim c\tensor d$ and $x\tensor y$. Write $a'= a\wedge x$, $b'=b\wedge y$, $c' = c\wedge x$ and $d'=d\wedge y$. We need to show $a'\tensor b'\sim c'\tensor d'$. 
	
	First, we claim the following inclusions hold:
	\[
	\Dc(a',b')\sqleq c',
	\quad
	\Uc(a',b')\sqleq d',
	\quad
	\Dc(c',d')\sqleq a',
	\quad
	\Uc(c',d')\sqleq b'.
	\]
	To see this, use monotonicity of $\dc$ to get
	\[
	\Dc(a',b') = a\wedge x\wedge \dc(b\wedge y) \sqleq a \wedge x\wedge \dc(b) = \Dc(a,b)\wedge x.
	\]
	From $a\tensor b\sim c\tensor d$ the right hand side is equal to $\Dc(c,d)\wedge x$, which in turn is contained in $c\wedge x= c'$, giving the first inclusion. The other inclusions are proven analogously.
	
	Note that $\Uc,\Dc$ inherit monotonicity from $\uc,\dc$. From the inclusions ${\Dc(a',b')\sqleq c'}$ and $\Uc(a',b')\sqleq d'$ and the idempotence of the reduced cones it thus follows that
	\[
	\Dc(a',b')=\Dc\left(\Dc(a',b'),\Uc(a',b')\right)\sqleq \Dc(c',d').
	\]
	Applying the same argument to $\Dc(c',d')\sqleq a'$ and $\Uc(c',d')\sqleq b'$ we thus get the equality $\Dc(a',b')=\Dc(c',d')$. That $\Uc(a',b')=\Uc(c',d')$ is proved dually, and hence we obtain the desired $a'\tensor b'\sim c'\tensor d'$.
\end{proof}

\begin{remark}
	From \cref{proposition:saturated simplified when meet preserved} we thus get a simplified saturation condition for~$\sim$, making the induced sublocale easier to describe.
\end{remark}

\subsection{Separation}
\label{section:separation}
The localic relation $R_\ucdc$ induced by a conic frame $(L,\uc,\dc)$ will of course be the one defined by the binary relation~$\sim$ on $L\tensor L$. The relation~$R_\ucdc$ again needs to have open cones. Based on the spatial intuition in~\cref{section:relations to cones}, we want to obtain left adjoints $s_!\dashv s_\ucdc^{-1}$ and~$t_!\dashv t_\ucdc^{-1}$ that correctly mirror the spatial intuition
\[
s_!(x\tensor y) = x\wedge \dc y
\quad\text{and}\quad
t_!(x\tensor y) = \uc x\wedge y.
\]
Since $s_!,t_!$ are eventually going to be join-preserving, they should be fully determined by these equations. However, for join-preservation to hold, we need at least that whenever~${x\tensor y \sqleq \bigvee x_i\tensor y_i}$ in~$L\tensor L$:
\[
x\wedge \dc y = s_!(x\tensor y)
\sqleq
s_!\left(\bigvee x_i\tensor y_i\right)
=
\bigvee s_!(x_i\tensor y_i)
= 
\bigvee x_i\wedge \dc y_i,
\]
and an analogous inclusion with~$\uc$ for~$t_!$. It is not immediately obvious how this join-preservation property on~$L\tensor L$ follows from join-preservation of~$\uc,\dc$ on~$L$, so in this section we study cones $\uc,\dc$ that satisfy these inclusions, and call them \emph{separating}.
Nevertheless, we will show that~$\uc,\dc$ are separating iff they are join-preserving, so for conic frames the assumption is again harmless.

\begin{definition}
	\label{definition:separating cones}
	We say a pair of cones $\uc,\dc\colon L\to L$ on a frame are \emph{separating} if for all $x,y,x_i,y_i\in L$:
	\[
	x\tensor y\sqleq \bigvee x_i\tensor y_i
	\quad\implies\quad
	\begin{array}{l}
		\Uc(x,y)\sqleq \displaystyle\bigvee \Uc(x_i,y_i),\\
		\Dc(x,y)\sqleq \displaystyle\bigvee \Dc(x_i,y_i).
	\end{array}
	\]
\end{definition}

When restricted to rectangles, the left adjoints~$s_!,t_!$ should overlap with the (output of the) reduced cones~$\Uc,\Dc$. Motivated by this and~\cref{theorem:adjoint iff preserves meets/joins}, we can try to construct the right adjoints directly.

\begin{definition}
	\label{definition:separating opens}
	Let $\uc,\dc\colon L\to L$ be a pair of cones on a frame. Define the \emph{separating opens} in $L\tensor L$ for $z\in L$ by:
	\[
	O_z^\Uc := \bigvee\{ x\tensor y: \Uc(x,y)\sqleq z \}
	\quad\text{and}\quad
	O_z^\Dc:= \bigvee\{ x\tensor y: \Dc(x,y)\sqleq z \}.
	\]
\end{definition}

\begin{lemma}
	\label{lemma:separating in terms of separating opens}
	The cones $\uc,\dc$ are separating iff for all $x,y,z\in L$:
	\begin{align*}
		x\tensor y\sqleq O_z^\Uc
		&\quad\iff\quad
		\Uc(x,y)\sqleq z,
		\\
		x\tensor y\sqleq O_z^\Dc
		&\quad\iff\quad 
		\Dc(x,y)\sqleq z.
	\end{align*}
\end{lemma}
\begin{proof}
	Suppose $\uc,\dc$ are separating. If $\Uc(x,y)\sqleq z$ then $x\tensor y$ appears in the join defining $O^\Uc_z$, so $x\tensor y\sqleq O^\Uc_z$. Conversely, if $x\tensor y \sqleq O^\Uc_z$ then by separation
	\[
	\Uc(x,y)\sqleq \bigvee\{\Uc(a,b):\Uc(a,b)\sqleq z\}\sqleq z.
	\]
	
	Now suppose the stated equivalences hold. Let $x\tensor y \sqleq \bigvee x_i\tensor y_i$ and define~$z$ as~$\bigvee \Uc(x_i,y_i)$. Thus $\Uc(x_i,y_i)\sqleq z$ for all~$i$, which now holds iff $x_i\tensor y_i\sqleq O^\Uc_z$ for all~$i$. Taking the join we get $x\tensor y \sqleq \bigvee x_i\tensor y_i \sqleq O^\Uc_z$, and applying the hypothesis once more this gives $\Uc(x,y)\sqleq z := \bigvee \Uc(x_i,y_i)$, as desired. 
\end{proof}

\begin{lemma}
	\label{lemma:separating implies kappa right adjoint}
	If $\uc,\dc$ are separating, the maps $O^\Uc_{(-)}$ and $O^\Dc_{(-)}$ are right adjoints.
\end{lemma}
\begin{proof}
	The proof works independently for each cone, so take a function $\uc\colon L\to L$ satisfying the condition in~\cref{definition:separating cones}. Denote:
	\[
	\tau \colon L\longrightarrow L\tensor L; \qquad z\longmapsto O_z^\Uc = \bigvee\{x\tensor y: \Uc(x,y)\sqleq z\}.
	\]
	It suffices to prove $\tau$ preserves all meets~(\cref{theorem:adjoint iff preserves meets/joins}). So, take a family $(z_i)_{i\in I}$ in $L$, and calculate:
	\begin{align*}
		a\tensor b\sqleq \tau\left(\bigwedge_{i\in I}z_i\right)
		&\iff
		\Uc(a,b)\sqleq \bigwedge_{i\in I} z_i
		&&(\text{\cref{lemma:separating in terms of separating opens}})
		\\&\iff
		\forall i\in I: \Uc(a,b)\sqleq z_i
		\\&\iff
		\forall i\in I: a\tensor b\sqleq \tau(z_i)
		&&(\text{\cref{lemma:separating in terms of separating opens}})
		\\&\iff
		a\tensor b\sqleq \bigwedge_{i\in I} \tau(z_i).
	\end{align*}
	Thus $\tau\left(\bigwedge_{i\in I}z_i\right)$ and $\bigwedge_{i\in I} \tau(z_i)$ contain the same rectangles, from which it follows by~\cref{lemma:equality if contain same basics} that they must be equal, and hence $\tau$ preserves meets. The proof that $\sigma\colon z\mapsto O^\Dc_z$ preserves meets is dual.
\end{proof}

Since by~\cref{remark:cone from reduced cone} the cones~$\uc,\dc$ are recovered from the reduced ones~$\Uc,\Dc$, we can recover the cones from the left adjoints of the separating opens, showing they are join-preserving.

\begin{lemma}
	\label{proposition:separating implies join-preserving}
	\label{lemma:separating implies join-preserving}
	If $\uc,\dc$ are separating, then they are join-preserving.
\end{lemma}
\begin{proof}
	Take the meet-preserving map $\tau\colon z\mapsto O^\Uc_z$ from the proof of~\cref{lemma:separating implies kappa right adjoint}.	It admits a left adjoint $\Tau\dashv \tau$ given through~\cref{theorem:adjoint iff preserves meets/joins} by
	\[
	\Tau\colon L\tensor L\longrightarrow L;
	\qquad
	o\longmapsto  \bigwedge\{z\in L: o\sqleq \tau(z)\}.
	\]
	Using \cref{lemma:separating in terms of separating opens} again, calculate for any $x\in L$ that:
	\[
	\Tau(x\tensor \top)
	=
	\bigwedge \{z: x\tensor \top \sqleq O^\Uc_z\}
	=
	\bigwedge\{z: \Uc(x,\top)\sqleq z\}
	=
	\Uc(x,\top) = \uc x.
	\]
	Finally, recall that the frame coproduct inclusion $\iota_1\colon L\to L\tensor L$ into the first component is given by $\iota_1(x)=x\tensor \top$~(\cref{lemma:projections}). Hence
	\[
	\uc= \Tau \circ \iota_1
	\]
	preserves all joins, since $\iota_1$ is a frame map and $\Tau$ is a left adjoint. Similarly, we get that $\dc = \Sigma\circ \iota_2$ preserves joins, where $\Sigma\dashv \sigma$ is the left adjoint to the map~$\sigma\colon z\mapsto O^\Dc_z$.
\end{proof}

Approaching the situation from the other side, we can characterise separation of~$\uc,\dc$ in terms of these left adjoints~$\Sigma,\Tau$.

\begin{lemma}
	\label{lemma:separating iff Sigma Tau}
	Cones $\uc,\dc$ are separating iff there exist join-preserving functions $\Sigma,\Tau\colon L\tensor L\to L$ such that on rectangles
	\[
	\Sigma(x\tensor y)= x\wedge \dc y
	\quad\text{and}\quad
	\Tau(x\tensor y) = \uc x\wedge y.
	\]
\end{lemma}
\begin{proof}
	If $\uc,\dc$ are separating, then we get join-preserving maps $\Sigma,\Tau\colon L\tensor L\to L$ via the proof of~\cref{lemma:separating implies join-preserving}. Note indeed that $\Sigma\dashv \sigma$, where $\sigma\colon z\mapsto O^\Dc_z$, so we get using~\cref{lemma:separating in terms of separating opens} that
	\[
	\Sigma(x\tensor y)
	=
	\bigwedge \{z: x\tensor y \sqleq O^\Dc_z\}
	=
	\bigwedge\{z: \Dc(x,y)\sqleq z\}
	=
	\Dc(x,y)
	=
	x\wedge \dc y,
	\]
	and similarly for $\Tau$.
	
	Conversely, if there exist join-preserving $\Sigma,\Tau\colon L\tensor L\to L$ that restrict correctly on rectangles, we get from~\cref{theorem:adjoint iff preserves meets/joins} right adjoints $\Sigma\dashv \sigma$ and $\Tau\dashv \tau$ that are explicitly calculated as
	\begin{align*}
		\sigma(z)
		&=
		\bigvee\{o\in L\tensor L: \Sigma(o)\sqleq z\}
		&&(\Sigma\dashv \sigma)
		\\&=
		\bigvee\{x\tensor y: \Sigma(x\tensor y)\sqleq z \}
		&&(\text{rectangle basis})
		\\&=
		\bigvee\{x\tensor y: \Dc(x,y)\sqleq z \}
		&&(\text{hypothesis on $\Sigma$})
		\\&=:
		O^\Dc_z,
	\end{align*}
	and similarly $\tau(z)= O^\Uc_z$. Unpacking $\Sigma\dashv \sigma$ on rectangles we thus get
	\[
	x\tensor y\sqleq O^\Dc_z
	\iff
	\Sigma(x\tensor y)\sqleq z
	\iff
	\Dc(x,y)\sqleq z,
	\]
	and similarly for $\Uc$, so separation follows from
	\cref{lemma:separating in terms of separating opens}.
\end{proof}

\begin{remark}
	Of course, these $\Sigma,\Tau$ will be the left adjoints~$s_!,t_!$ corresponding to the open source and target map of the induced localic relation~$R_\ucdc$. This will be made precise in~\cref{section:localic relation induced by cones}.	
\end{remark}

\begin{proposition}
	Cones $\uc,\dc$ are separating iff they are join-preserving.
\end{proposition}
\begin{proof}
	Given~\cref{lemma:separating implies join-preserving}, it remains to show that if $\uc,\dc$ are join-preserving, then they are separating. For this we use the following important fact: if $L,M$ are frames, then their coproduct $L\tensor M$ is equivalently described by their suplattice tensor product. This can be seen from the explicit construction of frame coproducts in~\cite[\S II.2.12]{johnstone1982StoneSpaces}. See also~\cite[\S 1]{johnstone1991PreframePresentationsPresent}.
	
	Then, the cones~$\uc,\dc$ being join-preserving implies immediately that the reduced cones~${\Uc,\Dc\colon L\times L\to L}$ are suplattice bihomomorphisms (preserving joins in each component). By the universal property of suplattice tensor products we hence get join-preserving functions $\Sigma,\Tau\colon L\tensor L\to L$ such that $\Sigma(x\tensor y)= \Dc(x,y)$ and~${\Tau(x\tensor y)=\Uc(x,y)}$, which implies $\uc,\dc$ are separating by~\cref{lemma:separating iff Sigma Tau}.
\end{proof}

\begin{remark}
	Thus separation and join-preservation are equivalent, and in particular the cones of any conic frame are separating. In what follows we temporarily keep using the terminology of separation, since this highlights the properties in~\cref{definition:separating cones,lemma:separating in terms of separating opens} that are needed in the proofs. But, of course, the proofs will work in any conic frame.
\end{remark}

\subsection{The induced quotient map}
Recall from~\cref{section:sublocales} that for any binary relation~$\sim$ on a frame, we get the quotient map~$\mu_\sim$ that takes an element and produces the smallest $\sim$-saturated element containing it. Applying this to the relation~$\sim$ on~$L\tensor L$ from~\cref{definition:sim} we get a frame quotient map:
\[
\mu_\sim\colon L\tensor L\longrightarrow L\tensor L/{\sim};
\qquad
\mu_\sim(o)=\bigwedge\{s\in L\tensor L/{\sim}: o\sqleq s\}.
\]
The point of this section is to show that the frame congruence~$\langle\sim\rangle$ corresponding to the quotient map~$\mu_\sim$ reduces correctly to~$\sim$ on rectangles, and hence does not make any wrong identifications. To do this, we first show that applying~$\mu_\sim$ to rectangles produces precisely the separating opens.

\begin{lemma}
	\label{proposition:separating opens are sim saturated}
	Assume $\uc,\dc$ are a separating parallel pair. Then the separating opens $O^\Uc_z,O^\Dc_z$ are $\sim$-saturated.
\end{lemma}
\begin{proof}
	By \cref{proposition:saturated simplified when meet preserved,proposition:sim is a meet congruence} we get the simplified saturation condition:
	\[
	a\tensor b\sim c\tensor d
	\quad\implies\quad
	\left(a\tensor b\sqleq O^\Uc_z \text{ iff } c\tensor d\sqleq O^\Uc_z\right).
	\]
	The hypothesis gives $\Uc(a,b)=\Uc(c,d)$, so using separation via~\cref{lemma:separating in terms of separating opens}:
	\[
	a\tensor b\sqleq O^\Uc_z \iff \Uc(a,b)\sqleq z \iff \Uc(c,d)\sqleq z \iff c\tensor d \sqleq O^\Uc_z.
	\]
	The proof for $O^\Dc_z$ is dual.
\end{proof}

\begin{lemma}
	\label{lemma:mu recovers separating opens}
	If $\uc,\dc$ are separating parallel, then:
	\[
	\mu_\sim(z\tensor \top) = O^\Dc_z
	\quad\text{and}\quad
	\mu_\sim(\top\tensor z) = O^\Uc_z.
	\]
\end{lemma}
\begin{proof}
	Note that for all $z\in L$ we have $\Dc(z,\top)=z\wedge \dc \top \sqleq z$, so by definition of the separating opens we have $z\tensor \top \sqleq O^\Dc_z$. Thus using that $O^\Dc_z$ is $\sim$-saturated (\cref{proposition:separating opens are sim saturated}) we get $\mu_\sim(z\tensor \top)\sqleq \mu_\sim(O^\Dc_z)=O^\Dc_z$. 
	
	For the converse inclusion, take $x\tensor y$ with $\Dc(x,y)\sqleq z$, which are the rectangles generating the separating open. Together with the trivial inclusion $\Uc(x,y)\sqleq \top$, from \cref{lemma:parallel in terms of sim} we get
	\[
	x\tensor y
	\sim
	\Dc(x,y)\tensor \Uc(x,y)
	\sqleq 
	z\tensor \top
	\sqleq \mu_\sim(z\tensor\top).
	\]
	But $\mu_\sim(z\tensor \top)$ is $\sim$-saturated, so it follows that $x\tensor y\sqleq \mu_\sim(z\tensor\top)$. Taking joins over such $x\tensor y$ we thus get $O^\Dc_z\sqleq \mu_\sim(z\tensor\top)$, as desired.
\end{proof}

Using this lemma, we show that restricting the frame congruence corresponding to~$\mu_\sim$ indeed agrees with~$\sim$. Moreover, this is the case precisely when the cones are separating.

\begin{proposition}
	\label{lemma:mu inclusion implies reduced cone inclusion}
	If $\uc,\dc$ are separating parallel, then:
	\[
	\mu_\sim(c\tensor d)\sqleq \mu_\sim(a\tensor b)
	\quad\implies\quad
	\Uc(c,d)\sqleq \Uc(a,b) \text{ and } \Dc(c,d)\sqleq \Dc(a,b).
	\]
\end{proposition}
\begin{proof}
	The quotient map $\mu_\sim$ equalises~$\sim$, so using \cref{lemma:parallel in terms of sim} we get
	\[
	\mu_\sim(a\tensor b) = \mu_\sim(\Dc(a,b)\tensor \Uc(a,b))
	\sqleq 
	\mu_\sim(\Dc(a,b)\tensor\top).
	\]
	Using \cref{lemma:mu recovers separating opens} to rewrite the right-hand side and using that $\mu_\sim$ is inflationary, assuming $\mu_\sim(c\tensor d)\sqleq \mu_\sim(a\tensor b)$ thus gives:
	\[
	c\tensor d \sqleq \mu_\sim(c\tensor d)\sqleq \mu_\sim(a\tensor b)\sqleq O^\Dc_{\Dc(a,b)}.
	\]
	By separation, this inclusion holds iff $\Dc(c,d)\sqleq \Dc(a,b)$, as desired. The inclusion for~$\Uc$ is proved dually.
\end{proof}

\begin{corollary}
	\label{corollary:mu implies sim}
	A parallel pair~$\uc,\dc$ is separating iff
	\[
	\mu_\sim(c\tensor d)= \mu_\sim(a\tensor b)
	\quad\iff\quad
	a\tensor b\sim c\tensor d.
	\]
\end{corollary}
\begin{proof}
	Observe that~$\mu_\sim$ equalises~$\sim$~(\cref{theorem:frame quotient theorem}), so the stated equivalence holds iff $\mu_\sim(a\tensor b)=\mu_\sim(c\tensor d)$ implies~$a\tensor b\sim c\tensor d$. The cones~$\uc,\dc$ being separating provides this via~\cref{lemma:mu inclusion implies reduced cone inclusion}.
	
	For the converse direction, assume the stated equivalence holds. By~\cref{lemma:separating in terms of separating opens} and its proof we see it suffices to prove $x\tensor y \sqleq O^\Dc_z$ implies $\Dc(x,y)\sqleq z$, as the other direction holds by construction of the separating opens. Similarly, we saw in the proof of~\cref{lemma:mu recovers separating opens} that the inclusion $O^\Dc_z\sqleq \mu_\sim(z\tensor \top)$ always holds. Thus, using~$\mu_\sim$ is idempotent, we get if~$x\tensor y\sqleq O^\Dc_z$ that ${\mu_\sim(x\tensor y)\sqleq \mu_\sim(z\tensor \top)}$. Equivalently, ${\mu_\sim(x\tensor y) = \mu_\sim((x\wedge z)\tensor y)}$, and by hypothesis this implies we have the relation~${x\tensor y \sim (x\wedge z)\tensor y}$, which in turn by definition unpacks to the desired ${\Dc(x,y)= \Dc(x\wedge z,y) \sqleq z}$. The proof for~$\Uc$ is dual.
\end{proof}

\subsection{Conic morphisms}
To finish this section on conic frames, we define their corresponding category~$\ucdcFrm$. We saw in~\cref{lemma:monotonicity in terms of cones} that the point-wise definition of monotonicity can be translated into a pointfree statement involving cones and preimage maps. We adopt this directly.

\begin{definition}
	\label{definition:conic morphism}
	A \emph{conic morphism} $h\colon (L,\uc,\dc)\to (M,\ucx,\dcx)$ between conic frames is a frame map $h\colon L\to M$ such that
	\[
	\ucx\circ h \sqleq h\circ \uc
	\quad\text{and}\quad
	\dcx\circ h \sqleq h\circ \dc.
	\]
\end{definition}

\begin{definition}
	The category of conic frames and morphisms is denoted~$\ucdcFrm$.
\end{definition}

Analogous to how $\id_X\colon (X,R)\to (X,Q)$ is monotone iff~$R\subseteq Q$~(\cref{remark:monotone inclusion}), we record the following fact characterising when the identity map of frames is a conic morphism. 

\begin{lemma}
	\label{lemma:id ucdc map iff cones inclusion}
	The identity map $\id_L\colon (L,\uc,\dc)\to (L,\ucx ,\dcx )$ defines a conic morphism iff ${\ucx \sqleq \uc}$ and~$\dcx \sqleq \dc$.
\end{lemma}

The following important characterisation shows that the conic morphisms are precisely those frame maps that preserve the generating relations~$\sim$. This is the key ingredient that will allow us to define functors between~$\rocLoc$ and~$\ucdcFrm^\op$.

\begin{proposition}
	\label{proposition:conic morphism iff preserves sim}
	A frame map $h\colon (L,\uc,\dc)\to (M,\ucx,\dcx)$ is a conic morphism iff
	\[
	a\tensor b\sim c\tensor d
	\quad\implies\quad
	h(a)\tensor h(b) \sim h(c)\tensor h(d).
	\]
\end{proposition}
\begin{proof}
	Suppose first that $h$ is a conic morphism, fix elements $x,y\in L$, and abbreviate~$\Uc= \Uc(x,y)$ and~$\Dc=\Dc(x,y)$. We claim that
	\[
	\Ucx(h(x),h(y))= \Ucx(h(\Dc),h(\Uc))
	\quad\text{and}\quad
	\Dcx(h(x),h(y))=\Dcx(h(\Dc),h(\Uc)).
	\]
	To see this, observe that $\Dc\sqleq x$ and $\Uc\sqleq y$, so since~$h$ and~$\Ucx$ are monotone we get immediately that~$\Ucx(h(\Dc),h(\Uc))\sqleq \Ucx(h(x),h(y))$, and we are left to prove the converse inclusion. For that, use $\ucx\circ h \sqleq h\circ \uc$ to get
	\[
	\Ucx(h(x),h(y)):= \ucx h(x)\wedge h(y) \sqleq h(\uc x) \wedge h(y) = h(\uc x\wedge y)= h(\Uc).
	\]
	The same argument using~$\dcx\circ h \sqleq h\circ \dc$ gives $\Dcx(h(x),h(y))\sqleq h(\Dc)$, which combined with~\cref{lemma:parallel in terms of reduced cones} gives
	\[
	\Ucx(h(x),h(y)) \sqleq \ucx \Dcx(h(x),h(y)) \sqleq \ucx h(\Dc),
	\]
	which in turn combined with the previous equation gives the desired inclusion:
	\[
	\Ucx(h(x),h(y))\sqleq \ucx h(\Dc)\wedge h(\Uc) = \Ucx(h(\Dc),h(\Uc)).
	\]
	The second equation in the claim is proved analogously. With the claim established, now assume that $a\tensor b \sim c\tensor d$. Explicitly, this means $\Uc(a,b)=\Uc(c,d)$ and $\Dc(a,b)=\Dc(c,d)$, which can then be used in conjunction with the claim to get
	\begin{align*}
	\Ucx(h(a),h(b)) 
	&=
	\Ucx(h(\Dc(a,b)),h(\Uc(a,b)))
	\\&= 
	\Ucx(h(\Dc(c,d)),h(\Uc(c,d)))
	\\&= \Ucx(h(c),h(d)),
	\end{align*}
	and similarly $\Dcx(h(a),h(b))=\Dcx(h(c),h(d))$, as was to be shown.
	
	For the converse direction of the proof, suppose that the stated implication holds. Parallelness gives through~\cref{lemma:parallel in terms of sim} that ${x\tensor \top \sim \Dc(x,\top)\tensor\Uc(x,\top)}$ for every~$x\in L$, so applying $h$ we get
	\[
	h(x)\tensor \top \sim h(\Dc(x,\top))\tensor h(\Uc(x,\top)),
	\]
	which unpacked gives
	\begin{align*}
		\ucx h(x)
		&=: \Ucx(h(x),\top)
		\\&=
		\Ucx(h(\Dc(x,\top)),h(\Uc(x,\top)))
		\\&=
		\ucx h(\Dc(x,\top)) \wedge h(\Uc(x,\top))
		\\&\sqleq 
		h(\Uc(x,\top))
		\\&= 
		h(\uc x),
	\end{align*}
	giving $\ucx\circ h \sqleq h \circ \uc$, as desired. The dual inclusion is analogous.
\end{proof}

\section{Localic relations induced by cones}
\label{section:localic relation induced by cones}
In this section we fix a locale $X$, assume a join-preserving parallel pair~$\uc,\dc$ on its underlying frame $\Opens X$, and construct a localic relation~$R_\ucdc$ on~$X$. We show it is the universal one with cones~$\uc,\dc$.
The relation $R_\ucdc$ will be constructed via the spatial~\cref{intuition:sim} of the binary relation~$\sim$ from~\cref{definition:sim}. Recall that~$\sim$ is defined on the rectangles in~${\Opens X\tensor \Opens X}$~by:
\[
A\tensor B\sim C\tensor D
\quad\iff\quad
\Uc(A,B)=\Uc(C,D)
\text{ and }
\Dc(A,B)=\Dc(C,D).
\]
This relation can of course be defined for any pair of cones $\uc,\dc$ preserving the bottom element of $\Opens X$, and will hence define a localic relation as follows.

\begin{definition}
	\label{definition:induced relation}
	If $\uc,\dc\colon \Opens X\to \Opens X$ are a pair of cones preserving the bottom element, we define the \emph{induced relation} $R_\ucdc\rightarrowtail X\times X$ as the sublocale set
	\[
	\Opens R_\ucdc := \Opens X\tensor \Opens X/{\sim}.
	\]
	The sublocale inclusion $r_\ucdc\colon R_\ucdc\rightarrowtail X\times X$ is represented by the frame map~$\mu_\sim$, so the source and target maps are given by
	\[
	s_\ucdc^{-1}(U) = \mu_\sim(U\tensor \top)
	\quad\text{and}\quad
	t_\ucdc^{-1}(V)=\mu_\sim(\top\tensor V).
	\]
\end{definition}

The results in~\cref{section:conic frames} are then sufficient to prove that the induced relation again has open cones. For this, we directly construct left adjoints of the source and target maps, based on the spatial intuition from~\cref{section:relations to cones}:
\[
s_! \mu_\sim(U\tensor V) = U\wedge \dc V
\qquad\text{and}\qquad
t_! \mu_\sim(U\tensor V) = \uc U\wedge V.
\]
Motivated by this, we make the following definition.

\begin{definition}
	\label{definition:induced source and target left adjoints}
	For a pair of cones $\uc,\dc\colon \Opens X\to \Opens X$ on a locale $X$, define functions $s_!,t_!\colon \Opens R_\ucdc\to \Opens X$ by:
	\begin{align*}
		s_!(O)&:=\bigvee\{\Dc(A,B):\mu_\sim(A\tensor B)\sqleq O\},
		\\
		t_!(O)&:=\bigvee\{\Uc(A,B):\mu_\sim(A\tensor B)\sqleq O\}.
	\end{align*}
\end{definition}

The first point of order is to show this indeed reduces to the correct spatial intuition on rectangles. After that, we show that~$s_!,t_!$ provide the desired left adjoints of~$s_\ucdc^{-1},t_\ucdc^{-1}$, and in fact make~$s_\ucdc,t_\ucdc$ open.

\begin{lemma}
	\label{lemma:induced source and target on rectangles}
	If $\uc,\dc$ are join-preserving parallel, then for all $U,V\in\Opens X$ we have:
	\[
	s_! \mu_\sim(U\tensor V) = U\wedge \dc V
	\qquad\text{and}\qquad
	t_! \mu_\sim(U\tensor V) = \uc U\wedge V.
	\]
\end{lemma}
\begin{proof}
	Note that $\mu_\sim(U\tensor V)$ is included in itself, so from the \cref{definition:induced source and target left adjoints} of~$s_!$ we always get $\Dc(U,V)\sqleq s_!\mu_\sim(U\tensor V)$. For the converse inclusion, pick $A,B\in\Opens X$ such that $\mu_\sim(A\tensor B)\sqleq \mu_\sim(U\tensor V)$. By~\cref{lemma:mu inclusion implies reduced cone inclusion} this gives in particular $\Dc(A,B)\sqleq \Dc(U,V)$, from which the result follows.
\end{proof}

\begin{proposition}
	\label{proposition:induced source and target adjunction on rectangles}
	If $\uc,\dc$ are join-preserving parallel then $s_!\dashv s_\ucdc^{-1}$ and $t_! \dashv t_\ucdc^{-1}$.
\end{proposition}
\begin{proof}
	For $U\in \Opens X$ we get by construction that $s_\ucdc^{-1}(U) = \mu_\sim(U\tensor \top)$, which by~\cref{lemma:mu recovers separating opens} equals $O^\Dc_U$. We thus need to show for all $O\in\Opens R_\ucdc$ and~${U\in\Opens X}$ that
	\[
	s_!(O)\sqleq U 
	\qquad\iff\qquad
	O\sqleq O^\Dc_U.
	\]
	
	Suppose the former inclusion holds. Then consider a rectangle $\mu_\sim(A\tensor B)\sqleq O$, so we get using~\cref{lemma:induced source and target on rectangles} that $\Dc(A,B) = s_!\mu_\sim(A\tensor B) \sqleq s_!(O)\sqleq U$. Now separation implies $A\tensor B\sqleq O^\Dc_U$, and since the rectangles $\mu_\sim(A\tensor B)$ form a join-basis for~$\Opens R_\ucdc$, the inclusion $O\sqleq O^\Dc_U$ follows.
	
	For the converse direction, take $O\sqleq s_\ucdc^{-1}(U)=O^\Dc_U$. Now, $s_!(O)$ is defined as the join over the $\Dc(A,B)$ where $\mu_\sim(A\tensor B)\sqleq O$, so it suffices to show those are contained in $U$. Since $\mu_\sim$ is inflationary, for such rectangles we get $A\tensor B\sqleq \mu_\sim(A\tensor B)\sqleq O\sqleq O^\Dc_U$, and by separation this gives $\Dc(A,B)\sqleq U$, as desired.
\end{proof}

\begin{remark}
	Since by~\cref{lemma:mu recovers separating opens} we have that~$s_\ucdc^{-1}=O^\Dc_{(-)}$ and~$t_\ucdc^{-1} = O^\Uc_{(-)}$, it follows from uniqueness of adjoints that~$s_!=\Sigma$ and~$t_!=\Tau$ from~\cref{lemma:separating iff Sigma Tau}.
\end{remark}

\begin{proposition}
	\label{proposition:induced source and target frobenius}
	If $\uc,\dc$ are join-preserving parallel, then  $s_!\dashv s_\ucdc^{-1}$ and $t_! \dashv t_\ucdc^{-1}$ satisfy Frobenius reciprocity.
\end{proposition}
\begin{proof}
	We need to show for all $O\in\Opens R_\ucdc$ and $U,V\in\Opens X$ that
	\[
	s_!\left(O\wedge s_\ucdc^{-1}(U)\right) = s_!(O)\wedge U
	\quad\text{and}\quad
	t_!\left(O\wedge t_\ucdc^{-1}(V)\right) = t_!(O)\wedge V.
	\]
	Both sides preserve joins, so it suffices to check equality in the case that~$O$ is a rectangle~$\mu_\sim(A\tensor B)$. Now, $\mu_\sim$ preserves binary meets, so
	\[
	O\wedge s_\ucdc^{-1}(U) = \mu_\sim(A\tensor B)\wedge \mu_\sim(U\tensor \top) = \mu_\sim((A\wedge U)\tensor B).
	\]
	Applying $s_!$ on both sides and using~\cref{lemma:induced source and target on rectangles} gives the desired equation:
	\[
	s_!\mu_\sim((A\wedge U)\tensor B) = \Dc(A\wedge U,B) = \Dc(A,B)\wedge U = s_!\mu_\sim(A\tensor B)\wedge U.
	\qedhere
	\]
\end{proof}

Putting it all together, we can state one of the central theorems.

\begin{theorem}
	\label{theorem:induced relation has open cones}
	If $\uc,\dc$ are a join-preserving parallel pair of cones, the relation~$R_\ucdc$ has open cones with $\Up = \uc$ and $\Down = \dc$.
\end{theorem}
\begin{proof}
	That $R_\ucdc$ has open cones follows by~\cref{proposition:induced source and target adjunction on rectangles,proposition:induced source and target frobenius}. That its cones in turn reproduce $\uc,\dc$ follows simply from~\cref{lemma:induced source and target on rectangles}:
	\[
	\Up U := t_! s_\ucdc^{-1}(U) = t_!\mu_\sim(U\tensor\top) = \uc U \wedge \top = \uc U,
	\]
	and similarly for $\Down U=\dc U$.
\end{proof}

\subsection{Relation induced by $\Up,\Down$}
\label{section:relation induced by UpDown}
To finish this section, we prove that an open cone relation~$R$ is always contained in the relation $R_\UpDown$ induced by its own cones~$\Up,\Down$ via the construction in~\cref{definition:induced relation}.

We establish notation for the proofs in the rest of this section. Fix $R$ on~$X$, and write its sublocale inclusion by $r$, with open source and target maps $s,t$. The induced cones $\Up,\Down$ are join-preserving parallel by~\cref{proposition:induced cones preserve joins,proposition:induced cones are parallel}, so they define a conic frame~$(\Opens X,\Up,\Down)$. Let $\sim$ denote their induced meet-semilattice congruence on the rectangles of $\Opens X\tensor \Opens X$~(\cref{definition:sim}), where we denote the reduced cones as~$\UpR,\DownR$~(\cref{definition:reduced cones}). By~\cref{theorem:induced relation has open cones} we get an induced open cone relation~$R_\UpDown$ whose sublocale inclusion is given by the frame quotient map~$\mu_\sim$. The strategy is to prove that the congruence of $r^{-1}$ agrees with~$\sim$ on rectangles, from which the claim~$R\subseteq R_\UpDown$ will follow. 

\begin{lemma}
	\label{lemma:reduced induced cones from s t r}
	We have $s_!r^{-1}(U\tensor V)= \DownR(U,V)$ and $t_!r^{-1}(U\tensor V)= \UpR(U,V)$.
\end{lemma}
\begin{proof}
	This follows straightforwardly from Frobenius reciprocity:
	\[
	s_! r^{-1}(U\tensor V) = s_!\left(s^{-1}(U)\wedge t^{-1}(V)\right)
	= U\wedge s_! t^{-1}(V) =\DownR(U,V).
	\qedhere
	\]
\end{proof}

\begin{lemma}
	\label{lemma:r equalises reduced form}
	We have $r^{-1}(U\tensor V) = r^{-1}\left(\DownR(U,V)\tensor \UpR(U,V)\right)$. 
\end{lemma}
\begin{proof}
	Note that $s_!\dashv s^{-1}$ and $t_!\dashv t^{-1}$ imply $O\sqleq s^{-1}s_!(O)$ and ${O\sqleq t^{-1}t_!(O)}$ for all $O\in\Opens R$, so we get in particular that~$t^{-1}(V)\sqleq s^{-1}s_!t^{-1}(V) = s^{-1}(\Down V)$, and similarly $s^{-1}(U)\sqleq t^{-1}(\Up U)$. Writing $t^{-1}(V) = t^{-1}(V) \wedge s^{-1}(\Down V)$ we thus get
	\begin{align*}
		r^{-1}(U\tensor V)
		&=
		s^{-1}(U) \wedge t^{-1}(V)
		&&(\text{\cref{lemma:coproduct of frame maps}})
		\\&=
		s^{-1}(U) \wedge t^{-1}(V) \wedge s^{-1}(\Down V)
		\\&= 
		s^{-1}\left(U\wedge \Down V\right)\wedge t^{-1}(V)
		\\&=
		s^{-1}\DownR(U,V)\wedge t^{-1}(V)
		\\&=
		r^{-1}(\DownR(U,V)\tensor V),
		&&(\text{\cref{lemma:coproduct of frame maps}})
	\end{align*}
	and with an analogous calculation in the second component the claim follows.
\end{proof}

\begin{lemma}
	\label{lemma:sim of induced cones is ker r}
	We have
	\[
	r^{-1}(A\tensor B)=r^{-1}(C\tensor D)
	\quad\iff\quad
	A\tensor B\sim C\tensor D.
	\]
\end{lemma}
\begin{proof}
	Starting with the left-hand side and applying $s_!,t_!$ to the equation gives the right hand side by~\cref{lemma:reduced induced cones from s t r}.
	
	Conversely, $A\tensor B\sim C\tensor D$ iff $\UpR(A,B)=\UpR(C,D)$ and~${\DownR(A,B)=\DownR(C,D)}$, so we get that $r^{-1}(\DownR(A,B)\tensor \UpR(A,B))=r^{-1}(\DownR(C,D)\tensor \UpR(C,D))$, which gives the desired equality by~\cref{lemma:r equalises reduced form}.
\end{proof}

\begin{proposition}
	\label{proposition:R in Rupdown}
	If $R$ is a localic relation with open cones then $R\subseteq R_\UpDown$.
\end{proposition}
\begin{proof}
	By~\cref{lemma:sim of induced cones is ker r} the frame map $r^{-1}\colon \Opens X\tensor \Opens X\to \Opens R$ equalises the relation~$\sim$, so by the frame quotient~\cref{theorem:frame quotient theorem} there exists a unique frame map $i^{-1}\colon \Opens X\tensor \Opens X/{\sim}\to \Opens R$ such that $i^{-1}\circ \mu_\sim = r^{-1}$, which exhibits the desired inclusion $i\colon R\rightarrowtail R_\UpDown$.
\end{proof}

Despite that the frame congruence of $R$ agrees with~$\sim$ on rectangles (as in \cref{lemma:sim of induced cones is ker r}), it is not necessarily the case that they agree globally. In other words, there are relations $R$ that are not recovered as $R_\UpDown$, as we will see in~\cref{counterexample:relation on Sierpinski}. Nevertheless, we will show that~$R_\UpDown$ cannot deviate from~$R$ too much, since we show in~\cref{proposition:R in Rupdown strongly dense} that the inclusion $R\subseteq R_\UpDown$ is strongly dense.

\section{An adjunction}
\label{section:adjunction}
In this section we put together the constructions~$R\mapsto \Up,\Down$ and~$\uc,\dc\mapsto R_\ucdc$ and establish an adjunction~${\cone\dashv\rel}$ between conic frames and open cone localic relations. The important characterisation in~\cref{proposition:conic morphism iff preserves sim} that morphisms between conic frames are really those that preserve the generating relations~$\sim$ then ensures the notion of monotonicity agrees in both categories. In terms of functoriality there is nothing to check, since on morphisms we are really just taking the opposite category. \cref{theorem:induced relation has open cones} shows that the composition $\uc,\dc\mapsto R_\ucdc\mapsto \Up,\Down$ is the identity map, giving at the level of functors that~$\cone\circ\rel=\id$, so the counit of the adjunction will be the identity natural transformation. From this we deduce that $R_\ucdc$ is the universal open cone relation with the cones~$\uc,\dc$~(\cref{corollary:Rucdc universal with cones}). Fixed points of the adjunction will be discussed in the next~\cref{section:fixed points}.

\begin{proposition}
	\label{proposition:functor ucdcFrm to rocLoc}
	There is a functor
	\begin{align*}
		\rel\colon \ucdcFrm^\op &\longrightarrow \rocLoc;
		\\
		(\Opens X,\uc,\dc)&\longmapsto (X, R_\ucdc);
		\\
		f^{-1} &\longmapsto f.
	\end{align*}
\end{proposition}
\begin{proof}
	That this is well-defined on objects is~\cref{theorem:induced relation has open cones}. Now take a map of locales $f\colon X\to Y$, where $\Opens X$ and $\Opens Y$ are equipped with join-preserving parallel cones such that $f^{-1}$ is a conic morphism, and denote the induced relations by~$R_\ucdc$ and~$Q_\ucdc$, respectively. We need to construct a frame map $\bar f^{-1}\colon \Opens Q_\ucdc\to \Opens R_\ucdc$ such that
	\[
	\bar f^{-1}\circ q_\ucdc^{-1} = r_\ucdc^{-1}\circ (f\times f)^{-1}.
	\]
	That $f^{-1}$ is a conic morphism means equivalently by~\cref{proposition:conic morphism iff preserves sim} that
	\[
	A\tensor B\sim C\tensor D
	\quad\implies\quad 
	f^{-1}(A)\tensor f^{-1}(B) \sim f^{-1}(C)\tensor f^{-1}(D),
	\]
	and the right hand side is in turn equalised by $r_\ucdc^{-1}$. Hence applying the frame quotient~\cref{theorem:frame quotient theorem} to the frame map $r_\ucdc^{-1}\circ (f\times f)^{-1}$, we get a unique frame map~$\bar f^{-1}$ satisfying the desired equality.
\end{proof}

\begin{proposition}
	\label{proposition:functor rocLoc to ucdcFrm}
	There is a functor
	\begin{align*}
		\cone\colon \rocLoc &\longrightarrow \ucdcFrm^\op;
		\\
		(X,R)&\longmapsto (\Opens X,\Up,\Down);
		\\
		f &\longmapsto f^{-1}.
	\end{align*}
\end{proposition}
\begin{proof}
	That this is well-defined on objects follows from \cref{proposition:induced cones preserve joins,proposition:induced cones are parallel}. We are left to show that if $f\colon (X,R)\to (Y,Q)$ is internally monotone, then its underlying frame map is a conic morphism. Monotonicity gives~${\bar f\colon R\to Q}$ that at the level of frames satisfies $r^{-1} \circ(f\times f)^{-1} = \bar f^{-1} \circ q^{-1}$. If $A,B,C,D\in \Opens Y$ satisfy $A\tensor B\sim C\tensor D$, this implies $q^{-1}(A\tensor B)=q^{-1}(C\tensor D)$, and hence applying $\bar f^{-1}$ on both sides we get
	\[
	r^{-1}\left(f^{-1}(A)\tensor f^{-1}(B)\right) = r^{-1}\left(f^{-1}(C)\tensor f^{-1}(D)\right).
	\]
	Since $R$ has open cones, by~\cref{lemma:sim of induced cones is ker r} this equality holds precisely when we have $f^{-1}(A)\tensor f^{-1}(B)\sim f^{-1}(C)\tensor f^{-1}(D)$, and so we see that~$f^{-1}$ is a conic morphism from~\cref{proposition:conic morphism iff preserves sim}.
\end{proof}

\begin{theorem}
	\label{theorem:adjunction}
	There is an adjunction with identity counit
	\[
	\begin{tikzcd}[cramped,column sep=1cm]
		\rocLoc & {\ucdcFrm^\op}.
		\arrow[""{name=0, anchor=center, inner sep=0}, "{\cone}", shift left=1.6, from=1-1, to=1-2]
		\arrow[""{name=1, anchor=center, inner sep=0}, "{\rel}", shift left=1.6, from=1-2, to=1-1]
		\arrow["\dashv"{anchor=center, rotate=-90}, draw=none, from=0, to=1]
	\end{tikzcd}
	\]
\end{theorem}
\begin{proof}
	By \cref{theorem:induced relation has open cones} we have for every conic frame $(\Opens X,\uc,\dc)$ that
	\[
	\cone\circ \rel(\Opens X,\uc,\dc)
	=
	\cone(X,R_\ucdc)
	=
	(\Opens X,\uc,\dc),
	\]
	so there is a strict equality $\cone\circ \rel=\id$, and the counit $\epsilon \colon \cone\circ \rel\to \id$ is thus taken to be the identity natural transformation.
	
	For the unit transformation, if $(X,R)$ is a locale equipped with an open cone relation, then $\rel\circ \cone(X,R) = (X,R_\UpDown)$. By \cref{proposition:R in Rupdown} there is a sublocale inclusion $R\subseteq R_\UpDown$, and so by~\cref{lemma:monotone map modification} the identity map ${\id_X\colon (X,R)\to (X,R_\UpDown)}$ is internally monotone. We can thus define the unit ${\eta\colon \id \to \rel\circ \cone}$ simply by ${\eta_{(X,R)}:=\id_X}$, which is clearly natural.
	
	We are left to show that $\eta,\epsilon$ satisfy the triangle identities:
	\[
	\epsilon_\cone \circ \cone(\eta) = \id_\cone
	\quad\text{and}\quad
	\rel(\epsilon)\circ \eta_\rel = \id_\rel.
	\]
	For a locale equipped with an open cone relation $(X,R)$ we have by definition that~$\epsilon_{\cone(X,R)} = \id_{(\Opens X,\Up,\Down)}$, and~$\cone(\eta_{(X,R)}) = \id_{\Opens X}$ is also the identity conic morphism on $(\Opens X,\Up,\Down)$ since $\cone(X,R_\UpDown) = (\Opens X,\Up,\Down)$ by~\cref{theorem:induced relation has open cones}.
	
	For the second equation, if $(\Opens X,\uc,\dc)$ is a conic frame then~$\epsilon_{(\Opens X,\uc,\dc)}$ is again the identity, which is preserved by the functor $\rel$. Once more applying~\cref{theorem:induced relation has open cones}, we see that $\rel(\Opens X,\uc,\dc)=(X,R_\ucdc)$ induces the cones $\Up=\uc$ and $\Down=\dc$, and so:
	\[
	(R_\ucdc)_\UpDown = R_\ucdc.
	\]
	Hence $\eta_{\rel(\Opens X,\uc,\dc)}= \eta_{(X,R_\ucdc)}$ is simply the identity arrow $\id_{(X,R_\ucdc)}$, and the adjunction $\cone\dashv\rel$ follows.
\end{proof}

\begin{remark}
	That the counit of this adjunction is strictly equal to the identity implies immediately several nice facts about the adjunction, for instance that it is idempotent and monadic~\cite[\S IV.3]{maclane1998CategoriesWorkingMathematician}. In particular, we can think of~$\ucdcFrm^\op$ as a strict reflective subcategory of~$\rocLoc$.
\end{remark}

An immediate consequence of $\cone\dashv\rel$ is that $R_\ucdc$ is the least upper bound of open cone relations whose induced cones are contained in $\uc,\dc$.

\begin{corollary}
	\label{corollary:Rucdc universal with cones}
	We have $R\subseteq R_\ucdc$ iff $\Up\sqleq \uc$ and $\Down\sqleq \dc$.
\end{corollary}
\begin{proof}
	The inclusion $R\subseteq R_\ucdc$ holds iff $\id_X\colon (X,R)\to (X,R_\ucdc)=\rel(\Opens X,\uc,\dc)$ is monotone, and by $\cone\dashv \rel$ this holds in turn iff $\epsilon_{(\Opens X,\uc,\dc)}\circ \cone(\id_X) = \id_{\Opens X}\colon (\Opens X,\uc,\dc)\to (\Opens X,\Up,\Down)$ is a conic morphism. By~\cref{lemma:id ucdc map iff cones inclusion} this is equivalent to $\Up\sqleq \uc$ and $\Down\sqleq \dc$.
\end{proof}

\section{Fixed points}
\label{section:fixed points}
Recall that the \emph{fixed points} of an adjunction are those objects in the respective categories for which the unit~$\eta$ and counit~$\epsilon$ are isomorphisms. Since the adjunction~$\cone\dashv\rel$ from \cref{theorem:adjunction} has~$\epsilon = \id$, all the conic frames are fixed points. Indeed, all conic frames are recovered from the cones of their induced relation. We start this section by outlining some basic (counter)examples of the fixed points in~$\rocLoc$. After that, in~\cref{section:closed relations,section:strong density} we prove that the inclusion~$R\subseteq R_\UpDown$ from~\cref{proposition:R in Rupdown} is strongly dense. As an immediate consequence, any (weakly) closed localic relation with open cones defines a fixed point, and we obtain as a corollary Kock's Godement theorem on effective localic equivalence relations~\cite{kock1989GodementTheoremLocales}.

Since the functors~$\cone$ and~$\rel$ do not alter the underlying locale, an object~$(X,R)$ in~$\rocLoc$ being fixed point should only depend on~$R$. We can make this precise as follows.

\begin{corollary}
	The function $(-)_\UpDown$ is a closure operator on open cone relations.
\end{corollary}
\begin{proof}
	This follows since it can be identified with the monad $\rel\circ \cone$ induced by the adjunction, which sends $(X,R) \mapsto (X,R_\UpDown)$.
\end{proof}

\begin{corollary}
	The fixed points of $\cone\dashv \rel$ are those $(X,R)\in\rocLoc$ such that $R\cong R_\UpDown$. 
\end{corollary}

With that said, we discuss some basic (counter)examples.

\begin{example}
	\label{example:diagonal relation is fixed point}
	We show that any diagonal relation $\Delta$ is a fixed point of the adjunction. Recall from \cref{example:induced cones of diagonal} that the diagonal relation on $X$ induces the cones $\Up = \Down = \id_{\Opens X}$. Let $R_{\id\!\id}$ denote the relation induced by the identity cones. Then by~\cref{proposition:R in Rupdown} we get $\Delta\subseteq R_{\id\!\id}$. Explicitly, this inclusion is witnessed via the frame quotient~\cref{theorem:frame quotient theorem}, noting that $\delta^{-1}$ equalises the relation $\sim$:
	\[
	A\tensor B\sim C\tensor D
	\iff
	A\wedge B= C\wedge D
	\iff 
	\delta^{-1}(A\tensor B)= \delta^{-1}(C\tensor D).
	\]
	Thus we get a map of frames $d^{-1}\colon \Opens R_{\id\!\id}\to \Opens X$ such that $d^{-1}\circ \mu_\sim = \delta^{-1}$. We claim that this is an isomorphism of frames, with inverse given by ${i^{-1}:=\mu_\sim \circ \iota_1}$, where~$\iota_1$ is the first frame coproduct injection~(\cref{lemma:coproduct of frame maps}). First simply calculate:
	\[
	d^{-1} i^{-1}(U) 
	=
	d^{-1}\mu_\sim \iota_1(U)
	=
	\delta^{-1}(U\tensor\top)
	=
	U\wedge \top = U.
	\]
	Second, $\mu_\sim$ is the frame map of a sublocale inclusion, and is hence surjective. To show $i^{-1}\circ d^{-1}$ is the identity it therefore suffices to check on rectangles $\mu_\sim(U\tensor V)$, which gives:
	\[
	i^{-1}d^{-1}\mu_\sim(U\tensor V)
	=
	i^{-1}\delta^{-1}(U\tensor V)
	=
	\mu_\sim\left((U\wedge V)\tensor \top\right).
	\]
	But now it is easy to see $(U\wedge V)\tensor \top \sim U\tensor V$, so the right hand side in fact equals~$\mu_\sim(U\tensor V)$, as desired. Moreover, using the same observation, we see $i^{-1}\circ \delta^{-1}(U\tensor V) = \mu_\sim((U\wedge V)\tensor \top)=\mu_\sim(U\tensor V)$. Thus the corresponding locale map $i$ defines an isomorphism of subobjects $\Delta\cong R_{\id\!\id}$, as desired.
\end{example}

\begin{counterexample}
	\label{counterexample:relation on Sierpinski}
	We construct a localic relation with open cones that is not a fixed point of the adjunction. The base space is the \emph{Sierpi\'{n}ski locale}~$\mathbb{S}$, defined by~$\Opens \mathbb{S} = \{\bot<a < \top\}$. It is a spatial locale. Now the coproduct frame~$\Opens \mathbb{S}\tensor \Opens \mathbb{S}$ has six elements:

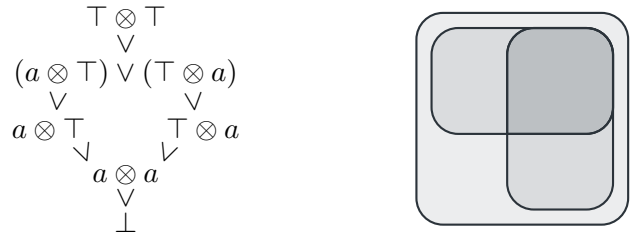
\[
\begin{tikzpicture}[
	baseline={(current bounding box.center)},
	every node/.style={inner sep=.4pt, outer sep=0pt}
	]
	\node (bot) at (0,0) {$\bot$};
	\node (aa) at (0,.62) {$a\tensor a$};
	\node (at) at (-1.02,1.22) {$a\tensor \top$};
	\node (ta) at (1.02,1.22) {$\top\tensor a$};
	\node (join) at (0,1.95) {$(a\tensor \top) \vee (\top \tensor a)$};
	\node (top) at (0,2.7) {$\top\tensor \top$};
	
	\node[rotate=90] at (0,.31) {$<$};
	\node[rotate=110] at (-.55,.91) {$<$};
	\node[rotate=70] at (.55,.91) {$<$};
	\node[rotate=85] at (-.90,1.59) {$<$};
	\node[rotate=95] at (.90,1.59) {$<$};
	\node[rotate=90] at (0,2.35) {$<$};
\end{tikzpicture}
\qquad\qquad\qquad
\begin{tikzpicture}[baseline={(current bounding box.center)}]
	
	\def\s{1}        
	\def\m{0.6}        
	\def\r{10pt}        
	
	\def\wa{0.7}
	
	\def\pT{0.7}
	\def\pTT{0.9}
	
	\tikzset{
		blob/.style={
			draw=figcolor!70!black,
			thick,
			rounded corners=\r,
			fill=figcolor,
			fill opacity=.1
		}
	}
	
	\coordinate (p00) at (\m,\m);
	\coordinate (p10) at (\m+\s,\m);
	\coordinate (p01) at (\m,\m+\s);
	\coordinate (p11) at (\m+\s,\m+\s);
	
	\path let \p1=(p00), \p2=(p10), \p3=(p01), \p4=(p11) in
	coordinate (x0) at (\x1,0)
	coordinate (x1) at (\x2,0)
	coordinate (y0) at (0,\y1)
	coordinate (y1) at (0,\y3);
	
	
	\draw[blob]
	($(p00)+(-\pTT,-\pTT)$)
	rectangle
	($(p11)+(\pTT,\pTT)$);
	
	\draw[blob]
	($(p10)+(-\wa,-\pT)$)
	rectangle
	($(p11)+(\wa,\pT)$);
	
	\draw[blob]
	($(p01)+(-\pT,-\wa)$)
	rectangle
	($(p11)+(\pT,\wa)$);
	
	\draw[blob]
	($(p11)+(-\wa,-\wa)$)
	--
	($(p11)+(-\wa,\wa)$)
	--
	($(p11)+(\wa,\wa)$)
	--
	($(p11)+(\wa,-\wa)$)
	--
	($(p11)+(-\wa,-\wa)$);
	
\end{tikzpicture}
\]
	
	\smallskip\noindent
	Define $R\rightarrowtail\mathbb{S}\times\mathbb{S}$ as the open sublocale induced by $u:=(a\tensor \top)\vee (\top\tensor a)$. By~\cref{example:open sublocale in overt X has open cones}, $R$ has open cones. We compute them explicitly.
	
	First recall that the frame of the open sublocale can be presented by the downset $\down u = \{o\in \Opens\mathbb{S}\tensor\Opens\mathbb{S}: o\sqleq u\}$, and the underlying frame map of the sublocale inclusion is just intersection with $u$~\cite[\S III.6]{picado2012FramesLocalesTopology}:
	\[
	r^{-1}\colon \Opens\mathbb{S}\tensor\Opens\mathbb{S}\longrightarrow \down u;
	\qquad
	o\longmapsto o\wedge u.
	\]
	This makes the source and target maps easy to compute: $s^{-1}(x) = (x\tensor \top)\wedge u$ and $t^{-1}(x)=(\top\tensor x)\wedge u$, which results in:
	
	\renewcommand{\arraystretch}{1.2} 
	\hfil
	\begin{tabular}{r|c|c|c}
		$x$ & $\bot$ & $a$ & $	\top$ \\ \hline\hline
		$s^{-1}(x)$ & $\bot$ & $a\tensor \top$ & $u$ \\ \hline
		$t^{-1}(x)$ & $\bot$ & $\top\tensor a$ & $u$ 
	\end{tabular}
	\hfil
	
	\noindent The left adjoints $s_!,t_!$ are also straightforwardly calculated. By $s_!\dashv s^{-1}$, we know that $s_!(o)$ is equal to the smallest element $x\in \Opens\mathbb{S}$ such that $o\sqleq s^{-1}(x)$, so with the help of the previous table we get:
	
	\hfil\begin{tabular}{r|c|c|c|c|c}
		$o$ & $\bot$ & $a \tensor a$ & $a\tensor \top$ & $\top \tensor a$ & $u$ \\ \hline\hline
		$s_!(o)$ & $\bot$ & $a$ & $a$ & $\top$ & $\top$ \\ \hline
		$t_!(o)$ & $\bot$ & $a$ & $\top$ & $a$ & $\top$
	\end{tabular}\hfil
	
	\smallskip
	\noindent Putting it together, the induced (reduced) cones are:
	\smallskip
	
	\hfil\begin{tabular}{r|c|c|c}
		$x$ & $\bot$ & $a$ & $\top$ \\ \hline\hline
		$\Up x$ & $\bot$ & $\top$ & $\top$ \\ \hline
		$\Down x$ & $\bot$ & $\top$ & $\top$ 
	\end{tabular}\hfil
	\begin{tabular}{r|c|c|c|c}
		$(x,y)$ & $(a,a)$ & $(a,\top)$ & $(\top,a)$ & $(\top,\top)$ \\ \hline\hline
		$\UpR(x,y)$ & $a$ & $\top$ & $a$ & $\top$ \\ \hline
		$\DownR(x,y)$ & $a$ & $a$ & $\top$ & $\top$
	\end{tabular}\hfil
	
	\smallskip\noindent In the latter table we have omitted pairs $(x,y)$ where $x$ or $y$ is~$\bot$, which would trivially evaluate to the bottom element. From this table we can then see that ${x\tensor y\sim c\tensor d}$ can only hold if $x\tensor y = c\tensor d$. Hence the frame congruence generated by~$\sim$ is just the equality relation on~$\Opens\mathbb{S}\tensor\Opens\mathbb{S}$, and so the induced relation~$R_\UpDown$ is the trivial relation $\mathbb{S}\times \mathbb{S}$, which is strictly bigger than the relation~$R$ we started with. Since all locales involved here are spatial, this also shows that open cone localic relations coming from topological spaces can fail to be fixed points.
\end{counterexample}

\begin{example}
	Recall from~\cref{example:open sublocale in overt X has open cones} that any open localic relation on an overt locale~$X$ has open cones. We show that open rectangles are always fixed points.
	
	From overtness we get a left adjoint~$\pos_X\dashv \terminal_X^{-1}$ of the terminal map~${\terminal_X\colon X\to 1}$, and similarly denote the left adjoints of the product projection maps by $\exists_i\dashv \pr_i^{-1}$. For an open $E\in\Opens (X\times X)$, seen as a localic relation, we get induced cones
	\[
	\Up U = \exists_2 (E\wedge U\tensor \top)
	\quad\text{and}\quad
	\Down U = \exists_1(E\wedge \top\tensor U).
	\]
	The relation $E$ defines a fixed point precisely when~$E\langle\sim\rangle\top\tensor \top$, which does not readily unpack to a more concrete condition. However, in the particular case that~$E$ is a rectangle $A\tensor B$ for some $A,B\in\Opens X$, we get
	\[
	\Up U = B\wedge \pos_X(A\wedge U)
	\quad\text{and}\quad
	\Down U = A\wedge \pos_X(B\wedge U),
	\]
	and by~\cref{corollary:mu implies sim} this defines a fixed point iff $A\tensor B\sim \top\tensor \top$. Calculating
	\[
	A\wedge \Down B =A\wedge A\wedge \pos_X(B\wedge B)= A\wedge\pos_X(B) =\top\wedge \Down \top,
	\]
	and similarly $\Up A\wedge B = \Up \top\wedge \top$, we see this always holds.
	
	Since the relation in~\cref{counterexample:relation on Sierpinski} is a finite union of open rectangles, this shows that fixed points of~$\cone\dashv\rel$ are not closed under finite joins of sublocales.
\end{example}

A general characterisation of the fixed points of~${\cone\dashv\rel}$ is unknown. Clearly, a necessary and sufficient condition for $R\cong R_\UpDown$ cannot be expressed solely as properties of the cones~$\Up,\Down$, since~$R$ and~$R_\UpDown$ share cones but may not both be fixed points. We record the following simple rewriting of the inclusion~$R_\UpDown\subseteq R$, making the fixed point condition somewhat more explicit.

\begin{proposition}
	We have $R\cong R_\UpDown$ iff for every $\sim$-saturated $o\in\Opens X\tensor \Opens X$:
	\[
	r^{-1}(x\tensor y)\sqleq r^{-1}(o)
	\quad\implies\quad
	x\tensor y\sqleq o.
	\]
\end{proposition}
\begin{proof}
	We already get a sublocale inclusion $R\subseteq R_\UpDown$ from~\cref{proposition:R in Rupdown}, so that~$R\cong R_\UpDown$ is equivalent to the inclusion~$R_\UpDown\subseteq R$. Using~$r^{-1}\dashv r_\ast$ together with the fact that $\sim$-saturated opens are precisely the fixed points of~$\mu_\sim$, we see that the stated condition is really equivalent to the inclusion $r_\ast r^{-1}\sqleq \mu_\sim$ of nuclei, which is precisely the desired sublocale inclusion.
\end{proof}

\subsection{Kernel pairs}
Recall from~\cref{example:open kernel pairs} that the kernel pair~$X\times_Q X$ of an open map~${q\colon X\to Q}$ defines an open cone localic equivalence relation. It is always a fixed point. In order to prove this we collect the following important result from the literature.

\begin{lemma}
	\label{lemma:open map pullback square}
	Suppose $Y\times_X^{f_1,f_2}Z$ is the pullback of $f_1,f_2$, with projections~$\pi_1,\pi_2$. If~$f_1$ is open then~$\pi_2$ is open. If~$f_2$ is open then~$\pi_1$ is open. If $f_1,f_2$ are open then
	\[
	f_1^{-1}\circ (f_2)_! = (\pi_1)_!\circ \pi_2^{-1}.
	\]
\end{lemma}
\begin{proof}
	This is \cite[Proposition~V.4.1]{joyal1984ExtensionGaloisTheory} or \cite[Proposition~C3.1.11]{johnstone2002Elephant2}.
\end{proof}

\begin{proposition}
	\label{proposition:kernel pair is fixed point}
	If~$q\colon X\to Q$ is an open map of locales, $X\times_Q X$ is a fixed~point.
\end{proposition}
\begin{proof}
	By definition, the kernel pair $k\colon X\times_Q X\rightarrowtail X\times X$ has source and target maps~$s,t$ defined by the pullback square~$q\circ s = q\circ t$, so it follows by~\cref{lemma:open map pullback square} that~${\Up = q^{-1}q_! =\Down}$. Denote by~$r_\UpDown\colon R_\UpDown\rightarrowtail X\times X$ the relation induced by these cones. The adjunction~$q_!\dashv q^{-1}$ gives~$\Up q^{-1}(V) = q^{-1}q_!q^{-1}(V) = q^{-1}(V)$ for any~$V\in\Opens Q$. Similarly, $\Up \top = \top$. It follows that
	\[
	q^{-1}(V)\tensor \top \sim q^{-1}(V)\tensor q^{-1}(V)\sim \top \tensor q^{-1}(V),
	\]
	which gets equalised under $r_\UpDown^{-1}=\mu_\sim$. This means $q\circ \pr_1\circ r_\UpDown = q\circ \pr_2\circ r_\UpDown$, so by universality of the pullback we get a sublocale inclusion~$R_\UpDown\subseteq  X\times_Q X$.
	
	For the converse inclusion, first observe that the cones~$\Up=\Down$ are inflationary via the unit of the adjunction~$q_!\dashv q^{-1}$. Moreover, by the pullback square~$q\circ s = q\circ t$ we get that for every~$U\in\Opens X$:
	\[
	s^{-1}(\Up U) = s^{-1}q^{-1}q_!(U) = t^{-1}q^{-1}q_!(U) = t^{-1}(\Up U).
	\]
	Using this it follows that for every $U,V\in\Opens X$:
	\begin{align*}
		k^{-1}(U\tensor V)
		&=
		s^{-1}(U)\wedge t^{-1}(V)
		&&(\text{\cref{definition:source and target}})
		\\&=
		s^{-1}(U\wedge \Up U) \wedge t^{-1}(V\wedge \Up V)
		&&(\text{$\Up$ inflationary})
		\\&=
		s^{-1}(U)\wedge s^{-1}(\Up U)\wedge t^{-1}(V)\wedge t^{-1}(\Up V)
		&&(\text{$s^{-1},t^{-1}$ preserve $\wedge$})
		\\&=
		s^{-1}(U)\wedge t^{-1}(\Up U)\wedge t^{-1}(V)\wedge s^{-1}(\Up V)
		&&(\text{previous equation})
		\\&=
		s^{-1}(U\wedge \Up V)\wedge t^{-1}(V\wedge \Up U)
		&&(\text{$s^{-1},t^{-1}$ preserve $\wedge$})
		\\&=
		k^{-1}((U\wedge \Up V)\tensor (V\wedge \Up U)).
		&&(\text{\cref{definition:source and target}})
	\end{align*}
	This implies via~\cref{lemma:parallel in terms of sim} that~$k^{-1}$ equalises the generating relation~$\sim$ induced by~$\Up=\Down$, so by the frame quotient~\cref{theorem:frame quotient theorem} we get a sublocale inclusion ${X\times_Q X \subseteq R_\UpDown}$, and the result follows.
\end{proof}

Not only that, but if $R$ is an equivalence relation~(formally defined according to~\cref{definition:localic relation properties}) with open cones, the induced relation~$R_\UpDown$ is actually just its kernel pair~$X\times_{X/R}X$.

\begin{proposition}
	\label{proposition:kernel pair of equivalence relation R is Rupdown}
	For an open cone localic equivalence relation,~${R_\UpDown \cong X\times_{X/R}X}$.
\end{proposition}
\begin{proof}
	We saw in the proof of~\cref{proposition:kernel pair is fixed point} that the cones of the kernel pair of $q\colon X\tworightarrow X/R$ are just~$q^{-1}q_!$. Here specifically,~$q$ is defined as the coequaliser of the source and target maps $s,t$ of $R$ itself, so by~\cite[\S IV.3.2]{picado2012FramesLocalesTopology} we can identify the frame of~$X/R$ by the subframe~$\Opens X/R = \{ U\in\Opens X: s^{-1}(U)=t^{-1}(U)\}$ of~$\Opens X$, where~$q^{-1}$ is just the subframe inclusion.
	
	Borrowing the result~\cref{proposition:relation property implies cone property} from below, we find that the cones of~$R$ satisfy~$\Up=\Down$, are inflationary, and are subidempotent. We claim that $\Opens X/R$ is just the image of~$\Up$. Taking~$V\in \Opens X/R$, we get~$s^{-1}(V) = t^{-1}(V)$, so by~${t_!\dashv t^{-1}}$ we get~$\Up V \sqleq V$, which with inflation implies~$V=\Up V$. Moreover, for any $U\in\Opens X$ we get from the inclusion~$\Up^2 U\sqleq \Up U$ and the adjunctions~$s_!\dashv s^{-1}$,~$t_!\dashv t^{-1}$ that~$s^{-1}(\Up U)=t^{-1}(\Up U)$, which shows~$\Up$ lands in~$\Opens X/R$. The claim follows. Moreover, if $U\sqleq V\in \Opens X/R$ then $U\sqleq \Up U \sqleq \Up V = V$, so $\Up U$ is the unique least element in~$\Opens X/R$ that contains~$U$.
	
	On the other hand, $q^{-1}q_!(U)$ is also an element in $\Opens X/R$ containing $U$ via the unit of~$q_!\dashv q^{-1}$. Moreover, if $U\sqleq V\in \Opens X/R$, then~$V$ is actually shorthand for $q^{-1}(V)$, and we get~${q^{-1}q_!(U)\sqleq q^{-1}q_!q^{-1}(V)= q^{-1}(V)}$. By uniqueness, it follows that~$\Up = q^{-1}q_!$. Hence $R$ and its kernel pair $X\times_{X/R}X$ share the same cones, and using~\cref{proposition:kernel pair is fixed point} it follows that $R_\UpDown\cong X\times_{X/R}X$.	
\end{proof}

\begin{remark}
	\label{remark:kernel pair fixed point}
	For any localic equivalence relation $R$ with open cones, the strongly dense inclusion $R\subseteq X\times_{X/R}X$ into its kernel pair from~\cite{kock1989GodementTheoremLocales} therefore agrees with the inclusion $R\subseteq R_\UpDown$ from~\cref{proposition:R in Rupdown}. As mentioned, we will prove in~\cref{proposition:R in Rupdown strongly dense} that the latter inclusion is also strongly dense in the general setting. Any equivalence relation~$R$ that is not effective, for example in~\cite[p.468]{kock1989GodementTheoremLocales}, therefore provides an example of a non-fixed point~${R\not\cong R_\UpDown}$.
\end{remark}

\subsection{Closed relations}
\label{section:closed relations}
Closed preorders are important in topological order theory~\cite{nachbin1965TopologyOrder}, and closed localic partial orders are used in for example the work of Townsend~\cite{townsend1996preframeTechniquesConstructiveLocale,townsend1997LocalicPriestleyDuality} to give a localic Priestley duality. In this section we show that the class of closed localic relations with open cones define fixed points of~$\cone\dashv\rel$.

\begin{remark}
	\label{remark:closed sublocales}
	Recall from~\cite[\S III.6.1.2]{picado2012FramesLocalesTopology} the definition of a \emph{closed sublocale}. A localic relation $R$ on $X$ is called \emph{closed} if it is a closed sublocale of~$X\times X$. 
	Concretely, this means there exists $u\in \Opens X\tensor \Opens X$ so that the frame~$\Opens R$ is of the form ${\up u = \{o\in\Opens X\tensor \Opens X: u\sqleq o\}}$, and the sublocale inclusion can be represented by~${r^{-1}(o)= o\vee u}$. Important here, the corresponding frame congruence is precisely the principal frame congruence generated by the singleton relation $\{(u,\bot)\}\subseteq(\Opens X\tensor \Opens X)\times (\Opens X\tensor\Opens X)$.
\end{remark}

\begin{proposition}
	\label{proposition:R closed is fixed point}
	If $R$ is a closed localic relation with open cones, then $R\cong R_\UpDown$.
\end{proposition}
\begin{proof}
	Suppose that $R$ is the closed sublocale represented by~$u\in\Opens X\tensor\Opens X$. By~\cref{remark:closed sublocales} the corresponding frame congruence~$\langle(u,\bot)\rangle$ is the one generated by~${\{(u,\bot)\}}$. Given~\cref{proposition:R in Rupdown} it then suffices to show $\langle(u,\bot)\rangle \subseteq\langle\sim\rangle$.
	
	Since $u$ is the bottom element in $\Opens R=\up u$, we get for any rectangle $A\tensor B\sqleq u$ that $r^{-1}(A\tensor B)= u=r^{-1}(\bot)$, so by~\cref{lemma:sim of induced cones is ker r} it follows that~$A\tensor B\sim \bot$. Since $u$ is exactly the join over those rectangles $A\tensor B\sqleq u$, under the generated frame congruence $\langle\sim\rangle$ we get that~$u\langle\sim\rangle \bot$, and the result follows.
\end{proof}

\begin{remark}
	Note that \emph{open} localic relations with open cones are not necessarily fixed points, as the~\cref{counterexample:relation on Sierpinski} shows. Neither is closedness necessary in being a fixed point: take any non-strongly Hausdorff locale (\cite{isbell1972AtomlessPartsSpaces}) with the diagonal relation~(\cref{example:diagonal relation is fixed point}).
\end{remark}

A closed relation on a space does not necessarily induce a closed localic relation via the functor~$\loc\colon \rTop\to \rLoc$ in~\cref{proposition:functor rTop to rLoc}. There is a counterexample for any space~$S$ for which the canonical dense inclusion map $\pi_S$ that embeds the spatial product~$\loc(S\times S)$ into the localic product~$\loc(S)\times \loc(S)$ is not an isomorphism, for example~$S=\mathbb{Q}$. Namely, the top relation~${S\times S}$ is trivially a closed relation on~$S$, but if the induced relation on~$\loc(S)$ were closed then by density this would imply that~$\loc(S\times S)= \loc(S)\times \loc(S)$, a contradiction.

However, since the functor $\loc\colon\Top\to\Loc$ preserves closed maps, it follows straightforwardly that if $\pi_S$ is an isomorphism then any closed relation $R\subseteq S\times S$ with open cones induces a closed localic relation $\loc(R)\rightarrowtail \loc(S)\times \loc(S)$ with open cones, which is hence a fixed point by~\cref{proposition:R closed is fixed point}: $\loc(R)\cong \loc(R)_\UpDown$. 

\begin{example}
	If $S$ is a locally compact space then~$\pi_S$ an isomorphism by~\cite[Proposition~II.2.13]{johnstone1982StoneSpaces}, so any closed relation with open cones defines a fixed point.	In particular, the real line $(\mathbb{R},\leq)$ with Euclidean topology and standard ordering thus induces a closed localic relation with open cones $(\loc(\mathbb{R}),\loc(\leq))$, and is a fixed point:~$\loc(\leq)\cong \loc(\leq)_\UpDown$.
\end{example}

\begin{example}
	It is unknown if the causal relation~$\caus$ of a spacetime $M$ generally induces a fixed point. If $M$ is \emph{causally simple} then~$\caus$ is closed~\cite[Definition~4.112]{minguzzi2019LorentzianCausalityTheory}, and since $M$ is a manifold it is locally compact, so via the previous example~$\loc(\caus)$ is a fixed point. It would be interesting to see if~$\loc(\caus)$ being a fixed point corresponds to some property of~$M$ on the \emph{causal ladder}~\cite[\S 4]{minguzzi2019LorentzianCausalityTheory}.
\end{example}

\subsection{Strongly dense inclusion}
\label{section:strong density}
There is a notion of closedness more suitable constructively, introduced in~\cite{johnstone1989ConstructiveClosedSubgroup} to prove a constructive version of the closed subgroup theorem for localic groups. It is called \emph{weakly closed}, or also sometimes \emph{fibrewise closed}. Classically this notion coincides with that outlined in~\cref{remark:closed sublocales}, but not generally. Accompanying this constructive notion of closedness we have the notion of \emph{strong density}. We prove that the inclusion~$R\subseteq R_\UpDown$ from~\cref{proposition:R in Rupdown} is strongly dense. This allows us to show that all weakly closed localic relations with open cones defined fixed points of~$\cone\dashv\rel$.

In the following, denote the terminal maps in~$\Loc$ by~$!_X\colon X\to 1$. The terminal locale~$1$ has frame~$\Opens 1$, the initial frame of truth values. Classically,~$\Opens 1 \cong \{\bot,\top\}$. 

\begin{definition}
	\label{definition:strongly dense}
	A locale map~$f\colon X\to Y$ is called \emph{strongly dense} if for all~${p\in\Opens 1}$:
	\[
	f_\ast \circ \terminal_X^{-1}(p)\sqleq \terminal_Y^{-1}(p).
	\]
	A sublocale is called strongly dense if its inclusion map is strongly dense.
\end{definition}

\begin{remark}
	In classical locale theory a map of locales~$f\colon X\to Y$ is called \emph{dense} if $f_\ast(\bot)=\bot$~\cite[\S III.8]{picado2012FramesLocalesTopology}. A sublocale is dense if its inclusion is dense, which means equivalently that its corresponding sublocale set contains the bottom element. Clearly strongly dense implies dense, since taking $p\in \Opens 1$ to be the bottom element, strong density says~$f_\ast(\bot)=f_\ast \circ \terminal_X^{-1}(\bot) \sqleq \terminal_Y^{-1}(\bot)=\bot$. The converse is not generally true constructively.
\end{remark}

\begin{definition}
	A sublocale $A\in\Sl(X)$ is called \emph{weakly closed} if every strongly dense inclusion $A\subseteq B$ for $B\in\Sl(X)$ is an isomorphism.
\end{definition}

\begin{remark}
	Any closed sublocale is weakly closed, but again the converse cannot be proved constructively.
\end{remark}

We show that the sublocale inclusion~$R\subseteq R_\UpDown$ is strongly dense, from which it then follows immediately that any weakly closed open cone localic relation~$R$ is a fixed point of~$\cone\dashv\rel$.

\begin{proposition}
	\label{proposition:R in Rupdown strongly dense}
	For any open cone localic relation $R\subseteq R_\UpDown$ is strongly dense.
\end{proposition}
\begin{proof}
	Let $r\colon R\rightarrowtail X\times X$ denote the original relation, and $r_\UpDown\colon R_\UpDown\rightarrowtail X\times X$ the induced relation, so that $r_\UpDown^{-1}=\mu_\sim$ as in~\cref{section:relation induced by UpDown}. By~\cref{proposition:R in Rupdown} we get an inclusion $i\colon R\rightarrowtail R_\UpDown$ satisfying~$r=r_\UpDown\circ i$. We show that~$i$ is strongly dense, meaning for all~$p\in\Opens 1$:
	\[
	i_\ast \circ \terminal_R^{-1}(p)\sqleq \terminal_{R_\UpDown}^{-1}(p).
	\]
	Since $i^{-1}\dashv i_\ast$, we get for every~$u\in \Opens R$ that~${i_\ast(u)=\bigvee\{o\in \Opens R_\UpDown: i^{-1}(o)\sqleq u\}}$. Thus it suffices to show that $o\sqleq \terminal_{R_\UpDown}^{-1}(p)$ for every $o\in \Opens R_\UpDown$ with~${i^{-1}(o)\sqleq \terminal_R^{-1}(p)}$.
	Since $\mu_\sim$ is surjective, for each such $o\in\Opens R_\UpDown$ we can find $w\in \Opens X\tensor \Opens X$ such that $o=\mu_\sim(w)$. Then $r^{-1}(w) = i^{-1}\mu_\sim(w) = i^{-1}(o) \sqleq \terminal_R^{-1}(p)$. In particular, for every rectangle $a\tensor b\sqleq w$ we get $r^{-1}(a\tensor b)\sqleq r^{-1}(w)\sqleq \terminal_R^{-1}(p)$, and thus
	\begin{align*}
		r^{-1}(a\tensor b)
		&=
		r^{-1}(a\tensor b)\wedge \terminal_R^{-1}(p)
		\\&=
		r^{-1}(a\tensor b) \wedge r^{-1}(\terminal_X^{-1}(p)\tensor\top)
		&&(\terminal_R = \terminal_X\circ \pr_1\circ r)
		\\&=
		r^{-1}\left(a\tensor b \wedge \terminal_X^{-1}(p)\tensor \top\right)
		&&(\text{$r^{-1}$ preserves $\wedge$})
		\\&=
		r^{-1}\left((a\wedge \terminal_X^{-1}(p))\tensor b)\right).
	\end{align*}
	By~\cref{lemma:sim of induced cones is ker r} we get~$a\tensor b\sim (a\wedge \terminal_X^{-1}(p))\tensor b$, which is equalised by~$\mu_\sim$, and hence using~$\terminal_{R_\UpDown}= \terminal_X\circ \pr_1\circ r_\UpDown$ gives:
	\[
	\mu_\sim(a\tensor b)
	=
	\mu_\sim\left((a\wedge \terminal_X^{-1}(p))\tensor b\right)
	\sqleq
	\mu_\sim(\terminal_X^{-1}(p)\tensor \top) 
	=
	\terminal_{R_\UpDown}^{-1}(p).
	\]
	Taking joins over rectangles $a\tensor b\sqleq w$ gives $o=\mu_\sim(w)\sqleq \terminal_{R_\UpDown}^{-1}(p)$, as desired.
\end{proof}

\begin{corollary}
	\label{corollary:weakly closed R is fixed point}
	If $R$ is weakly closed with open cones, then $R\cong R_\UpDown$.
\end{corollary}

In particular, since we saw in~\cref{proposition:kernel pair of equivalence relation R is Rupdown} that for an equivalence relation~$R$ with open cones the induced relation~$R_\UpDown$ is precisely the kernel pair of~$R$, we recover Kock's Theorem A~\cite{kock1989GodementTheoremLocales}.

\begin{corollary}
	\label{corollary:Kock}
	Any weakly closed localic equivalence relation $R$ with open cones is the kernel pair of its quotient map $X\tworightarrow X/R$. 	
\end{corollary}
\begin{proof}
	By~\cref{proposition:kernel pair of equivalence relation R is Rupdown} the inclusion~$R\subseteq R_\UpDown$ from~\cref{proposition:R in Rupdown} reduces to the kernel pair inclusion~$R\subseteq X\times_{X/R}X$. This inclusion is strongly dense by~\cref{proposition:R in Rupdown strongly dense}, so if~$R$ is weakly closed it follows~${R\cong X\times_{X/R}X}$. 
\end{proof}

\section{Localic preorders and other properties}
\label{section:preorders}
Preorders are reflexive transitive relations. The axioms of reflexivity and transitivity can be stated in any category with finite limits. Recall the diagonal relation~${\Delta\subseteq X\times X}$ is the relation whose source and target maps are the identity~(\cref{example:diagonal relation}). In general, the composition $R\circ Q$ of two relations $R$ and $Q$ is defined as the image of a certain projection map $R\times_X Q \to X\times X$. See~\cref{remark:relation composition} for the precise construction. The definition of reflexivity and transitivity is then the straightforward internalisation of the set-theoretic conditions. For the purposes of illustration we also include some other basic properties, where we recall that if~$R$ has source and target~$(s,t)$, then~$R^\op$ is the relation defined by~$(t,s)$, which has open cones iff~$R$ has open cones. First, we recall the spatial intuition.

\begin{definition}
	A relation $R$ on a space $S$ is called:
	\begin{itemize}
		\item \emph{reflexive} if $xRx$ for all $x\in S$;
		\item \emph{transitive} if $xRy$ and $yRz$ imply $xRz$;
		\item \emph{symmetric} if $xRy$ implies $yRx$;
		\item \emph{interpolative} if $xRy$ implies there exists $z\in S$ with $xRz$ and $zRy$;
	\end{itemize}
\end{definition}

\begin{lemma}
	\label{lemma:reflexive and transitive in terms of cones}
	A relation $R$ on a space $S$ is:
	\begin{itemize}
		\item reflexive iff $\Delta\subseteq R$ iff $\up,\down$ are inflationary;
		\item transitive iff $R\circ R\subseteq R$ iff $\up,\down$ are subidempotent;
		\item symmetric iff $R= R^\op$ iff $\up = \down$;
		\item interpolative iff $R\subseteq R\circ R$ iff $\up\subseteq \up^2$ and $\down\subseteq\down^2$.
	\end{itemize}
\end{lemma}

Based on this characterisation, the localic internal analogues of these properties can be defined as follows.

\begin{definition}
	\label{definition:localic relation properties}
	A localic relation $R$ is called:
	\begin{itemize}
		\item \emph{reflexive} if $\Delta\subseteq R$;
		\item \emph{transitive} if $R\circ R \subseteq R$;
		\item \emph{symmetric} if $R= R^\op$;
		\item \emph{interpolative} if $R\subseteq R\circ R$.
	\end{itemize}
\end{definition}

The point of this section is to study the localic analogue of~\cref{lemma:reflexive and transitive in terms of cones}: that a relation is reflexive/transitive iff the cones are inflationary/subidempotent. Generally, any reflexive/transitive localic relation $R$ has inflationary/subidempotent cones $\Up,\Down$, but if the relation is not of the form $R_\ucdc$ then the converse implication may fail~(\cref{counterexample:preorder via cones}).

The main tool in translating these inclusions of relations into inclusions of cones is the following result, which is just a simple consequence of functoriality.

\begin{proposition}
	\label{proposition:relation inclusion implies cone inclusion}
	If $Q\subseteq R$ then $\Upsub{Q}\sqleq \Upsub{R}$ and $\Downsub{Q}\sqleq \Downsub{R}$.
\end{proposition}
\begin{proof}
	We have $Q\subseteq R$ iff $\id_X\colon (X,Q)\to (X,R)$ is monotone, so applying the functor $\cone$ we get a conic morphism $\id_{\Opens X}\colon (\Opens X,\Upsub{R},\Downsub{R})\to (\Opens X,\Upsub{Q},\Downsub{Q})$, which by~\cref{lemma:id ucdc map iff cones inclusion} holds iff $\Upsub{Q}\sqleq \Upsub{R}$ and $\Downsub{Q}\sqleq \Downsub{R}$.
\end{proof}

\subsection{Cones of compositions}
\label{section:cones of compositions}
In order to apply the result of~\cref{proposition:relation inclusion implies cone inclusion} to the transitivity axiom~${R\circ R\subseteq R}$, we first need to calculate the induced cones of~$R\circ R$. In this section we calculate the cones for an arbitrary binary composition~${R\circ Q}$ of open cone localic relations, and show that they are equal to~${\Upsub{Q}\circ \Upsub{R}}$ and~$\Downsub{R}\circ \Downsub{Q}$, as may be anticipated. In order to prove this, in addition to the already stated~\cref{lemma:open map pullback square}, we first collect some basic facts about open maps of locales.

\begin{lemma}
	\label{lemma:open descends along epi}
	Let $g\colon Y\to Z$ be a locale map, and $e\colon X\tworightarrow Y$ an epimorphism so that $f= g\circ e$ is open. Then $g$ is open with $g_!=f_!\circ e^{-1}$.
\end{lemma}
\begin{proof}
	Recall that $e$ is epimorphic in $\Loc$ iff $e^{-1}$ is an injective frame map~\cite[Lemma~IX.4.2]{maclane1994SheavesGeometryLogic}. It suffices to check that $g_!:= f_!\circ e^{-1}$ defines a left adjoint to~$g^{-1}$ that satisfies the Frobenius reciprocity condition. For that, take $V\in\Opens Y$ and~$W\in \Opens Z$, and calculate:
	\begin{align*}
		g_!(V)\sqleq W 
		&\iff
		f_!e^{-1}(V)\sqleq W
		\\&\iff
		e^{-1}(V) \sqleq f^{-1}(W)
		\\&\iff
		e^{-1}(V)\sqleq e^{-1}g^{-1}(W)
		&&(f=g\circ e)		
		\\&\iff
		V\sqleq g^{-1}(W),
		&&(\text{$e^{-1}$ injective})
	\end{align*}
	as desired: $g_!\dashv g^{-1}$. For reciprocity, check similarly:
	\begin{align*}
		g_!\left(V\wedge g^{-1}(W)\right)
		&:=
		f_!e^{-1}\left(V\wedge g^{-1}(W)\right)
		\\&=
		f_!\left(e^{-1}(V)\wedge e^{-1}g^{-1}(W) \right)
		&&(\text{$e^{-1}$ preserves $\wedge$})
		\\&=
		f_!\left(e^{-1}(V)\wedge f^{-1}(W)\right)
		&&(f=g\circ e)
		\\&= 
		f_!e^{-1}(V)\wedge W
		&&(\text{reciprocity $f$})
		\\&=:
		g_!(V)\wedge W.
		&&\qedhere
	\end{align*}
\end{proof}

\begin{lemma}
	\label{lemma:open maps compose}
	If $f,g$ are open then $f\circ g$ is open with $(f\circ g)_! = f_!\circ g_!$.
\end{lemma}
\begin{proof}
	This follows using that adjoints compose~\cite[\S IV.8]{maclane1998CategoriesWorkingMathematician}, and a straightforward check that Frobenius reciprocity carries over:
	\[
	f_!g_!\left(U\wedge g^{-1}f^{-1}(V)\right)
	=
	f_!\left(g_!(U)\wedge f^{-1}(V)\right)
	=
	f_!g_!(U)\wedge V.
	\qedhere
	\]
\end{proof}

\begin{remark}
	\label{remark:relation composition}
	All ingredients to calculate the cones of a composition of two open cone localic relations are now in place.
	We briefly restate the setup (for more details see~\cite[\S D.2]{schaaf2024TowardsPointFreeSpacetimes}), starting with two open cone relations~${r\colon R\rightarrowtail X\times X}$ and~${q\colon Q\rightarrowtail X\times X}$. Denote their source and target maps by~$s_R,t_R$ and~$s_Q,t_Q$, respectively. Next take the pullback~${P:= R\times_X^{t_R,s_Q} Q}$, which has pullback projections~${\pi_R\colon P\to R}$ and~$\pi_Q\colon P\to Q$ satisfying~$t_R\circ \pi_R = s_Q\circ \pi_Q$. Define the maps~$p_1:= s_R\circ \pi_R$ and~$p_2:= t_Q\circ \pi_Q$,
	and consider $p=(p_1,p_2)\colon P\to X\times X$. The composite relation~$R\circ Q$ is the image sublocale of~$p$, and we write its factorisation as
	\[
	\begin{tikzcd}[cramped]
		P \ar[r,two heads,"e"]& R\circ Q  \ar[r,tail,"m"]& X\times X.
	\end{tikzcd}
	\]
	The source and target maps of $R\circ Q$ are denoted $s=\pr_1\circ m$ and $t=\pr_2\circ m$, respectively. These are then related to the old source and target maps of $R$ and $Q$ as follows:
	\[
	s\circ e = p_1 = s_R\circ \pi_R
	\qquad\text{and}\qquad
	t\circ e = p_2 = t_Q\circ \pi_Q.
	\]
\end{remark}

\begin{proposition}
	\label{proposition:cones of RQ}
	If $R$ and $Q$ have open cones, then so does $R\circ Q$, and the induced cones are
	\[
	\Upsub{R\circ Q} = \Upsub{Q}\circ \Upsub{R}
	\qquad\text{and}\qquad
	\Downsub{R\circ Q} = \Downsub{R}\circ \Downsub{Q}.
	\]
\end{proposition}
\begin{proof}
	If $s_R,t_R,s_Q,t_Q$ are open, by~\cref{lemma:open map pullback square} the pullback projections~$\pi_R,\pi_Q$ are also open. Since open maps compose (\cref{lemma:open maps compose}), the maps $p_1=s_R\circ\pi_R$ and $p_2=t_Q\circ\pi_Q$ are open as well. By~\cref{lemma:open descends along epi} and the equalities $s\circ e=p_1$ and $t\circ e=p_2$, it follows that $s,t$ are open, and hence $R\circ Q$ has open cones. Explicitly, the left adjoints are given by $s_! = (p_1)_!\circ e^{-1}$ and~$t_! = (p_2)_!\circ e^{-1}$.
	
	We can then calculate the down cone $\Downsub{R\circ Q}$ as follows. First, for $U\in\Opens X$ note
	\[
	s_!t^{-1}(U)
	=
	(p_1)_!e^{-1}t^{-1}(U)
	=
	(p_1)_!p_2^{-1}(U),
	\]
	and unpacking $p_1=s_R\circ\pi_R$ and $p_2=t_Q\circ\pi_Q$ this gives using~\cref{lemma:open maps compose}:
	\[
	(p_1)_!p_2^{-1}(U)
	=
	(s_R\circ\pi_R)_!(t_Q\circ\pi_Q)^{-1}(U)
	=
	(s_R)_!(\pi_R)_!\pi_Q^{-1}t_Q^{-1}(U).
	\]
	On the right hand side we use~\cref{lemma:open map pullback square} to rewrite $(\pi_R)_!\circ \pi_Q^{-1} = t_R^{-1}\circ (s_Q)_!$, which gives the desired form:
	\[
	s_!t^{-1}(U)
	=
	(s_R)_! t_R^{-1} (s_Q)_! t_Q^{-1}(U)
	=
	\Downsub{R}\Downsub{Q} U,
	\]
	and so we see $\Downsub{R\circ Q}=\Downsub{R}\circ\Downsub{Q}$.
	
	The computation for the up cone is analogous, starting with
	\[
	t_!s^{-1}(U)
	=
	(p_2)_!p_1^{-1}(U)
	=
	(t_Q)_!(\pi_Q)_!\pi_R^{-1}s_R^{-1}(U)
	\]
	and applying the dual of~\cref{lemma:open map pullback square} to substitute $(\pi_Q)_!\circ\pi_R^{-1}=s_Q^{-1}\circ (t_R)_!$, we get the desired $\Upsub{R\circ Q}=\Upsub{Q}\circ\Upsub{R}$ as follows:
	\[
	t_!s^{-1}(U)
	=
	(t_Q)_! s_Q^{-1} (t_R)_! s_R^{-1}(U)
	=
	\Upsub{Q}\Upsub{R}U.
	\qedhere
	\]
\end{proof}

\subsection{Properties in terms of cones}
We are now ready to characterise the relational properties of~\cref{definition:localic relation properties} in terms of cones.

\begin{proposition}
	\label{proposition:relation property implies cone property}
	If a localic relation $R$ with open cones is:
	\begin{itemize}
		\item reflexive then $\Up,\Down$ are inflationary;
		\item transitive then $\Up,\Down$ are subidempotent;
		\item symmetric then $\Up=\Down$;
		\item interpolative then $\Up\sqleq \Up^2$ and $\Down\sqleq \Down^2$.
	\end{itemize}
\end{proposition}
\begin{proof}
	The proof is a straightforward application of~\cref{proposition:relation inclusion implies cone inclusion}: reflexivity means~$R$ contains the diagonal relation, whose cones are the identity~(\cref{example:induced cones of diagonal}), so the claim follows. If $R$ is transitive or interpolative, we get the desired inclusions after applying~\cref{proposition:cones of RQ} to see the cones of $R\circ R$ are exactly~$\Up^2$ and~$\Down^2$. Lastly, the claim about symmetry follows by observing~$\Upsub{R^\op}=\Downsub{R}$ and~$\Downsub{R^\op}=\Upsub{R}$.
\end{proof}

\begin{counterexample}
	\label{counterexample:preorder via cones}
	Take again the relation $r\colon R\rightarrowtail \mathbb{S}\times\mathbb{S}$ from~\cref{counterexample:relation on Sierpinski}, defined on the Sierpi\'{n}ski locale. We have opens~${\Opens\mathbb{S} = \{\bot<a<\top\}}$, the relation $R$ is the open sublocale defined by $u=(a\tensor \top)\vee (\top\tensor a)$, and the internal cones were calculated as $\Up \bot = \bot$, $\Up a =\top = \Up \top$, and $\Down=\Up$. Thus the cones are inflationary and subidempotent. However, we claim that $R$ is neither reflexive nor transitive.
	
	First, suppose for the sake of contradiction that $R$ is reflexive, so there exists a locale map $i\colon \mathbb{S}\to R$ such that $\delta = r\circ i$. Then we calculate
	\[
	\delta^{-1}(u) = \delta^{-1}(a\tensor \top)\vee \delta^{-1}(\top\tensor a)= a\vee a = a,
	\]
	while on the other hand $r^{-1}(u) = u\wedge u = u$ is the top element in $\Opens R$, so since~$i^{-1}$ is a frame map we get $i^{-1}(u)=\top\neq a$, a contradiction.
	
	To show that $R$ is not transitive, we claim that $R\circ R = \mathbb{S}\times\mathbb{S}$. Recall from~\cref{remark:relation composition} that $R\circ R$ is defined as the image of the map $R\times_{\mathbb{S}}^{t,s}R\to \mathbb{S}\times\mathbb{S}$ given by $p=(s\circ\pi_1,t\circ\pi_2)$, where $\pi_i$ are the pullback projections. Let $c\colon \mathbb{S}\times \mathbb{S}\to \mathbb{S}$ be the map defined by $c^{-1}(a)=\top\tensor \top$, and consider the maps
	\[
	(\pr_1,c),(c,\pr_2)\colon \mathbb{S}\times\mathbb{S}\longrightarrow \mathbb{S}\times\mathbb{S}.
	\]
	Using \cref{lemma:pair of locale maps} we get $(\pr_1,c)^{-1}(u) =\top= (c,\pr_2)^{-1}(u)$, so their images are contained in $R$, giving maps ${m,n\colon \mathbb{S}\times\mathbb{S} \to R}$ satisfying ${t\circ m = c = s\circ n}$. Hence by universality of the pullback there exists a map~$\sigma\colon \mathbb{S}\times \mathbb{S} \to R\times_{\mathbb{S}}^{t,s}R$ with $\pi_1\circ \sigma = m$ and $\pi_2\circ \sigma = n$. This gives $p\circ \sigma = \id$, and hence $\mathbb{S}\times\mathbb{S}\subseteq R\circ R$.
\end{counterexample}

To finish this section, we show that a partial converse of \cref{proposition:relation property implies cone property} holds for the fixed points of the adjunction.

\begin{proposition}
	\label{proposition:property of fixed point iff property of cones}
	For any conic frame $(\Opens X,\uc,\dc)$, the relation $R_\ucdc$ is:
	\begin{itemize}
		\item reflexive iff $\uc,\dc$ are inflationary;
		\item transitive iff $\uc,\dc$ are subidempotent;
		\item symmetric iff $\uc=\dc$.
	\end{itemize}
\end{proposition}
\begin{proof}
	We just need to prove the implications converse to those in~\cref{proposition:relation property implies cone property}. The main mechanism is universality from~\cref{corollary:Rucdc universal with cones}.
	
	First, that $\uc,\dc$ are inflationary immediately gives $R_{\id\!\id}\subseteq R_\ucdc$, and we saw in~\cref{example:diagonal relation is fixed point} that the relation induced by the identity cones is precisely the diagonal, so $R_\ucdc$ is reflexive. Next, if $\uc^2\sqleq \uc$ and $\dc^2\sqleq \dc$ then by~\cref{proposition:cones of RQ} the cones of $R_\ucdc\circ R_\ucdc$ are contained in those of $R_\ucdc$, so by~\cref{corollary:Rucdc universal with cones} we get $R_\ucdc\circ R_\ucdc\subseteq R_\ucdc$. Lastly, the claim about symmetry is clear.
\end{proof}

\begin{remark}
	\label{remark:composition of fixed point is not fixed point?}
	We do not have a proof of an analogous characterisation of $R_\ucdc$ being interpolative. The problem is that~\cref{corollary:Rucdc universal with cones} does not apply to $R_\ucdc\circ R_\ucdc$, for which it is unknown if it is still a fixed point. In general, one might expect a formula of the sort $R_\ucdc \circ R_\ucdcx = R_{\ucx\circ \uc,\dc\circ\dcx}$. From~\cref{corollary:Rucdc universal with cones,proposition:cones of RQ} we indeed get the inclusion~$\subseteq$, but whether the converse inclusion holds (or an explicit counterexample) is an open question.
\end{remark}

\begin{corollary}
	A weakly closed localic relation $R$ with open cones is reflexive/transitive iff the cones $\Up,\Down$ are inflationary/subidempotent.
\end{corollary}
\begin{proof}
	Apply \cref{corollary:weakly closed R is fixed point,proposition:property of fixed point iff property of cones}.
\end{proof}

\section{Egli-Milner ordered locales}
\label{section:Egli-Milner ordered locales}
In this standalone section we briefly outline the connection between conic frames and the previously studied notion of \emph{ordered locale} from~\cite{heunen2024OrderedLocales}. We recall only the necessary basics, and refer for more exposition to~\cite[\S 4]{schaaf2024TowardsPointFreeSpacetimes}. In the original paper it was proved that there is an adjunction between a certain subcategory of ordered locales $\LeqLoc$ and a certain subcategory of preordered topological spaces with open cones, lifting the generalised Stone duality~$\loc\dashv\pt$ between~$\Top$ and~$\Loc$. In this section we prove that the full subcategory~$\ucdcFrm^\op_\leq$ consisting of conic frames with inflationary and subidempotent cones is equivalent to a certain full subcategory of~$\LeqLoc$ used in~\cite{heunen2025CausalCoverageOrdered}.

\begin{definition}
	An \emph{$\Leq$-ordered locale} is a pair $(X,\Leq)$ where $X$ is a locale and $\Leq$ is a preorder on $\Opens X$ such that
	\[\tag{$\vee$}\label{vee}
	\forall i: U_i\Leq V_i
	\quad\implies\quad
	\bigvee U_i \Leq \bigvee V_i.
	\]
\end{definition}

\begin{remark}
	The axiom~\eqref{vee} just says that ${\Leq}\subseteq\Opens X\times \Opens X$ is a sub-suplattice.
\end{remark}

\begin{example}
	\label{example:Egli-Milner order}
	The $\Leq$-ordered locales can be seen as a localic axiomatisation of the Egli-Milner order on a space~\cite[\S 11.1]{vickers1989TopologyLogic}. If $(S,\leq)$ is any preordered space, then $\loc(S)$ becomes an $\Leq$-ordered locale by equipping it with the \emph{Egli-Milner order}:
	\[
	U\Leq V 
	\quad\iff\quad
	U\subseteq \down V\text{ and } V\subseteq \up U.
	\]
\end{example}

\begin{example}
	\label{example:Egli-Milner from monads}
	In fact, we can define generalised Egli-Milner orders. Recall that a \emph{monad} (or \emph{closure operator}) on a poset is a monotone, inflationary and subidempotent function~(\cref{definition:closure operator}). For any pair $\uc,\dc$ of monads on $\Opens X$, we can define an $\Leq$-ordered locale~$(X,\Leq_\ucdc)$ via
	\[
	U\Leq_\ucdc V
	\quad\iff\quad
	U\sqleq \dc V
	\text{ and }
	V\sqleq \uc U.
	\]
	This is a preorder iff $\uc,\dc$ are inflationary and idempotent, and it satisfies~\eqref{vee} just because $\uc,\dc$ are monotone.
	We call $\Leq_\ucdc$ the \emph{Egli-Milner order} induced by~$\uc,\dc$.
\end{example}

\begin{example}
	For any locale $X$ the inclusion order $\sqleq$ of its frame defines the structure of an $\Leq$-ordered locale $(X,\sqleq)$. It can be realised as the Egli-Milner order of $\dc=\id_{\Opens X}$ and where $\uc$ is constantly $\top$. This serves as a localic variant of the specialisation order on a topological space, see~\cite[Example~5.54]{schaaf2024TowardsPointFreeSpacetimes}. Note these cones are parallel, but~$\uc$ is not join-preserving.
\end{example}

Thus we see that any pair of monads $\uc,\dc$ induces a $\Leq$-ordered locale. Conversely, the following construction shows how to extract such monads from the preorder~$\Leq$.

\begin{definition}
	If $(X,\Leq)$ is an $\Leq$-ordered locale, its \emph{cones}~${\Up,\Down\colon \Opens X\to \Opens X}$ are
	\[
	\Up U : = \bigvee\{V\in \Opens X: U\Leq V\}
	\quad\text{and}\quad
	\Down U := \bigvee\{W\in \Opens X: W\Leq U\}.
	\]
	We say $(X,\Leq)$ has \emph{join-preserving cones} if~$\Up,\Down$ preserve joins:
	\[\tag{$\Up\mspace{-3mu}\Down$-$\vee$}\label{cone vee}
	\Up\bigvee U_i = \bigvee \Up U_i
	\quad\text{and}\quad
	\Down\bigvee U_i = \bigvee\Down U_i.
	\]
\end{definition}

\begin{lemma}
	\label{lemma:properties of Leq-cones}
	In any $\Leq$-ordered locale $(X,\Leq)$ we have that:
	\begin{itemize}
		\item if $U\Leq V$ then $U\sqleq \Down V$ and $V\sqleq \Up U$;
		
		\item if $U\sqleq V$ then $\Up U \sqleq \Up V$ and $\Down U \sqleq \Down V$;
		
		\item $U\Leq \Up U$ and $\Down U \Leq U$;
		
		\item $U\sqleq \Up U$ and $U\sqleq \Down U$;
		
		\item $\Up\Up U \sqleq \Up U$ and $\Down \Down U \sqleq \Down U$.
	\end{itemize}
\end{lemma}
\begin{proof}
	This is~\cite[Lemma~3.8]{heunen2024OrderedLocales}.
\end{proof}

\begin{example}
	If $(S,\leq)$ is a preordered space, then its induced $\Leq$-ordered locale from~\cref{example:Egli-Milner order} has the cones
	\[
	\Up U = (\up U)^\circ
	\quad\text{and}\quad
	\Down U = (\down U)^\circ.
	\]
	From this we can see that $\Up,\Down$ will not generally preserve joins, essentially due to the interior operator not preserving (even finite) unions. However, in the case that~$(S,\leq)$ has open cones (in the same sense as~\cref{definition:open cones on space}), then the induced~$\Up,\Down$ will preserve arbitrary joins.
\end{example}

\begin{example}
	\label{example:Leq-cones from monads}
	If $\uc,\dc$ is any pair of monads on $\Opens X$ then the corresponding Egli-Milner order~$(X,\Leq_\ucdc)$ has cones~${\Up = \uc}$ and~$\Down = \dc$. This is easy to see. First observe that $U\Leq_\ucdc \uc U$, since $U\sqleq \dc\uc U$ and~$\uc U \sqleq \uc U$. Thus~$\uc U$ appears in the join defining $\Up U$, and hence $\uc U\sqleq \Up U$. On the other hand, for any~$V$ such that~$U\Leq_\ucdc V$ we get by construction that $V\sqleq \uc U$, and since by~\cref{lemma:properties of Leq-cones} it holds~$U\Leq_\ucdc \Up U$, we get~$\Up U \sqleq \uc U$, and hence equality.
\end{example}

Using these cones, we can say what it means for a locale map between \mbox{$\Leq$-ordered} locales to be monotone, analogous to the definition of conic morphisms.

\begin{definition}
	A locale map $f\colon (X,\Leq_X)\to (Y,\Leq_Y)$ between $\Leq$-ordered locales is called \emph{monotone} if
	\[
	\Upsub{X} \circ f^{-1}\sqleq f^{-1}\circ \Upsub{Y}
	\quad\text{and}\quad
	\Downsub{X} \circ f^{-1}\sqleq f^{-1}\circ \Downsub{Y}.
	\]
\end{definition}

\begin{definition}
	Let $\LeqLoc$ denote the category of $\Leq$-ordered locales and monotone locale maps between them.
\end{definition}

\begin{remark}
	It is possible to state the definition of a monotone map between \mbox{$\Leq$-ordered} locales purely in terms of the preorders $\Leq$, but this is not needed here, and in fact is equivalent to the stated definition in terms of cones. See~\cite[\S 4.3]{schaaf2024TowardsPointFreeSpacetimes} for more details.
	
	However, the fact that monotonicity is fully determined by the cones has the following consequence. We say that $(X,\Leq)$ has \emph{order determined by cones} if $\Leq$ is recovered as the Egli-Milner order~$\Leq_\UpDown$ of its own cones~$\Up,\Down$. It is unknown if there exist $\Leq$-ordered locales that are not determined by cones, however from this characterisation of monotonicity and the fact that~$\Leq$ and~$\Leq_\UpDown$ share the same cones, it follows that the identity locale map defines an isomorphism ${\id_X\colon (X,\Leq)\cong (X,\Leq_\UpDown)}$ in~$\LeqLoc$. Thus, since $\Leq_\UpDown$ is itself determined by cones, the category~$\LeqLoc$ is in fact equivalent to its full subcategory of~$\Leq$-ordered locales determined by cones. See also the remark after~\cite[Proposition~4.34]{schaaf2024TowardsPointFreeSpacetimes}.
\end{remark}

Just as join-preservation of $\Up,\Down$ is not automatic, neither is parallelness. To fix this, we introduce the following condition.

\begin{definition}
	We say an $\Leq$-ordered locale \emph{respects meets} if we have:
		\[
		\begin{tikzcd}[every label/.append style = {font = \normalsize},column sep=0.25cm, row sep=0.2cm]
			{U} & {\exists U'} \\
			{V} & {V'}
			\arrow["\sqleq"{anchor=center, rotate=-90}, draw=none, from=2-1, to=1-1]
			\arrow["\sqleq"{anchor=center, rotate=-90}, draw=none, from=2-2, to=1-2]
			\arrow["{}"{description}, "\Leq"{anchor=center}, draw=none, from=1-1, to=1-2]
			\arrow["{}"{description}, "\Leq"{anchor=center}, draw=none, from=2-1, to=2-2]
		\end{tikzcd}
		\qquad\text{and}\qquad
		\begin{tikzcd}[every label/.append style = {font = \normalsize},column sep=0.25cm, row sep=0.2cm]
			{\exists U} & {U'} \\
			{V} & {V'.}
			\arrow["\sqleq"{anchor=center, rotate=-90}, draw=none, from=2-1, to=1-1]
			\arrow["\sqleq"{anchor=center, rotate=-90}, draw=none, from=2-2, to=1-2]
			\arrow["{}"{description}, "\Leq"{anchor=center}, draw=none, from=1-1, to=1-2]
			\arrow["{}"{description}, "\Leq"{anchor=center}, draw=none, from=2-1, to=2-2]
		\end{tikzcd}
		\]
\end{definition}

\begin{remark}
	The first diagram is read as: if $U\sqleq V\Leq V'$ then there exists $U'$ such that $U\Leq U'\sqleq V'$. The second diagram is parsed similarly. Note that this is just the notion of `bisimulation', see for example~\cite[\S 2.2]{blackburn2001ModalLogic}.
\end{remark}

\begin{lemma}
	\label{lemma:respects meets iff parallel}
	An $\Leq$-ordered locale respects meets iff order is determined by cones and~$\Up,\Down$ are parallel.
\end{lemma}
\begin{proof}
	This is~\cite[Lemma~4.45]{schaaf2024TowardsPointFreeSpacetimes}.
\end{proof}

We are now ready to prove the claimed isomorphism of categories. First we introduce their notation and construct functors between them, from which the isomorphism follows straightforwardly.

\begin{definition}
	Let $\pLeqLoc$ denote the full subcategory of $\LeqLoc$ containing $\Leq$-ordered locales respecting meets and satisfying~\eqref{cone vee}.
\end{definition}

\begin{remark}
	The objects in this category are used in~\cite{heunen2025CausalCoverageOrdered} to study a notion of~`causal coverage'. Note also that any $\Leq$-ordered locale coming from an ordered space with open cones lives in this category.
\end{remark}

\begin{definition}
	Let $\pucdcFrm$ denote the full subcategory of $\ucdcFrm$ containing conic frames $(L,\uc,\dc)$ so that $\uc,\dc$ are inflationary and subidempotent.
\end{definition}

\begin{proposition}
	There is a functor
	\begin{align*}
		\mathsf{EM}\colon \pucdcFrm^\op &\longrightarrow \pLeqLoc
		\\
		(\Opens X,\uc,\dc)&\longmapsto (X,\Leq_\ucdc);
		\\
		f^{-1} &\longmapsto f.
	\end{align*}
\end{proposition}
\begin{proof}
	That the Egli-Milner order $\Leq_\ucdc$ defines an~$\Leq$-ordered locale follows by~\cref{example:Egli-Milner from monads}, that it respects meets follows by~\cref{lemma:respects meets iff parallel}, and that its cones preserve joins follow since $\uc,\dc$ are join-preserving. That conic morphisms induce monotone maps of \mbox{$\Leq$-ordered} locales is true by definition.
\end{proof}

\begin{proposition}
	There is a functor
	\begin{align*}
		\mathsf{Cone}\colon \pLeqLoc &\longrightarrow \pucdcFrm^\op
		\\
		(X,\Leq)&\longmapsto (\Opens X,\Up,\Down);
		\\
		f &\longmapsto f^{-1}.
	\end{align*}
\end{proposition}
\begin{proof}
	By~\cref{lemma:properties of Leq-cones} we see $\Up,\Down$ are closure operators. That $\Up,\Down$ are join-preserving follows from~\eqref{cone vee}, and that they are parallel follows by~\cref{lemma:respects meets iff parallel}. Again, that this assignment is well-defined on morphisms holds by definition.
\end{proof}

\begin{theorem}
	There is an isomorphism of categories $\pucdcFrm^\op \cong \pLeqLoc$.
\end{theorem}
\begin{proof}
	It suffices to check $\mathsf{Cone}$ and $\mathsf{EM}$ are mutually inverse on objects. That
	\[
	(\Opens X,\uc,\dc)
	\longmapsto 
	(X, \Leq_\ucdc)
	\longmapsto
	(\Opens X, \Up,\Down)
	=
	(\Opens X,\uc,\dc)
	\]
	just follows since $\Leq_\ucdc$ has cones $\Up=\uc$ and $\Down = \dc$~(\cref{example:Leq-cones from monads}), and that
	\[
	(X,\Leq)
	\longmapsto
	(\Opens X, \Up,\Down)
	\longmapsto
	(X,\Leq_\UpDown)
	=
	(X,\Leq)
	\]
	follows since~$\Leq$ respects meets, so by~\cref{lemma:respects meets iff parallel} it is determined by cones.
\end{proof}

\begin{remark}
	\label{remark:adjunction with spaces}
	As briefly remarked, the main result of~\cite{heunen2024OrderedLocales} is an adjunction between certain full subcategories of $\LeqLoc$ and $\rocTop$. However, these subcategories are unsatisfactory from a point-free perspective, since they are defined in terms of points. The isomorphism $\pucdcFrm^\op \cong \pLeqLoc$ suggests that obtaining an adjunction between ordered spaces and conic frames will likewise be unsuccessful.
	
	The main obstacle in the assignment from $\LeqLoc$ to $\rocTop$ is that the preorder~$\leq$ derived from~$\Leq$ needs to have open cones. From a structural view, one might attempt a construction via the route $\ucdcFrm^\op\xrightarrow{\rel}\rocLoc\xrightarrow{\pt}\rocTop$, but since the functor~$\pt$ does not preserve open maps, it is not guaranteed that~$\pt(R_\ucdc)$ is a relation with open cones.
	
	While the adjunction~$\loc\dashv\pt$ between $\Top$ and $\Loc$ lifts straightforwardly to an adjunction between~$\rTop$ and~$\rLoc$ (using the functor in~\cref{proposition:functor rTop to rLoc}), the same cannot be said for the categories~$\oTop$ and~$\oLoc$ of \emph{preordered} spaces and locales. The bottleneck is the fact that the left adjoint $\loc$ does not preserve the transitivity axiom, as demonstrated by a counterexample in~\cite[Remark~C5.3.5]{johnstone2002Elephant2}. This seems to suggest that, while localic relations are point-free duals to relations on spaces, localic preorders are not the point-free duals of ordered spaces.
\end{remark}

\section{Discussion}
\label{section:discussion}
We finish the paper with some general comments, discussion, and directions for future research.

\subsection{Non-open cones.}
A natural next step would be to attempt a similar reconstruction for arbitrary localic relations $R$ on $X$ in terms of their generalised internal cones ${\up,\down\colon \Sl(X)\to \Sl(X)}$. At face value, these cones will capture more of the structure of the initial relation. It is yet unknown if there is a counterexample similar to~\ref{counterexample:relation on Sierpinski}, where distinct localic relations produce the same internal cones. Moreover, since the proofs in the present work rely deeply on openness and the frame theoretic techniques of generating sublocales, the construction would likely look quite different.

\subsection{Internal cones}
\label{section:internal cones}
Continuing that thought, as briefly mentioned at the start of~\cref{section:localic relations}, since epimorphisms in $\Loc$ are not preserved under pullback~\cite{pitts1983AmalgamationInterpolation,banaschewski1990PushingOutFrames}, the general theory of relations with respect to orthogonal factorisation systems~\cite{klein1970RelationsCategories} shows that composition of localic relations is not associative. Picking a non-pullback-stable epimorphism $e$, an explicit non-associative triple of relations can be constructed from its graph. As a consequence, there exists a transitive relation $R$ whose internal cones~$\up,\down$ are not subidempotent (cf.~\cref{counterexample:preorder via cones}). In fact, this counterexample works generally for any category with an orthogonal factorisation system~$(\calE,\calM)$ where~$\calE$ is not pullback stable.

An explicit example in the spatial world may be illustrative. Equip $\Top$ with the orthogonal factorisation system $(\text{dense image},\text{closed embedding})$~\cite[\S 2]{klein1970RelationsCategories}, which is not pullback stable. Subobjects are just closed subsets, and so internal relations $R$ on a space $S$ are just closed subsets $R\subseteq S\times S$. It can be shown that internal transitivity corresponds to set-theoretic transitivity. On the other hand, if~$\up,\down$ denote the set-theoretic cones of $R$, then the internal cones are obtained via their closure: $\overline{\up A}$ and $\overline{\down A}$, for any closed subset~$A\subseteq S$. Thus subidempotence of the internal cones comes down to the inclusions $\up \overline{\up A} \subseteq \overline{\up A}$ and $\down \overline{\down A} \subseteq \overline{\down A}$. This fails for example on the space
\[
S = \underbrace{\overbrace{\{1,2,3,\ldots\}}^{\mathbb{N}} \cup\{\infty\}}_{\text{discrete}}
\cup
\underbrace{\left\{
	\frac{1}{n}:n\in\mathbb{N}
	\right\}\cup \{0\}}_{\text{subspace of $\mathbb{R}$}}
\]
with the transitive closed relation
\[
R:= \Delta_S \cup \left\{\left(n,\frac{1}{n}\right):n\in\mathbb{N} \right\} \cup \left\{(0,\infty) \right\},
\]
where we get for $A=\mathbb{N}$ that $\up \mathbb{N} = \mathbb{N}\cup \left\{\frac{1}{n}:n\in\mathbb{N}\right\}$, so $\overline{\up \mathbb{N}}=\up \mathbb{N}\cup \{0\}$, and finally $\up \overline{\up \mathbb{N}} = \overline{\up \mathbb{N}}\cup \{\infty\} \not\subseteq \overline{\up \mathbb{N}}$. This is analogous to Example~2 in~\cite[\S 22]{munkres2000Topology}, showing not all product projection maps are closed.

Moreover, mirroring properties of the topological closure, these internal cones will preserve finite joins, but not generally infinite ones. Take for example the family of closed subsets $A_n = \{1/n\}$.  The internal join $\bigvee_{n\in\mathbb{N}} A_n$ is just the closure of their set-theoretic union $\overline{\bigcup_{n\in\mathbb{N}} A_n}= \{1/n:n\in\mathbb{N}\}\cup\{0\}$, so
\[
\overline{\up \bigvee_{n\in\mathbb{N}} A_n}
=
\overline{
\up \left\{\frac{1}{n}:n\in\mathbb{N}\right\} \cup \up \{0\}
}
=
\left\{\frac{1}{n}:n\in\mathbb{N}\right\} \cup  \{0,\infty\}.
\]
On the other hand $\overline{\up A_n} = A_n$, and hence
\[
\bigvee_{n\in\mathbb{N}}\overline{\up A_n}
=
\overline{\bigcup_{n\in\mathbb{N}} A_n}
=
\overline{
\left\{\frac{1}{n}:n\in\mathbb{N}\right\}
}
=
\left\{\frac{1}{n}:n\in\mathbb{N}\right\}\cup\{0\},
\]
missing the point~$\infty$.

Lastly, the internal cones are not parallel. Take for instance the closed subsets~$A=\mathbb{N}$ and~$B=\{0\}$. Then $\overline{\up A}\cap B= \{0\}$, but $A\cap \overline{\down B}= A\cap B$, which in the present convention is empty. Thus parallelness fails:
\[
\{0\}=\overline{\up A}\cap B \not\subseteq \overline{\up\left( A\cap \overline{\down B}\right)} = \overline{\up \varnothing} = \varnothing.
\]
This shows that in a more general theory of internal relations the usual set-theoretic intuition of cones breaks down. It also suggests that parallelness of the localic cones $\Up,\Down$ is more so an artifact of Frobenius reciprocity, and hence a consequence of the open cone assumption, than a fundamental property of cones that can be used to do build a theory of abstract order theory based on cones.

\subsection{One-sided openness.}
In some important order-theoretic examples, such as the specialisation order of a topological space, the relation has open cones in only one direction. Thus we might want a more general theory of one-sided conic frames~$(L,\dc)$ corresponding to localic relation whose source map is open, but whose target map may not be. Observe that for any~$\dc$ the pair $(\top,\dc)$ is actually parallel. Though, of course, the constant map~$\top$ does not preserve joins. In particular it does not preserve the bottom element, which is needed to make the generating relation~$\sim$ well defined. Nevertheless, motivated by~\cref{lemma:parallel in terms of sim} we can replace~$\sim$ simply by
\[
x\tensor y \sim
\Dc(x,y)\tensor \Uc(x,y),
\]
which works even if~$\uc,\dc$ do not preserve the bottom. In particular for~$\uc=\top$ this gives~$x\tensor y\sim \Dc(x,y)\tensor y$. We obtain a localic relation~$R_{\dc}$ in the same way as~\cref{definition:induced relation}, and the proofs in~\cref{section:separation} allow us to construct an open source map. The resulting adjunction with localic relation with open source maps is explained further in~\cite[\S 4]{schaaf2026LocalicEsakiaDuality}.

\subsection{Non-parallelness.}
The altered version of~$\sim$ of the previous section also raises the question of whether the present adjunction~$\cone\dashv\rel$ could be generalised by dropping the parallelness assumption on~$\uc,\dc$. Note that by~\cref{lemma:parallel in terms of sim} the relation~$x\tensor y\sim \Dc(x,y)\tensor \Uc(x,y)$ agrees with the original definition of~$\sim$ whenever~$\uc,\dc$ are parallel. This altered version also has the additional benefit that it can be defined without assuming~$\uc,\dc$ preserve the bottom element. If we then still obtain an adjunction of the form $\cone\dashv\rel$, the counit will no longer be the identity, but instead would define a comonad that takes a non-parallel conic frame and produces its `parallel interior'.

\subsection{Relations between different locales}
\label{section:relations between different locales}
In this work we have only considered localic relations~$R$ defined on a single locale~$X$. More generally, we might want to consider localic relations between locales~$X$ and~$Y$. For such~$R\rightarrowtail X\times Y$ with open source and target maps, we get cones
\[
\Up:=t_!s^{-1}\colon \Opens X \longrightarrow \Opens Y
\quad\text{and}\quad
\Down := s_!t^{-1}\colon \Opens Y\longrightarrow \Opens X.
\]
This suggests the more general definition of conic frame being a tuple $(L,M,\uc,\dc)$ where $L,M$ are frames, and $\uc\colon L\to M$ and $\dc\colon M\to L$ are join-preserving maps that are parallel. Note indeed that parallelness still makes sense, for~$x\in L$ and~$y\in M$:
\[
\uc x\wedge y \sqleq \uc (x\wedge \dc y)
\qquad\text{and}\qquad
x \wedge \dc y \sqleq \dc (\uc x\wedge y),
\]
where the first inclusion takes place in~$M$ and the second in~$L$. Clearly any relation~${R\rightarrowtail X\times Y}$ with open cones then induces such a conic frame~$(\Opens X,\Opens Y,\Up,\Down)$. Conversely, from $(L,M,\uc,\dc)$ we define the generating relation~$\sim$ on rectangles of the coproduct frame~$L\tensor M$ by~$x\tensor y\sim \Dc(x,y)\tensor \Uc(x,y)$, where we note that now that the reduced cones have types~${\Uc\colon L\times M \to M}$ and~${\Dc\colon L\times M\to L}$. It seems that most proofs above will go through for this more general setting. The problem in~\cref{remark:composition of fixed point is not fixed point?} of whether fixed points $R_\ucdc$ are closed under relational composition persists here, so it is not clear if the conic frame perspective would be of any help in studying, for example, the general (non-associative) calculus of localic relations.

\subsection{Topological over frames?}
Like the category of relations and monotone functions is topological over $\Set$ (\cite[\S 21]{adamek1990AbstractConcreteCategories}), we might ask if~$\ucdcFrm$ is topological over~$\Frm$. We show this is not the case. Denote by $U\colon \ucdcFrm\to\Frm$ the forgetful functor, sending $(L,\uc,\dc)\mapsto L$.

Consider the frames~$\Opens \mathbb{R}$ and~$\Powerset\mathbb{R}$ and the subframe inclusion $i\colon \Opens \mathbb{R}\hookrightarrow \Powerset\mathbb{R}$. The powerset has a conic frame structure $(\Powerset\mathbb{R},\{0\}\cap -,\{0\}\cap-)$ (recall~\cref{example:conic frames basic examples}). If~$U$ were a topological functor, there would be an initial conic frame structure~$\uc,\dc$ on~$\Opens \mathbb{R}$ that makes~$i$ a conic morphism. In particular, that would entail
\[
\{0\}=\{0\}\cap i(\mathbb{R}) \subseteq i(\uc\mathbb{R}) = \uc\mathbb{R},
\]
and similarly $\{0\}\subseteq \dc\mathbb{R}$. Pick a strict open neighbourhood ${0\in W\subsetneq \uc\mathbb{R}\cap \dc\mathbb{R}}$, for which we get a conic frame~$(\Opens\mathbb{R},W\cap -,W\cap -)$ and making the frame inclusion~$i$ a conic morphism:
\[
i\colon (\Opens\mathbb{R},W\cap - ,W\cap -)\longrightarrow (\Powerset\mathbb{R},\{0\}\cap-,\{0\}\cap-).
\]
But since the conic frame structure~$\uc,\dc$ is initial, this would make the identity map
\[
\id_{\Opens\mathbb{R}}\colon (\Opens\mathbb{R},W\cap -,W\cap -)\longrightarrow (\Opens \mathbb{R},\uc,\dc)
\]
conic, which by~\cref{lemma:id ucdc map iff cones inclusion} implies $\uc \mathbb{R}\cap \dc\mathbb{R}\subseteq W$, contradicting the choice of~$W$.

Nevertheless, assuming classical logic, there is a ${\text{discrete}\dashv\text{forget}\dashv\text{codiscrete}}$ adjunction. Concretely, define a functor~$D\colon \Frm\to\ucdcFrm$ by equipping each frame~$L$ with the conic structure~$(L,\tau,\tau)$, where $\tau\bot = \bot$ and $\tau x = \top$ for any~${x\neq \bot}$. On the other hand, define~$C\colon \Frm\to \ucdcFrm$ by equipping a frame~$L$ with the conic structure~$(L,\bot,\bot)$~(\cref{example:conic frames basic examples}). This gives $D\dashv U\dashv C$. Classicality is used to define $\tau$, which of course represents the top relation~$X\times X$ on a locale~$X$ (recall~\cref{example:open sublocale in overt X has open cones}).

\subsection{Connections to logic?}
In~\cref{remark:parallelness in the literature} we noted that the parallelness condition on~$\uc,\dc$ is related to the theory of Boolean algebras with operators~\cite{jonsson1951BooleanAlgebrasOperators}, and connects to modal logic~\cite[\S 3.3]{goldblatt2003MathematicalModalLogic}. The connections between localic relations and conic frames to modal logic could be investigated. Generally, an object~$(X,R)\in\rLoc$ might be interpreted as a point-free Kripke frame, and if~$R$ has open cones then the induced cones~$\Up,\Down$ could be interpreted as future and past diamond operators, with their right adjoints as the corresponding box operators.

Analogous to the remarks in~\cite[\S 9.5]{schaaf2024TowardsPointFreeSpacetimes}, note that for a conic frame $(L,\uc,\dc)$ the reduced cone~$x\wedge \dc(-)$ preserves arbitrary joins, and so admits a right adjoint~${x\wedge \dc(-) \dashv x\to_{\dc}-}$, characterised by
\[
x\wedge \dc y\sqleq z
\iff
\dc y \sqleq x\to z
\iff
y\sqleq x\to_{\dc}z,
\]
where~$\to$ is just the usual Heyting implication of~$L$. This is analogous to the `causal implications' from~\cite{akbartabatabai2021ImplicationSpacetime}. Note for example that the separating opens from~\cref{definition:separating opens} can then be expressed as
\[
O^\Uc_z = \bigvee_{x\in L} (x\to_{\dc} z)\tensor x
\quad\text{and}\quad
O^\Dc_z = \bigvee_{x\in L} x\tensor (x\to_{\dc}z).
\]

As mentioned, a one-sided notion of conic frame~$(L,\dc)$ is used in~\cite{schaaf2026LocalicEsakiaDuality} to describe a constructive, localic Esakia duality for Heyting algebras, based on the localic Priestley duality of Townsend~\cite{townsend1997LocalicPriestleyDuality} and the recently introduced Heyting frames~\cite{bezhanishvili2023FrametheoreticPerspectiveEsakia}.

Lastly, connections of conic frames to biframes~\cite{banaschewski1983BiframesBispaces}, d-frames~\cite{jung2010BitopologicalNatureStone} and the recently introduced ad-frames~\cite{goubault2026StoneDualityPreordered} could be investigated.

\subsection*{Acknowledgments} This work has been partially funded by the French National Research Agency (ANR) within the framework of ``Plan France 2030'', under the research projects EPIQ ANR-22-PETQ-0007, HQI-Acquisition ANR-22-PNCQ-0001 and HQI-R\&D ANR-22-PNCQ-0002.

\subsection*{AI disclosure}
AI tools were used in this work for checking proofs, literature search, proofreading, and typesetting diagrams. All final text and proofs were written manually by the author, who has verified the mathematical claims and references, and takes full responsibility for the contents of the paper.

\printbibliography
\end{document}